\documentclass[a4paper]{article}

\usepackage{bm}
\usepackage{appendix}
\usepackage[all]{xy}\usepackage[latin1]{inputenc}        
\usepackage[dvips]{graphics,graphicx}
\usepackage{amsfonts,amssymb,amsmath,xcolor,mathrsfs, amstext}
\usepackage{amsbsy, amsopn, amscd, amsxtra, amsthm,authblk, enumerate}
\usepackage{upref}
\usepackage{geometry}
\usepackage[displaymath]{lineno}
\usepackage{float}
\usepackage{bbm}
\usepackage{stackrel}
\usepackage{subfigure}

\usepackage[colorlinks,
            linkcolor=blue,
            anchorcolor=green,
            citecolor=blue
            ]{hyperref}

\numberwithin{equation}{section}

\def\epsilon{\varepsilon}
\def\eps{\varepsilon}

\newcommand{\ol}{\overline}
\newcommand{\wt}{\widetilde}
\def\alb#1\ale{\begin{align*}#1\end{align*}}
\newcommand{\eqb}{\begin{equation}}
\newcommand{\eqe}{\end{equation}}

\DeclareMathOperator{\diam}{diam}

\newcommand{\bbC}{\mathbb{C}}
\newcommand{\bbD}{\mathbb{D}}
\newcommand{\bbE}{\mathbb{E}}
\newcommand{\bbH}{\mathbb{H}}
\newcommand{\bbN}{\mathbb{N}}

\newcommand{\bbQ}{\mathbb{Q}}
\newcommand{\bbR}{\mathbb{R}}
\newcommand{\bbP}{\mathbb{P}}
\newcommand{\bbT}{\mathbb{T}}
\newcommand{\bbZ}{\mathbb{Z}}

\newcommand{\cA}{\mathcal{A}}
\newcommand{\cB}{\mathcal{B}}
\newcommand{\cC}{\mathcal{C}}

\newcommand{\cF}{\mathcal{F}}

\newcommand{\cN}{\mathcal{N}}

\newcommand{\din}{\partial^{\textrm{in}}}
\newcommand{\dout}{\partial^{\textrm{out}}}
\newcommand{\de}{\partial^{\textrm{e}}}
\newcommand{\Fill}{\mathrm{Fill}}
\newcommand{\ep}[1]{\overline{\overline{#1}}}
\newcommand{\one}{\mathbf{1}}
\newcommand{\wh}{\widehat}

\newcommand{\LF}{\mathrm{LF}}
\newcommand{\SLE}{\mathrm{SLE}}

\newcommand{\epmu}{\overline{\mu}_\kappa}
\newcommand{\epnu}{\overline{\nu}_\kappa}
\newcommand{\epmumu}{\overline{\widehat{\mu}}_\kappa}
\newcommand{\mumu}{\widehat{\mu}_\kappa}

\newtheorem{theorem}{Theorem}[section]
\newtheorem{lemma}[theorem]{Lemma}
\newtheorem{proposition}[theorem]{Proposition}
\newtheorem*{proposition*}{Proposition}

\newtheorem*{corollary*}{Corollary}
\newtheorem{definition}[theorem]{Definition}
\newtheorem*{definitions*}{Definitions}

\newtheorem*{example*}{\bf Example}
\theoremstyle{remark}
\newtheorem{remark}{\bf Remark}[section]

\newtheorem{descrip}{\bf Description}[section]

\numberwithin{equation}{section}

\title{Backbone exponent for two-dimensional percolation}
\author{Pierre Nolin\thanks{City University of Hong Kong} \qquad Wei Qian\thanks{City University of Hong Kong. On leave from CNRS, Laboratoire de Math\'ematiques d'Orsay, Universit\'e Paris-Saclay.} \qquad Xin Sun\thanks{Beijing International Center for Mathematical Research, Peking University.} \qquad Zijie Zhuang\thanks{University of Pennsylvania}}
\date{}

\begin{document}

\maketitle 

\begin{abstract}
We derive an exact expression for the celebrated backbone exponent for Bernoulli percolation in dimension two at criticality. It turns out to be a root of an elementary function. Contrary to previously known arm exponents for this model, which are all rational, it has a transcendental value. Our derivation relies on the connection to the SLE$_\kappa$ bubble measure, the coupling between SLE and Liouville quantum gravity, and the integrability of Liouville conformal field theory.  Along the way,  we derive a formula not only for $\kappa=6$ (corresponding to percolation), but for all $\kappa \in (4,8)$.
\end{abstract}

\section{Introduction}
Bernoulli percolation is a simple and very natural process of statistical mechanics, defined on a lattice. It was introduced by Broadbent and Hammersley \cite{BH57} to model the large-scale properties of a random material.
Two-dimensional (2D) percolation is especially well understood, thanks to its connection to  conformal invariance and Schramm-Loewner evolution (SLE). 
Introduced by Schramm in the groundbreaking work~\cite{Sc00}, SLE is a one-parameter family of random non-self-crossing curves characterized by conformal invariance and domain Markov property, denoted by $\SLE_\kappa$. It is conjectured to describe the scaling limits of a large class of two-dimensional random systems at criticality.
In another breakthrough \cite{Sm01} published shortly after, Smirnov proved that critical site percolation on the triangular lattice converges to a conformally invariant scaling limit, which can thus be described by $\SLE_6$.

The arm exponents are a set of scaling exponents encoding important geometric information of percolation at and near its criticality. They  describe the probability of observing connections across annuli of large modulus by disjoint connected paths of specified colors.
When there is at least one arm of each of the two colors, such exponents are called polychromatic
arm exponents. Otherwise, they are called monochromatic arm exponents. Based on the link with   SLE, the exact values of the one-arm exponent and all the polychromatic arm exponents were derived rigorously for site percolation on the triangulation lattice in \cite{LSW02} and \cite{SW01}, respectively. The value of these exponents were also predicted in the physics literature; see~\cite{ADA99} and references therein.

Despite the SLE connection, the evaluation of monochromatic arm exponents beyond the one-arm case has been a longstanding mystery. In this paper, we derive the exact value for the monochromatic two-arm exponent, namely with two disjoint connections of the same color; see its formal definition in Section~\ref{subsec:def}. This exponent is also known as the backbone exponent, with a rich history that we will review in Section~\ref{subsec:def}. Prior to our work, there is no theoretical prediction for the backbone exponent in the literature that is consistent with numerical approximations. We show that it is the root of a simple elementary function, see \eqref{eq:sol-xi} below. Moreover, it is transcendental with the first few digits given by $0.35666683671288$.

Our derivation consists of two steps. First, we express the backbone exponent in terms of a variant of SLE$_6$ call the SLE$_6$  bubble measure, see Section~\ref{subsec:intro-sle}. Then we exactly solve the SLE$_6$ problem using the coupling between SLE and Liouville quantum gravity (LQG),  and the exact solvability of Liouville conformal field theory (CFT).  This approach to the exact solvability of SLE was developed by the third-named author with Ang, Holden, Remy in~\cite{AHS21,AS21,ARS21}. In Section~\ref{subsec:intro-lqg}, we summarize the main idea and the novel difficulty in our setting. Our method also works for SLE$_\kappa$ with $\kappa\in (4,8)$. The result we obtained is expected to give the monochromatic two-arm exponent for the $q$-random-cluster model with cluster weight $q\in (0,4)$, and the  Hausdorff dimension of the so-called thin gasket for conformal loop ensemble with $\kappa\in (4,8)$. For all rational $\kappa$, the exponent is transcendental. 

\subsection{Main results on the backbone exponent}\label{subsec:def}

Bernoulli percolation can be described as follows (precise definitions will be given in Section~\ref{subsec:percolation}). We consider an infinite lattice, such as the hypercubic lattice $\bbZ^d$: vertices are points with integer coordinates, and edges connect any two vertices which are at a Euclidean distance $1$, i.e. differing by $\pm 1$ along exactly one of the $d$ coordinates. Let $p \in (0,1)$ ($p$ is called the \emph{percolation parameter}). Bernoulli site percolation is obtained by tossing a (biased) coin for each vertex $v$: $v$ is colored black with probability $p$, and white otherwise (there is also another standard version of Bernoulli percolation, called bond percolation, where one colors the edges instead). We are then typically interested in the connectivity property of the random coloring, i.e. of the picture obtained by grouping vertices into black and white connected components, also called \emph{clusters}. In particular, under general hypotheses on the underlying lattice, there is a percolation threshold $p_c \in (0,1)$ such that there exists (almost surely) no infinite connected component of black sites when $p < p_c$, while for $p > p_c$, there is at least one (and exactly one in the case of finite-dimensional lattices such as $\bbZ^d$).

A lot is now known in the case of the (planar) triangular lattice $\bbT$, where $p_c = \frac{1}{2}$ from a ``self-duality'' property of $\bbT$. A groundbreaking result of Smirnov in 2001 \cite{Sm01} established the conformal invariance of percolation at criticality (i.e. for $p=\frac{1}{2}$) in the continuum limit, a very strong property which had been uncovered in physics \cite{Ca92, LPS94} (and is expected to hold on any ``reasonable'' two-dimensional lattice with enough symmetry). The introduction of Schramm Loewner Evolution (SLE) processes in \cite{Sc00} revolutionized the study of such systems, and one of its early striking successes was the computation of most classical critical exponents for Bernoulli percolation on $\bbT$ -- the link between the discrete model and SLE (with parameter $\kappa = 6$ in this case) relying on Smirnov's result. In particular, it was shown in \cite{SW01} that the probability $\theta(p)$ for a given vertex on the lattice to belong to an infinite (black) cluster decays like $\theta(p) = (p-p_c)^{5/36+o(1)}$ as $p \searrow p_c$.

This result used the so-called $1$-arm and $4$-arm exponents for Bernoulli percolation, derived in \cite{LSW02} and \cite{SW01} based on earlier computations for SLE \cite{LSW01a, LSW01b}. 
Let us denote these exponents by $\alpha_1$ and $\alpha_4$, respectively. They describe the asymptotic behavior, as $n \to \infty$, of two events shown in Figure~\ref{fig:arm_events}. First, the event that the connected component of $0$ reaches distance at least $n$. This is equivalent to the existence of a path (``arm'') which is black, i.e. made entirely of black vertices, from $0$ to a vertex at distance $n$, and its probability decays like $n^{-\alpha_1+o(1)}$ as $n \to \infty$. Second, the existence of four paths with alternating colors, as shown in Figure~\ref{fig:arm_events}, going to distance $n$ as well. This latter event can be interpreted as two \emph{distinct} black clusters getting close to each other (within a distance $1$) around $0$, and it has a probability $n^{-\alpha_4+o(1)}$. In order to get the exponent $\frac{5}{36}$ for $\theta$ (as well as various other critical exponents), another key input was Kesten's scaling relations for two-dimensional percolation \cite{Ke87}, which connect quantities such as $\theta$ to the above-mentioned arm events.

More generally, arm events can be considered for any prescribed sequence of colors, replacing $B$ and $BWBW$ (with obvious notation) by sequences such as $BWWBWBBW$. Here, we always require two successive arms with the same color to be disjoint, i.e. two different arms cannot use the same vertex. The corresponding exponents can be computed for almost all patterns: these are the celebrated polychromatic exponents \cite{ADA99}, established in \cite{SW01}. They take the value $\alpha_j = \frac{j^2-1}{12}$, for any sequence of length $j \geq 2$ containing both colors (the actual pattern does not matter for the asymptotic behavior, just the total number of arms). Furthermore, the one-arm exponent is also known rigorously \cite{LSW02}, equal to $\alpha_1 = \frac{5}{48}$ (note that it does not fit into the previous family).

\begin{figure}
	\centering
	\subfigure{\includegraphics[width=.3\textwidth]{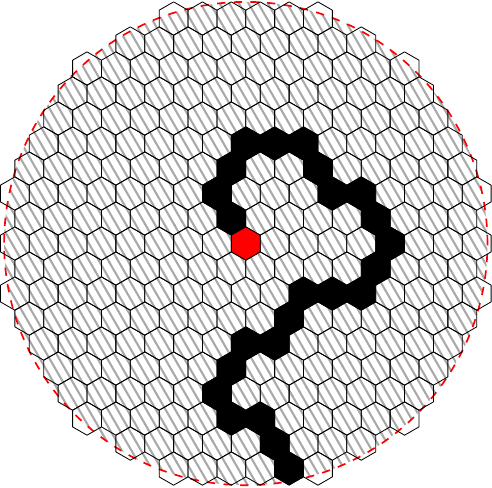}}
	\hspace{0.5cm}
	\subfigure{\includegraphics[width=.3\textwidth]{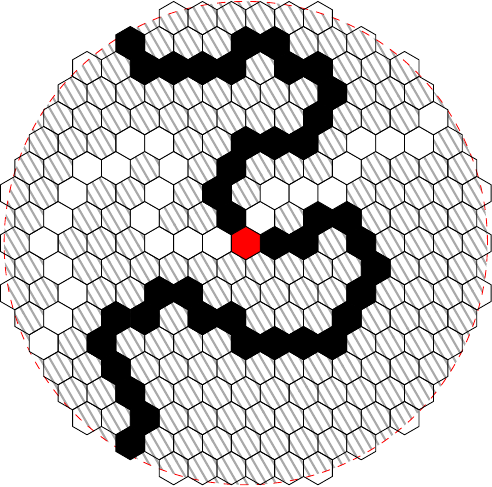}}
		\hspace{0.5cm}
	\subfigure{\includegraphics[width=.3\textwidth]{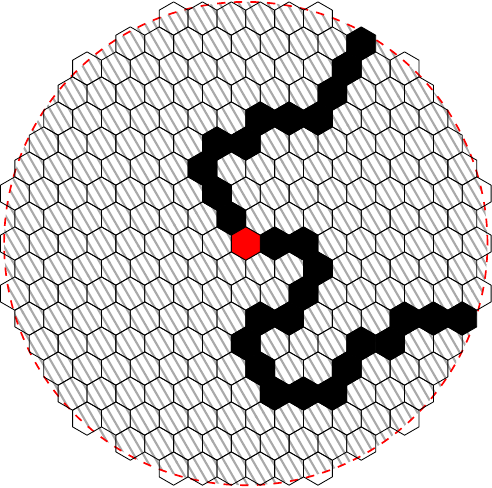}}
	\caption{Site percolation on the triangular lattice is often represented as a random coloring of the faces of the dual hexagonal lattice. \emph{Left:} The one-arm event corresponds to the existence of a black path connecting the vertex $0$ (indicated by a red hexagon) to distance $n$. \emph{Center:} The four-arm event requests the existence of four paths with alternating colors, i.e. two black ones and two white ones, each connecting a neighbor of $0$ to distance $n$. It can be interpreted as the event that two distinct connected components ``meet'' in $0$. \emph{Right:} In this paper, we determine the exponent corresponding to the existence of two disjoint black arms. This monochromatic two-arm exponent is most often called \emph{backbone exponent}.}
	\label{fig:arm_events}
\end{figure}

However, the exponents are different in the case of $j \geq 2$ arms with the same color (this was shown rigorously in \cite{BN11}, based on a discrete combinatorial argument), and the computation of these monochromatic exponents remained elusive so far. Of particular interest is the \emph{backbone exponent}, corresponding to $j=2$ disjoint arms of the same color. This exponent was originally introduced in physics, so as to describe the ``inner skeleton'' of a large percolation cluster at criticality (see Figure~\ref{fig:sim_bb}), where e.g. electrical current between distant vertices would flow. Formally, the backbone exponent $\xi$ is defined as
\begin{equation}\label{eq:backbone-def}
    \xi := - \lim_{n\to\infty}\frac{\log \pi_{BB}(0,n)}{\log n},
\end{equation}
where $\pi_{BB}(0,n)$ denotes the probability at $p_c$ that there exist two disjoint black arms, each from a neighboring vertex of $0$ to distance $n$ (this limit was shown to exist in \cite{BN11}). Our main result is the exact evaluation of $\xi$.
\begin{theorem}
\label{thm:kappa=6}
    The backbone exponent $\xi$ is the unique solution in the interval $(\frac{1}{4},\frac{2}{3})$ to the equation
    \begin{equation}
        \label{eq:sol-xi}
        \frac{\sqrt{36 \xi +3}}{4} + \sin \Big(\frac{2 \pi \sqrt{12\xi+1}}{3} \Big) =0.   
    \end{equation}
\end{theorem}
To see that Equation~\eqref{eq:sol-xi} has a unique solution in  $(\frac{1}{4},\frac{2}{3})$, by taking $\rho = \sqrt{12\xi+1}$, it suffices to show that $f(\rho) = \frac{\sqrt{3}}{4} \rho +\sin(\frac{2\pi \rho}{3})$ has a unique root in the interval $(2,3)$. This holds because $f(2) = 0$ and $f(3) = \frac{3\sqrt{3}}{4}>0$; moreover, $f(\rho)$ is decreasing in $(2,\rho_0)$ and increasing in $(\rho_0,3)$, where $\rho_0 = 3-\frac{3}{2\pi} \arccos(-\frac{3\sqrt{3}}{8\pi}) \in (2,3)$. Using~\eqref{eq:sol-xi}, we obtain the numerical value of $\xi$:
\begin{equation}
  \xi = 0.35666683671288\ldots.   
\end{equation} 
The best numerical result obtained so far is $\xi = 0.35661 \pm 0.00005$ in \cite{FKZD22}, which is based on Monte Carlo simulations.  See also \cite{G99, JZJ02, DBN04} for earlier numerical approximation results on $\xi$.

As mentioned above, the one-arm exponent and the polychromatic $j$-arm exponents, $j \geq 2$, all have rational values. In contrast, the backbone exponent is irrational. Recall that a complex number is said to be algebraic if it belongs to the algebraic closure of $\bbQ$ (the set of rational numbers), i.e. if it can be represented as one of the roots of a polynomial with rational coefficients, and it is called transcendental otherwise. Using Theorem~\ref{thm:kappa=6} and standard results in transcendental number theory, we can show that the backbone exponent is not only irrational, but in fact transcendental.
\begin{theorem}\label{thm:tran6}
   The backbone exponent  $\xi$ in Theorem~\ref{thm:kappa=6} is a transcendental number.
\end{theorem}
Given Theorem~\ref{thm:kappa=6}, the proof of Theorem~\ref{thm:tran6} is quite short and we provide it here.
\begin{proof}
    Let $\rho = \sqrt{12\xi+1}$. If $\rho$ is algebraic, then $x = e^{\frac{2 \pi i \rho}{3}} = (-1)^\frac{2 \rho}{3}$ is also algebraic as it solves $\frac{\sqrt{3}\rho}{4} + \frac{1}{2i}(x-x^{-1})=0$. By the Gelfond-Schneider theorem (see e.g.\ Theorem 10.1 in \cite{N56}), $\frac{2 \rho}{3}$ must be rational, and therefore, $\rho$ must be rational. However, if $\rho$ is rational, then necessarily $\cos(\frac{4 \pi \rho}{3}) = 1-2 \sin(\frac{2 \pi \rho}{3})^2 = 1-\frac{3\rho^2}{8}$ is rational. This is in contradiction with Niven's theorem (see Corollary 3.12 in \cite{N56}), which ensures that if $\frac{4\rho}{3}$ and $\cos(\frac{4\rho}{3}\pi)$ are both rational, then $\frac{4\rho}{3} = \frac{k}{6}$ for some integer $k$. Therefore, $\rho$ is transcendental, and so is $\xi$.
\end{proof}

\begin{figure}[!t]
\centering
\includegraphics[width=0.7\textwidth]{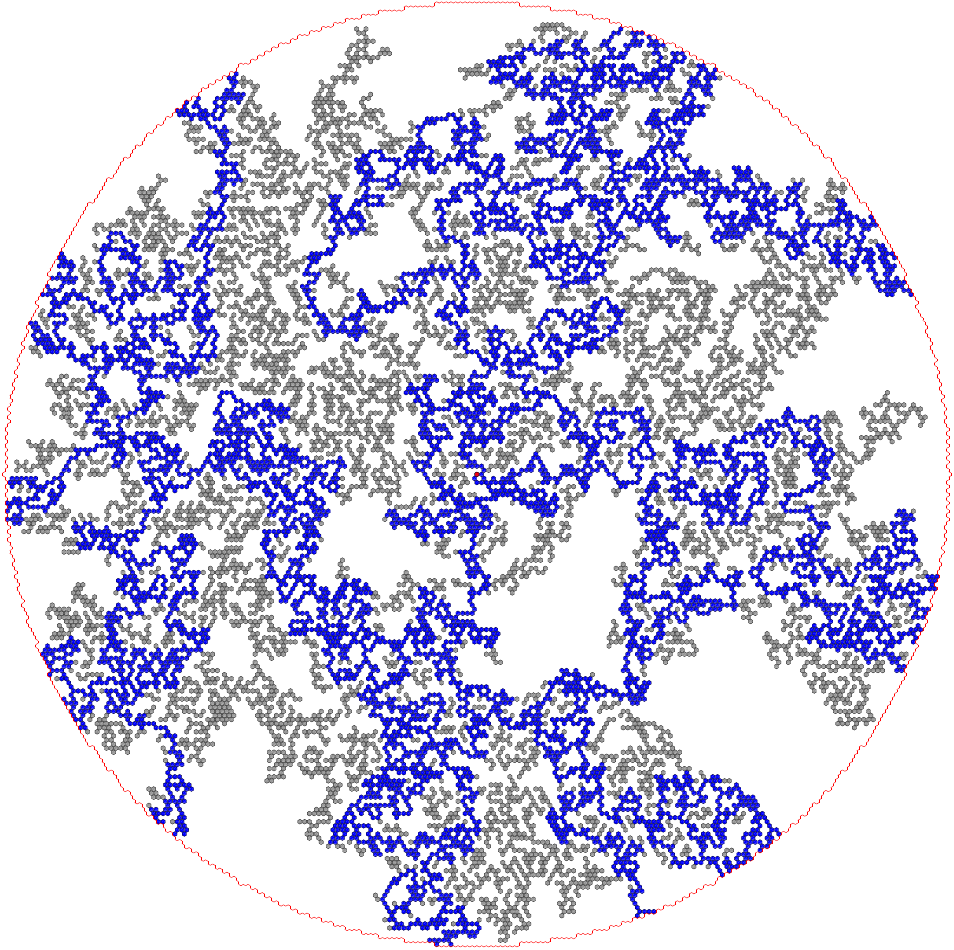}
\caption{This figure considers the (Euclidean) ball $B_n$ with radius $n=100$ centered on the vertex $0$ (indicated by a red hexagon). It shows the cluster of $0$ in $B_n$ at $p = p_c = \frac{1}{2}$, conditionally on the existence of one arm, that is, under the requirement that this cluster reaches the boundary of $B_n$. The backbone in the sense of \cite{Ke86b}, i.e. the set of the vertices which are connected by two disjoint paths to 0 and the boundary, respectively, is then depicted in blue. In words, this means that the gray parts are those where a random walk is just ``wasting time''.}
\label{fig:sim_bb}
\end{figure}

We conclude this subsection with some historical remarks. The backbone is a very natural geometric object, and because of this, it arises in a wide variety of problems in physics. In 1971, Last and Thouless \cite{LT71} first measured experimentally the conductivity of a randomly perforated sheet of conducting graphite paper. Punching holes at the sites of a square lattice, they observed that as the density of holes approaches $1- p_c$, the conductivity $\Sigma$ between two opposite sides plummets much less sharply than the density $\theta$ of the infinite cluster. To explain this phenomenon, they argued that as $p$ approaches $p_c$, paths on the infinite cluster become very constricted, and many are actually dead ends, which lead nowhere and are thus useless for conduction. The idea that slightly supercritical infinite clusters can be represented as a skeleton, to which many dangling ends are attached, was developed successfully in many subsequent works, see for example \cite{SS75}, \cite{DG76}, \cite{St77} (in particular Figure~1 in that paper).

We also want to mention that in our derivation of the backbone exponent, an important role is played by the external frontier (or accessible perimeter) of a percolation cluster, obtained by removing those sections of the boundary which can only be reached through ``narrow'' passages. To our knowledge, this object was first studied in \cite{GA86}, and its introduction was motivated by a clear physical question: which part of the boundary is really visible from the outside, e.g. by particles diffusing randomly in the medium around the cluster? For more references about the geometry of percolation clusters and the related physical aspects, the reader can consult \cite{SA92} (especially Chapters~5 and 6).

As $p$ approaches $p_c$ from above, the unique infinite cluster looks more and more like a critical cluster conditioned to be large, an object called the Incipient Infinite Cluster (IIC for short) which was constructed rigorously by Kesten \cite{Ke86a}. On the mathematical side, there have been several works using that the backbone has a negligible size compared to the whole IIC. The fact that a random walk spends most of its time in dead ends was used by Kesten to establish subdiffusivity of the random walk on the IIC \cite{Ke86b}. Roughly speaking, it was shown in that paper that it takes a time at least $n^{2+\eps}$, for some $\eps > 0$, for the random walk to reach distance $n$. Such statements are made quantitative in \cite{DHS13}, in terms of the one-arm and the backbone exponents, in this situation and also for a related model coming from statistical physics, known as invasion percolation. Finally, a very recent work \cite{GL22} establishes subdiffusivity in the chemical distance (i.e. the distance on the percolation cluster), which is a stronger result. Moreover, that paper provides explicit bounds in terms of the same two exponents $\alpha_1$ and $\xi$.

So far, few results were known rigorously about the backbone exponent, and more generally the monochromatic $j$-arm exponents, $j \geq 2$. In \cite{BN11}, their existence was proved, as well as the fact that the $j$th exponent lies strictly between the $j$th and the $(j+1)$th polychromatic ones (so that in particular, they do form a distinct family of exponents). This comparison was based on a counting inequality for subsets $A$, $B$ of $\{0,1\}^n$, for any $n \geq 1$, namely $|A \circ B| \leq |A \cap \tilde{B}|$, where $A \circ B$ denotes the \emph{disjoint occurrence} of $A$ and $B$ (see \cite{BK85}), and $\tilde{B}$ is obtained from $B$ by switching $0$ and $1$, i.e. $\tilde{B} := \{ 1 - \omega \: : \: \omega \in B\}$ (in this general form, this inequality is due to Reimer \cite{Re00}). In the proof, it appeared that monochromatic events are qualitatively very different from polychromatic ones. In the latter case, the ``winding'' angle of the arms is essentially prescribed by the configuration, while in the former case, we have a lot of freedom on how to choose the arms, so a whole interval of winding angles can be achieved, typically. This observation actually highlights the fundamental difference between the two cases, and illustrates why the monochromatic exponents are seemingly more complicated to grasp.

\subsection{The backbone exponent in terms of the SLE bubble measure}\label{subsec:intro-sle}
By the seminal works of Schramm~\cite{Sc00} and Smirnov~\cite{Sm01}, it is known that for critical site percolation on the triangular lattice with mesh size $\delta$, boundaries of connected components are described by $\SLE_6$ in the scaling limit, i.e.\ as $\delta\to 0$. More specifically, given a simply connected domain $D$ with a nice enough boundary, and two boundary points $a$ and $b$, we can consider critical percolation in $D$ with Dobrushin boundary conditions (i.e.\ with one boundary arc colored black and the complementary arc colored white). If we consider the white and black clusters of these two arcs, the interface between them converges to a chordal $\SLE_6$ in $D$ from $a$ to $b$. 
On the other hand, if we color the entire boundary with the same color instead, say white, it is proved in \cite{CN06} that the whole collection of interfaces between black and white clusters converges to a continuum limit. This limit is a random collection of non-simple non-crossing loops which can be constructed using $\SLE_6$. In \cite{MR2494457}, Sheffield introduced a more general one-parameter family of random collections of non-crossing loops, under the name conformal loop ensemble (CLE) and denoted by CLE$_\kappa$. The collection of loops in \cite{CN06} corresponds to CLE$_6$.

The starting point of our proof of Theorem~\ref{thm:kappa=6} is to express the backbone exponent in terms of CLE$_6$. More precisely, we relate the percolation backbone exponent to the external frontier of an outermost white cluster, which can be approximated by the outer boundary of the corresponding CLE$_6$ loop. We then exploit a relation between CLE$_\kappa$ and the SLE$_\kappa$ bubble measure, analogous to the one described in \cite[Section 6.1]{MR3708206}. The SLE$_\kappa$ bubble measure was introduced in \cite{MR2979861}, as it arises naturally in the construction of CLE$_\kappa$ for $\kappa\in(8/3,4]$. The works \cite{MR3708206,MSW-non-simple} pointed out that the same procedure also allows one to define measures on SLE$_\kappa(\rho)$ bubbles,  and further related SLE$_\kappa(\rho)$ bubbles to CLE$_\kappa$ and BCLE$_\kappa$ (boundary CLE) in the case $\kappa\in(2,8)$.
Recently, \cite{Z22} systematically studied SLE$_\kappa(\rho)$ bubble measures for $\kappa>0$ and $\rho>-2$, and proved further useful properties.

\begin{figure}[h!]%
    \centering
    \includegraphics[width=5cm]{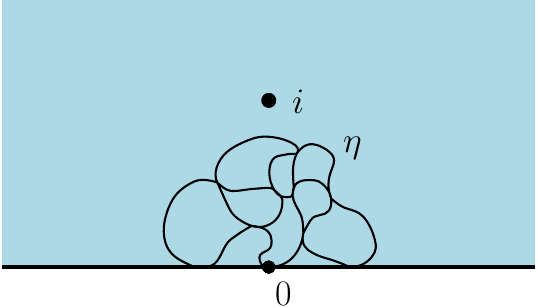}
    \qquad \qquad
    \includegraphics[width=5cm]{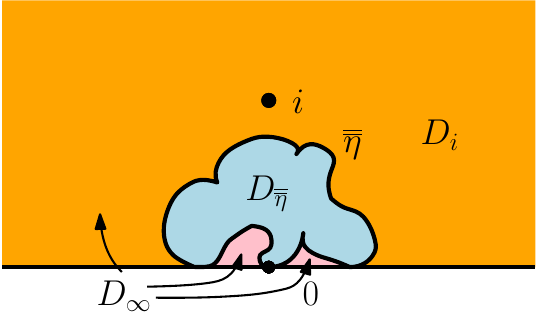}
    \caption{In the left picture, we show an ${\rm SLE}_\kappa$ bubble $\eta$ on the complement of the event $E_i$. In the right picture, we show its outer boundary $\ep \eta$, as well as the orange domain $D_i$ and the light cyan domain $D_{\ep \eta}$. The domain $D_\infty$ refers to the union of the orange and pink regions.
    }%
    \label{fig:sle-bubble}%
\end{figure}

In this subsection, we only provide the minimal amount background on the  $\SLE_\kappa$ bubble measure in order to convey the main idea of the proof of Theorem~\ref{thm:kappa=6} and its generalization to $\kappa\in (4,8)$. More details will be given in Section~\ref{sec:bubble}. 
For $\kappa\in (4,8)$, consider the SLE$_\kappa$ bubble measure $\mu_\kappa$ rooted at the origin on the upper half plane $\bbH$. It is an infinite measure supported on random curves on $\bbH \cup \bbR$ passing through the origin. The random curve  is almost surely non-simple and touches $\partial \bbH=\bbR$ infinitely many times. See Figure~\ref{fig:sle-bubble} for an illustration.  The measure $\mu_\kappa$ can be obtained from the $\epsilon\to 0$ limit of the chordal SLE$_\kappa$ measure on $\bbH$ from $0$ to  $-\epsilon$, which is only defined up to a multiplicative constant. We will fix the constant in the following way. Let $D_i$ be the connected component of the complement of the curve on $\bbH$ containing $i\in \bbH$.  Let $E_i$ be the event that $\partial D_i\cap \bbR$ is empty. It can be shown that $\mu_\kappa(E_i)<\infty$. We fix the multiplicative constant by requiring $\mu_\kappa(E_i)=1$. Let $\nu_\kappa$ be the restriction of $\mu_\kappa$ to the complement of $E_i$.   Let $\psi$ be a conformal map from $\mathbb{H}$ to $D_i$ such that $\psi(i)=i$. Then the backbone exponent can be expressed as follows.

\begin{theorem} 
    \label{thm:solution}
    The backbone exponent $\xi$ from~\eqref{eq:backbone-def} satisfies
    \begin{equation}
    \label{eq:thm-solution}
    \nu_6[|\psi'(i)|^{-\xi} -1  ]=1. 
    \end{equation}
\end{theorem}
Theorem~\ref{thm:solution} is proved in Section~\ref{sec:bubble}. It is believed that the connection to CLE$_6$ holds for a wide class of 2D percolation models, in particular for bond percolation (obtained by ``coloring'' the edges, instead of the vertices) on the square lattice $\bbZ^2$. Recently, rotational invariance was proved for sub-sequential limits in that case~\cite{DC-invariance}. Our Theorem~\ref{thm:moment} below exactly solves $\nu_\kappa[|\psi'(i)|^{-\xi} -1]$, which combined with Theorem~\ref{thm:solution} yields Theorem~\ref{thm:kappa=6}.
\begin{theorem}
    \label{thm:moment}
    Fix $4<\kappa<8$. Then $\nu_\kappa[|\psi'(i)|^\lambda -1 ]<\infty$ for  $\lambda > \frac{2}{\kappa}-1$ and $\lambda\mapsto \nu_\kappa[|\psi'(i)|^\lambda -1]$ is analytic in $D := \{ \lambda: \mathrm{Re} \lambda>\frac{2}{\kappa}-1  \}$.
  Let $\theta=\pm\sqrt{(\frac{\kappa}{4}-1)^2-\frac{\kappa}{2}\lambda}$. Then for all $\lambda \in D$,
    \begin{equation}
    \label{eq:thm-moment}
    \nu_\kappa[|\psi'(i)|^\lambda -1 ] =1 + \frac{2   \Gamma(\frac{4(1-\theta)}{\kappa})\Gamma(\frac{4(1+\theta)}{\kappa}) }{\kappa\cos(\frac{4\pi}{\kappa}) \Gamma(\frac{8}{\kappa}-1) \sin(\frac{4\pi  \theta}{\kappa})}\Big(\sin(\frac{8\pi  \theta}{\kappa}) - \theta \sin(\frac{8\pi}{\kappa}) \Big)\,.
    \end{equation}
Here the right side of~\eqref{eq:thm-moment} is an even function in $\theta$ so that its value is fixed by  $\theta^2=(\frac{\kappa}{4}-1)^2-\frac{\kappa}{2}\lambda$.
\end{theorem}

In light of Theorems~\ref{thm:solution} and~\ref{thm:moment}, for $\kappa\in (4,8)$, we define the SLE$_\kappa$ bubble exponent $\xi=\xi(\kappa)$ to be the unique number in  $(0,1-\frac{2}{\kappa})$ such that 
    \begin{equation}\label{eq:exponent-def}
    \nu_\kappa[|\psi'(i)|^{-\xi(\kappa)} -1  ]=1. 
    \end{equation}
We have the following description of $\xi(\kappa)$  which implies Theorem~\ref{thm:kappa=6} once setting $\kappa=6$.  
\begin{theorem}
\label{thm:1.2}
 For $\kappa \in (4,8)$,  the equation
    \begin{equation}
    \label{eq:backbone-generick}
    \sin(\frac{8 \pi}{\kappa})\sqrt{\frac{\kappa  x}{2} + (1-\frac{\kappa}{4})^2} - \sin \bigg( \frac{8\pi}{\kappa}\sqrt{\frac{\kappa x}{2} + (1-\frac{\kappa}{4})^2} \bigg) = 0
    \end{equation}
has at most two solutions in $(0,1-2/\kappa)$, which are $1-\kappa/8$ and   $\xi(\kappa)$. Let $\kappa_0$ be the unique point in $(4,8)$ such that $\tan(8\pi/\kappa_0)=8\pi/\kappa_0$.  For $\kappa=\kappa_0$, we have $\xi(\kappa)=1-\kappa/8$.    Otherwise, $\xi(\kappa)\neq 1-\kappa/8$.
\end{theorem}
\begin{proof}
Taking $\lambda = -\xi$ and $\theta = \pm \sqrt{(\frac{\kappa}{4} - 1)^2 + \frac{\kappa \xi}{2} }$ in Theorem~\ref{thm:moment} and combining it with~\eqref{eq:exponent-def} yields the desired result.
\end{proof}

From our proof of Theorem~\ref{thm:solution} in Section~\ref{sec:bubble}, it will be natural to expect that $2-\xi(\kappa)$ gives the fractal dimension of the so-called thin gasket of CLE$_\kappa$, which is the remaining region after removing every loop in the CLE$_\kappa$ and all the bounded regions that are disconnected from $\infty$ by the loops.
Recall that percolation is the Fortuin-Kasteleyn (FK) random-cluster model with cluster weight $q=1$.
It is widely conjectured that CLE$_\kappa$ for $\kappa\in [4,8)$ describes the scaling limit of these models with $q\in (0,4]$ and  $\kappa=\frac{4\pi}{\pi - \arccos(\sqrt{q}/2)}$. This conjecture was proved for $q=2$ and $\kappa=16/3$, namely, the FK Ising case~\cite{CDHC-Ising}. We expect that $\xi(\kappa)$ describes the backbone exponent for the $q$-random-cluster model for all $q\in (0,4]$, with $\xi(4)=\lim_{\kappa\to 4}\xi(\kappa)$.
For the FK Ising case, we believe that all the key ingredients are available; See Remark~\ref{rmk:Ising} for detail. 
Sharp numerical approximation for $q\in \{1,2,2+\sqrt{3},3,4\}$ was obtained in~\cite{DBN04, FKZD22}.  
Our exact value of $\xi(\kappa)$ in Theorem~\ref{thm:1.2} matches their numerical results perfectly, see the table below. 
\begin{table}[ht] 
\centering
 \begin{tabular}{||c c c||} 
 \hline
 $(q,\kappa)$ & Numerical & Our formula \\ [0.5ex] 
 \hline\hline
 (1,6) & 0.35661(5) & 0.3566668367\ldots  \\
 (2,16/3) & 0.26773(8) & 0.2678678166\ldots \\
 (3,24/5) & 0.2062(3) & 0.2059232891\ldots  \\
 (2+$\sqrt{3}$,48/11) & 0.1616(5)  & 0.1602191369 \ldots \\
 (4,4) & 0.1247(6) & 0.125  \\ [1ex] 
 \hline
 \end{tabular}
\end{table}

As an intriguing fact, the solution  $1-\kappa/8$ to~\eqref{eq:backbone-generick} should be equal to the polychromatic $2$-arm exponent of the FK random-cluster model. We are curious if there is a conceptual explanation for this fact. Note that the monochromatic 2-arm exponent is smaller than the polychromatic one if and only if $\kappa\in(4,\kappa_0)$. The numerical value of $\kappa_0$ is $5.593245\ldots$. On the other hand, the number $1-2/\kappa$ in Theorem~\ref{thm:1.2} should correspond to the polychromatic 3-arm exponent,  which is greater than the monochromatic 2-arm exponent (see Lemma~\ref{lem:er2} and Remark~\ref{rmk:alpha}).

As a final observation, we have the following extension of Theorem~\ref{thm:tran6}.
\begin{theorem}
\label{thm:transc}
    For any rational $\kappa \in (4,8)$, the  exponent $\xi(\kappa)$ is transcendental.
\end{theorem}
\noindent The proof of Theorem~\ref{thm:transc} is similar to Theorem~\ref{thm:tran6} but more technical. We give the details in Section~\ref{sec:trans}.

\subsection{Exact solvability of SLE via Liouville quantum gravity and Liouville CFT}\label{subsec:intro-lqg}

Liouville quantum gravity (LQG) is a theory of random surfaces originated from string theory~\cite{Polyakov81}. Starting from~\cite{DS11}, LQG has become a central topic in random geometry. See~\cite{Sheffield-ICM} for a comprehensive survey on LQG, including its connection to random planar maps~\cite{LeGall-BM,Mier-BM,GMS-RWRE,HS-Cardy}. A key aspect of LQG is its coupling with SLE, which describes the scaling limit of random planar maps decorated with statistical physics models. As first demonstrated by Sheffield~\cite{Shef-zipper}, when two LQG surfaces are welded together conformally along their boundaries, the resulting interface is an SLE curve. This perspective was further developed by Duplantier, Miller, and Sheffield~\cite{DMS14} which gives birth to the mating-of-trees theory, a framework to study the SLE/LQG coupling in terms of more classical probabilistic objects such as Brownian motion, L\'evy process, and Bessel process. Among subsequent works in this vein,  \cite{MSW-simple, MSW-non-simple} established useful connections between CLE, SLE bubbles, and LQG surfaces. See~\cite{GHS-survey} for a survey.

The idea of using LQG to derive results which are purely about lattice models and SLE is an important motivation to have a probabilistic framework for LQG in the first place~\cite{DS11}, especially under the name of Knizhnik-Polyakov-Zamolodchikov (KPZ) relation. There are a number of successful applications of this idea, such as~\cite{GHM-KPZ} and~\cite{MSW-non-simple} discussed in Section~\ref{subsub:related} below. 
The main novelty of recent applications, including those in our current paper, lies in the incorporation of Liouville CFT, which we now review.

The geometry of LQG surfaces is governed by variants of Gaussian free field (GFF), a family of random generalized functions with  logarithmic correlation. In particular, the random area of the surface is of the form $e^{\gamma h} \, dxdy$ where $\gamma\in (0,2)$ is a parameter and $h$ is a variant of GFF.
As suggested in Polyakov's seminal work~\cite{Polyakov81}, the variants of GFF  that are most relevant to string theory and 2D quantum gravity are sampled according to Liouville conformal field theory (CFT), namely  the quantum field theory with Liouville action. Liouville CFT was made sense rigorously  by David, Kupiainen, Rhodes, and Vargas~\cite{DKRV} and follow-up works. As a primary example of CFT in the framework of Belavin, Polyakov, and Zamolodchikov~\cite{BPZ84}, Liouville CFT enjoys rich and deep exact solvability. This was demonstrated in physics and rigorously shown by Kupiainen, Guillarmou,  Rhodes, and Vargas  recently; see~\cite{KRV-DOZZ,GKRV-sphere} and references therein.

It was proved at various levels of generality and precision that natural  LQG surfaces in the mating of trees framework, despite their different-looking appearance, can be described by Liouville CFT; see~\cite{AHS-sphere,Cerle-disk,AHS21,ASY22}. Recently, with Ang, Holden, and Remy in~\cite{AHS21,ARS21,AS21}, the third named author combined the mating of trees theory and the Liouville CFT   to produce a number of exact results in both frameworks which are hard to access using a single framework. Moreover, several exact formulae purely about SLE were derived through the coupling with LQG. Besides the aforementioned foundational work on mating of trees and Liouville CFT, the key ingredients for this set of work include the conformal welding of finite-area LQG surfaces (see~\cite{AHS20,AHS21,ASY22})  and the exact solvability of Liouville CFT on surfaces with boundary~\cite{Remy-Duke,RZ22,ARS21,ARSZ23}.

Our proof of Theorem~\ref{thm:moment} is another example of an exact formula for SLE obtained from the SLE/LQG coupling and the Liouville CFT. Our starting point is a result from Wu~\cite{Wu23}, which explains how an independent coupling of SLE bubble and Liouville fields arises from the conformal welding of LQG surfaces for $\kappa\in (0,4)$. Although we are considering the SLE$_\kappa$ bubble for $\kappa\in (4,8)$, Theorem~\ref{thm:moment} is in fact a statement about the outer boundary of the SLE$_\kappa$ bubble, which is a variant of SLE$_{16/\kappa}$ bubble covered by~\cite{Wu23}. Following the method in~\cite{AHS21}, the paper~\cite{Wu23}
gave an exact formula for the moment of a conformal derivative defined in terms of the SLE bubble measure.  However, in all previous exact results on SLE obtained this way, including the ones in~\cite{AHS21,AS21,ASY22}, the measure on curves is a probability measure. In the contrast, the measure $\nu_\kappa$ in Theorem~\ref{thm:moment} is infinite. In fact, if we consider $ \nu_\kappa[|\psi'(i)|^\lambda] $ instead of $ \nu_\kappa[|\psi'(i)|^\lambda -1 ]$ in~\eqref{eq:thm-moment}, we would get infinity. To handle this difficulty, we introduce an auxiliary conformal derivative in addition to $\psi'(i)$ and consider the joint moment of the two. In a certain range of parameters their joint moment is finite without subtracting 1 as in~\eqref{eq:thm-moment}.  We first derive an exact formula for this joint moment and then prove~\eqref{eq:thm-moment} via analytic continuation.

We prove Theorem~\ref{thm:moment}  in Sections~\ref{sec:welding} and~\ref{sec:solve}. Section~\ref{sec:welding} contains  the needed background on LQG and Liouville CFT, and most of the geometric arguments. Section~\ref{sec:solve} is devoted to calculations, which are quite involved. We find it surprising that the end result~\eqref{eq:thm-moment} is so simple. In fact,
we first express the aforementioned joint moment of conformal derivatives  in terms of the so-called boundary  structure constants in the Liouville theory without bulk potential. Although they were fully determined in~\cite{RZ22}, the formulae contain complicated special functions and integrals. We have to use the exact solvability from the mating-of-trees framework to land at~\eqref{eq:thm-moment}.

\subsection{Outlook and perspectives}\label{subsec:outlook}
We conclude the introduction with comments on related works  and open questions. 
\subsubsection{Related work}\label{subsub:related}
\begin{itemize}
\item There was an early attempt in \cite{LSW02} to derive the backbone exponent for percolation by using directly SLE$_6$. In Appendix~B of that paper, Lawler, Schramm and Werner sketch how to derive an explicit PDE problem, which should yield the backbone exponent (as the leading eigenvalue of some differential operator, with prescribed boundary conditions). However, the operator which arises is two-dimensional, contrary to the one-arm exponent and the polychromatic case where the derivations involve only one-dimensional operators, and it is quite singular. Because of this, it has not been proved useful to predict the value of the exponent.
More recently, various mathematical issues related to that operator, its regularity in particular, were considered by Garban and Mourrat~\cite{GM}.

\item As mentioned below Theorem~\ref{thm:1.2}, the continuum counterpart of the backbone of a FK random cluster is the thin gasket of CLE$_\kappa$, whose fractal dimension should be $2-\xi(\kappa)$. In~\cite{GHM-KPZ}, Gwynne, Holden, and Miller established a version of the KPZ relation. It implies a relation between the dimension of fractal sets defined by CLE$_\kappa$ for $\kappa\in (4,8)$ and 2D Brownian motion, the latter of which encodes the former via the mating-of-trees theory~\cite{DMS14}. This relates the backbone exponent with the dimension of a fractal set defined by Brownian motion; see~\cite[Section 7]{GHM-KPZ}. This reduction does not seem to simplify the problem.
In fact, our result yields the value of the dimension of that fractal set, which is also 
transcendental since the KPZ relation is quadratic.

\item The fuzzy Potts model is a generalization of the $q$-Potts model, where we color each connected component in  the critical  $q$-random-cluster model in black and white, with probability $r$ and $1-r$ respectively.
Based on the CLE percolation framework of~\cite{MR3708206}, the papers \cite{MSW-non-simple, KL-Fuzzy} explained that under the assumption that the random cluster model  converges to CLE$_\kappa$, all the boundary arm exponents and polychromatic interior arm exponents of the fuzzy Potts model can be expressed in terms of an SLE variant called  $\SLE_\kappa(\rho;\kappa-6-\rho)$.  Once reduced to a question about $\SLE_\kappa(\rho;\kappa-6-\rho)$, these arm exponents can be computed via traditional methods, without LQG, in contrast to our case. However, the exact correspondence between $r \in [0,1]$ and $\rho$ is highly nontrivial, and was derived in \cite{MSW-non-simple} using the CLE/LQG coupling developed in~\cite{MSW-simple,MSW-non-simple}, which yields exponents with rather unusual forms. The monochromatic $k$-arm exponent for the fuzzy Potts model does not depend on $k$, and it remains unknown. It would be interesting to see if our method can be used to derive it, and on the other hand, if the techniques in~\cite{MSW-simple,MSW-non-simple} allow one to derive the percolation backbone exponent without Liouville CFT. 

\end{itemize}
\subsubsection{Open questions}
\begin{itemize}
    \item An obvious open question is the exact evaluation of  the monochromatic $j$-arm exponents with $j \ge 3$. Given the simplicity of the backbone exponent, there is a strong hope that these exponents have explicit yet transcendental values. However, we do not  have a conjectural formula for them right now. There  are essential additional complications when we try to extend our method to treat the case of $j \ge 3$ arms, as we discuss in Remark~\ref{rmk:three-arm}.

\item Many exact predictions for 2D critical percolation are (non-rigorously) derived from conformal field theory (CFT) considerations, including the celebrated Cardy's formula~\cite{Ca92}. Recently, a formula on the four-point connection probability was proposed in~\cite{HJS-4pt}.  Is there a  more direct  CFT derivation for the backbone exponent? Such a derivation  could shed lights on the $j \ge 3$ case as well. 
\end{itemize}

\medskip
\noindent\textbf{Acknowledgements.} 
We are grateful to Bertrand Duplantier and Wendelin Werner for helpful communications. We also thank Sylvain Ribault and Rob van den Berg for useful comments on an earlier version of this manuscript.
P.N.\ is partially supported by a GRF grant from the Research Grants Council of the Hong Kong SAR (project CityU11318422).
W.Q.\ is partially supported by a GRF grant from the Research Grants Council of the Hong Kong SAR (project CityU11305823).
X.S.\ was partially supported by the NSF Career award 2046514, a start-up grant from the University of Pennsylvania, and a fellowship from the Institute for Advanced Study (IAS) at Princeton. 
Z.Z.\ was partially supported by NSF grant DMS-1953848. 
W.Q., X.S. and Z.Z. thank the warm hospitality and the stimulating atmosphere of the IAS, where they visited in 2022--2023.

\section{The backbone exponent and the SLE bubble measure}\label{sec:bubble}

In this section, we focus on Bernoulli percolation. In Section~\ref{subsec:percolation}, we first set notations and collect results that are used in our proofs. In Section~\ref{subsec:bb_exp}, we then establish a relation between the backbone exponent $\xi$ for critical percolation and an exponent for outer boundaries of CLE$_\kappa$ loops with $\kappa=6$.  This latter exponent is further related in Section~\ref{subsec:link_bubble} to the exponent $\xi(\kappa)$ defined in \eqref{eq:exponent-def}, which is a quantity about SLE bubbles. The fact that $\xi(\kappa)$ is well defined by \eqref{eq:exponent-def} is a consequence of Theorems~\ref{thm:moment} and~\ref{thm:1.2}. This section takes Theorem~\ref{thm:moment} as an input, which will be proved in the remainder of the paper.

\subsection{Preliminaries on Bernoulli percolation} \label{subsec:percolation}

Let $\|.\|$ denote the Euclidean distance in $\bbR^2$, and for any $r \geq 0$, let $\cB_r := \{z \in \bbR^2 \: : \: \|z\| \leq r\}$ be the closed ball of radius $r$ centered on the point $0$. For a subset $A \subseteq \bbR^2$, its interior and its closure are denoted by $\mathring{A}$ and $\bar{A}$, respectively, and we write $\partial A = \bar{A} \setminus \mathring{A}$ for the boundary of $A$. We also introduce the closed annulus $\cA_{r,r'} := \cB_{r'} \setminus \mathring{\cB_{r}}$, for $0 \leq r < r'$.

In the whole paper, we consider the triangular lattice $\bbT = (V_{\bbT}, E_{\bbT})$, embedded in $\bbR^2$. It is the planar graph with set of vertices
$$V_{\bbT} := \big\{ x + y e^{i \pi / 3} \in \bbC \: : \: x, y \in \bbZ \big\}$$
(identifying, as usual, $\bbC \simeq \bbR^2$), and set of edges
$$E_{\bbT} := \big\{ \{v,v'\} \: : \: v, v' \in V_{\bbT} \text{ with } \|v-v'\|=1 \big\}.$$
That is, we connect by an edge any two vertices $v$, $v'$ at a Euclidean distance $1$. Such vertices are said to be neighbors, and we use the notation $v \sim v'$. For any $r \geq 0$, we denote by $B_r := \cB_r \cap V_{\bbT}$ the set of vertices in $V_{\bbT}$ at a distance at most $r$ from $0$, and we let $A_{r,r'} := \cA_{r,r'} \cap V_{\bbT}$. A path $\gamma$ on $\bbT$ is a finite sequence of vertices $v_0, v_1, \ldots, v_n$, for some $n \geq 0$ called the length of the path, such that $v_i \sim v_{i+1}$ for all $i =0, \ldots, n-1$. We say that $\gamma$ connects $v_0$ and $v_n$. A circuit is a path $\gamma$, with some length $n$, whose vertices are all distinct, except that its endpoints $v_0$ and $v_n$ coincide.

The inner vertex boundary of a subset $A$ of $V_{\bbT}$ is defined as $\din A := \{v \in A \: : \: v \sim v' \text{ for some } v' \in V_{\bbT} \setminus A\}$, while its outer vertex boundary is $\dout A := \{v \in V_{\bbT} \setminus A \: : \: v \sim v' \text{ for some } v' \in A\}$ ($= \din(V_{\bbT} \setminus A)$). We also consider the edge boundary of $A$, which is $\de A := \{ \{v,v'\} \: : \: v \in A \text { and } v' \in V_{\bbT} \setminus A\}$. In the case when $A$ is finite, $\de A$ can be represented as a collection of circuits on the dual hexagonal lattice, the lattice obtained by drawing a hexagon around each $v \in V_{\bbT}$ (appearing in Figure~\ref{fig:arm_events}). The filling of $A$, denoted by $\Fill(A)$, is obtained by considering all finite connected components of $V_{\bbT} \setminus A$, and adding them to $A$. In words, $\Fill(A)$ is obtained by adding to $A$ all the finite ``holes'', containing the vertices disconnected from $\infty$ by $A$. If $A$ is finite, $\Fill(A)$ is simply the complement of the unbounded connected component of $V_{\bbT} \setminus A$. The subset $\de_{\infty} A := \de (\Fill(A))$ of $\de A$ is called external edge boundary of $A$, and it is a single dual circuit if $A$ is finite. We also sometimes use the external vertex boundaries $\din_{\infty} A := \din (\Fill(A))$ $(\subseteq \din A)$ and $\dout_{\infty} A := \dout (\Fill(A))$ $(\subseteq \dout A)$.

Finally, still in the case when $A$ is finite, we introduce its external frontier, which is a key object in our proofs. It is defined as the subset of all $v \in \dout_{\infty} A$ from which it is possible to find two disjoint infinite self-avoiding paths, each starting from a neighbor of $v$ and remaining in $V_{\bbT} \setminus A$. Note that it is a circuit contained in $\dout A$, which surrounds $A$. Informally, it is the boundary of $A$ obtained by ``filling'' passages with width $1$ (``fjords'' in the terminology of \cite{ADA99}).

We can now define Bernoulli site percolation on $\bbT$, by declaring each $v \in V_{\bbT}$ to be, independently of the other vertices, either black or white, with probability $p$ and $1-p$ respectively, for a given value $p \in (0,1)$. We obtain in this way a random coloring of $V_{\bbT}$, represented as a configuration $\omega = (\omega_v)_{v \in V} \in \Omega := \{B,W\}^V$ (where $B$ and $W$ stand for black and white, resp.). We denote by $\bbP_p$ the corresponding product probability measure on $\Omega$. A black (resp. white) path is a path whose vertices are all black (resp. white). Two vertices $v$ and $v'$ are said to be connected if they are both black, and they are connected by a black path, i.e. if there exists a black path $v = v_0, v_1, \ldots, v_n = v'$ for some $n \geq 0$. We denote it by $v \leftrightarrow v'$, and we let $\cC(v) := \{v' \in V_{\bbT} \: : \: v \leftrightarrow v'\}$ be the (black) connected component, or cluster, of $v$ (by convention, $\cC(v) = \emptyset$ if $v$ is white). We observe that if $\cC(v) \neq \emptyset$, then $\din (\cC(v))$ and $\dout (\cC(v))$ are respectively made of black and white vertices.

We let $\theta(p) := \bbP_p(|\cC(v)| = \infty)$, $p \in [0,1]$, be the probability that $v$ belongs to an infinite black cluster (which of course does not depend on $v$). Obviously $\theta(0) = 0$ and $\theta(1) = 1$ (all vertices are white for $p=0$, resp. black for $p=1$). A simple coupling argument yields that $\theta$ is a non-decreasing function, which leads to introduce the site percolation threshold (on $\bbT$) as
$$p_c^{\textrm{site}}(\bbT) := \sup\{ p \in [0,1] \: : \; \theta(p) = 0 \} \quad (= \inf\{ p \in [0,1] \: : \; \theta(p) > 0 \}) \in [0,1].$$
In the following, we simply write $p_c$. One of the first results of percolation theory states that under suitable hypotheses on the underlying graph, which are in particular satisfied by $\bbT$, $0 < p_c < 1$. Furthermore, a symmetry argument suggests that $p_c = \frac{1}{2}$ in this particular case, which is a celebrated result of Kesten \cite{Ke80}. Similarly, we say that two vertices are white-connected if there is a white path between them, giving rise to white clusters.

For any $0 \leq r < r'$, we consider the event $\cA_{BB}(A_{r,r'})$ that there exist two disjoint black paths, each connecting a vertex of $\din B_{r'}$ and a vertex of $\dout B_r$. This event belongs to a more general family of so-called arm events, where we ask the existence of some given number $j \geq 1$ of monochromatic paths (i.e. each of them is either black or white) ``crossing'' the annulus $A_{r,r'}$, with colors prescribed in counterclockwise order (here, by the sequence of colors $BB$). Later, we will use in particular the arm events associated with the sequences $BWW$ and $BWWBWW$.

We write
\begin{equation}
\pi_{BB}(r,r') := \bbP_{p_c}(\cA_{BB}(A_{r,r'})).
\end{equation}
In order to estimate two-arm events asymptotically, the fact that this quantity $\pi_{BB}$ is ``quasi-multiplicative'' plays an important role.
\begin{lemma} \label{lem:qm}
There exist $c_1, c_2 > 0$ such that: for all $0 < r_1 < r_2 < r_3$ with $r_2 \geq 2 r_1$ and $r_3 \geq 2 r_2$,
\begin{equation} \label{eq:qm}
c_1 \cdot \pi_{BB}(r_1,r_2) \pi_{BB}(r_2,r_3) \leq \pi_{BB}(r_1,r_3) \leq c_2 \cdot \pi_{BB}(r_1,r_2) \pi_{BB}(r_2,r_3).
\end{equation}
\end{lemma}
We stress that $c_1$, $c_2$ are ``universal'' constants, i.e. they do not depend on $r_1, r_2, r_3$. The right-hand side inequality follows immediately from the spatial independence of Bernoulli percolation, which implies that the events $\cA_{BB}(A_{r_1,r_2})$ and $\cA_{BB}(A_{r_2,r_3})$ are independent. We can thus simply take $c_2=1$, which we sometimes do. The left-hand side inequality is not difficult to obtain in this monochromatic case, from the Russo-Seymour-Welsh lower bounds. This was briefly mentioned in \cite{BN11} (see the end of Section~2.1 there), and for the sake of completeness we decided to provide a proof, in Appendix~\ref{sec:app_perc}. This quasi-multiplicativity property is standard in the case of Bernoulli percolation. It holds in general for any sequence of colors, but it is harder to establish in the polychromatic case, i.e. when arms of both colors are requested (such as in the $4$-arm case corresponding to $BWBW$). In that case, proving the left-hand side inequality requires first to prove a ``separation'' property for the endpoints of the arms, as done in \cite{Ke87} (see also \cite{No08}). Many other models are known to satisfy such a property, which is often key to connect the discrete model to its continuum limit. In the remainder of the paper, we are only concerned with the monochromatic sequence $BB$, so most often we drop it from the notation. In particular, we simply write $\cA$ and $\pi$.

As sketched in Section~4 of \cite{BN11}, the quasi-multiplicativity property is already enough to ensure the existence of an exponent, as soon as the existence of a limiting continuum probability, in the scaling limit, is checked. Indeed these probabilities are then automatically submultiplicative, which allows one to deduce the existence of the exponent from a standard argument.

In order to relate the discrete $2$-arm monochromatic event to its continuous analog, we rely on two classical a-priori estimates on $3$- and $6$-arm polychromatic events, which follow from standard arguments. For some universal $\beta > 0$, the following upper bounds hold, for all $r, r' > 0$ with $r' \geq 2r$.
\begin{enumerate}[(i)]
\item A-priori estimate on $3$ arms:
\begin{equation} \label{eq:3_arm}
\pi_3(r,r') := \pi_{BWW}(r,r') = \bbP_{p_c}(\cA_{BWW}(A_{r,r'})) \leq \bigg( \frac{r}{r'} \bigg)^{\beta} \pi(r,r').
\end{equation}

\item A-priori estimate on $6$ arms:
\begin{equation} \label{eq:6_arm}
\pi_6(r,r') := \pi_{BWWBWW}(r,r') = \bbP_{p_c}(\cA_{BWWBWW}(A_{r,r'})) \leq \bigg( \frac{r}{r'} \bigg)^{2+ \beta}.
\end{equation}
\end{enumerate}
Note that the exponents corresponding to $\pi_3$ and $\pi_6$ are known ($\alpha_3 = \frac{2}{3}$ and $\alpha_6 = \frac{35}{12}$, respectively) but they will not be needed. The upper bounds above can be established by soft arguments (essentially the Russo-Seymour-Welsh bounds, that we mention in Appendix~\ref{sec:app_perc}). As explained in \cite{BN11}, the $6$-arm event plays an important role to show that discrete events converge to continuous ones. Indeed, the fact that this exponent is strictly larger than $2$ implies that two macroscopic parts of loops (either of the same loop, or of two distinct ones) cannot come close to each other without touching in the discrete level.

\subsection{Derivation of the discrete exponent} \label{subsec:bb_exp}

In this subsection, we derive the backbone exponent, corresponding to $\pi(0,n)$, from estimates in the continuum which are proved in the remainder of the paper. Our starting point is a standard result from graph theory, Menger's theorem. In any annulus $A_{r_1,r_2}$, $0 \leq r_1 < r_2$, this result implies that the complement of $\cA_{BB}(A_{r_1,r_2})$ is exactly the event that there exists a ``quasi-white'' circuit surrounding $B_{r_1}$ in $A_{r_1,r_2}$, whose vertices are all white, except at most one (which is then black). This leads to the two cases depicted in Figure~\ref{fig:loops}, center and right. Consider the white cluster containing this circuit: its external edge boundary is a discrete loop that converges to a ${\rm CLE}_6$ loop \cite{CN06}, and its external frontier is a black circuit surrounding $B_{r_1}$. Conversely, the existence of such a black circuit does not guarantee the occurrence of $\cA_{BB}(A_{r_1,r_2})^c$. Indeed, \emph{fjords} (as in \cite{ADA99}) are attached to the external frontier, and black arms can take advantage of them to reach $B_{r_1}$ (see Figure~\ref{fig:bubbles}). However, we will argue that typically, such fjords do not come too close to $B_{r_1}$, so that the frontier still provides a good approximation. More precisely, we will use that the cost of such fjords (intersecting $B_{r_1}$ without covering it completely) can be controlled by the exponent $\alpha_3$, which is $> \xi$.

\begin{figure}
	\centering
	\subfigure{\includegraphics[width=.3\textwidth]{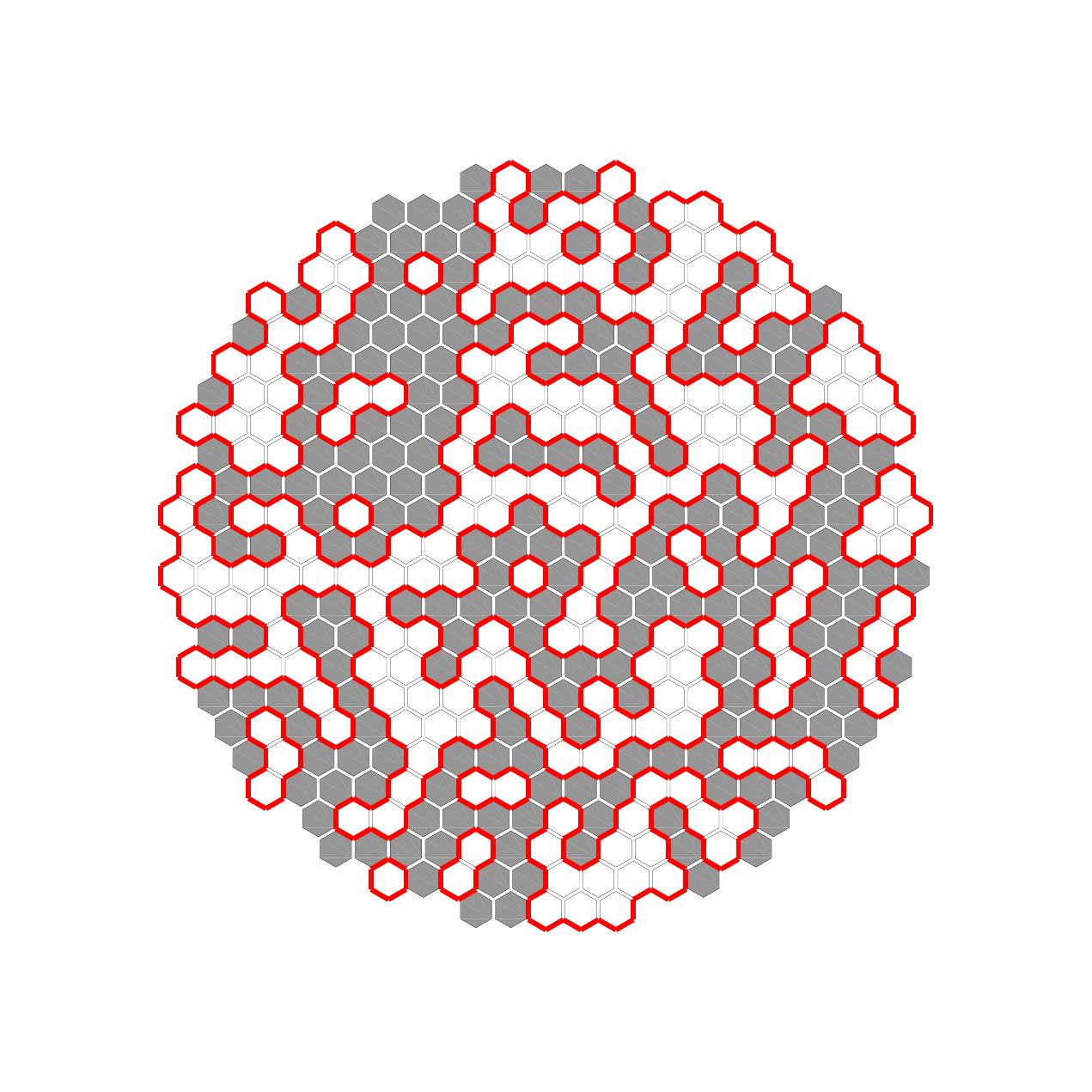}}
	\hspace{0.5cm}
	\subfigure{\includegraphics[width=.3\textwidth]{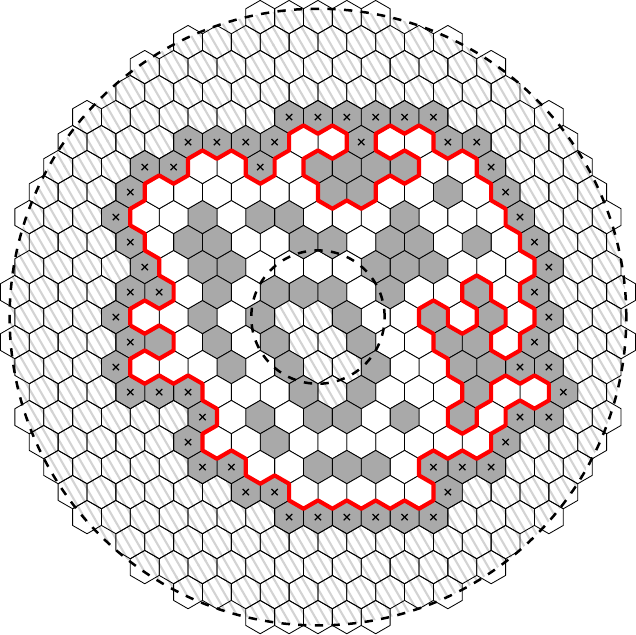}}
		\hspace{0.5cm}
	\subfigure{\includegraphics[width=.3\textwidth]{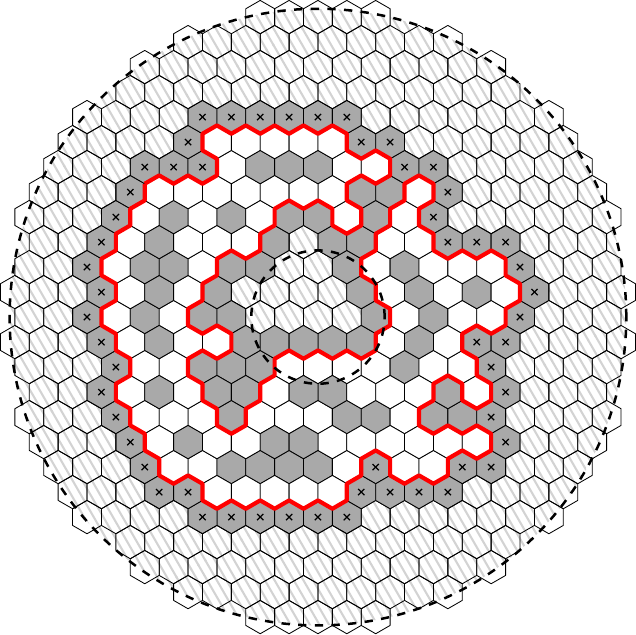}}
	\caption{Let $0 \leq r_1 < r_2$. Given a percolation configuration in $B_{r_2}$, we consider interfaces, shown in red: they are self-avoiding paths, on the dual hexagonal lattice, separating black and white clusters. \emph{Left:} If we consider black \emph{boundary conditions}, which amounts to adding an extra layer of black sites (along $\dout B_{r_2}$), the set of interfaces is a collection of circuits. \emph{Center:} There exists a white circuit surrounding the smaller ball $B_{r_1}$, i.e. there is no black arm in $A_{r_1,r_2}$. We consider the white connected component of this circuit, and its external edge boundary, which is a dual circuit. The external frontier is a black circuit, marked with crosses. \emph{Right:} It there is a quasi-white circuit, containing exactly one black vertex, then there may exist one arm in $A_{r_1,r_2}$ (but not two disjoint arms). Again, we consider the white cluster of this circuit, its external edge boundary, and its external frontier.}
	\label{fig:loops}
\end{figure}

\begin{figure}[t]
\centering
\includegraphics[width=0.45\textwidth]{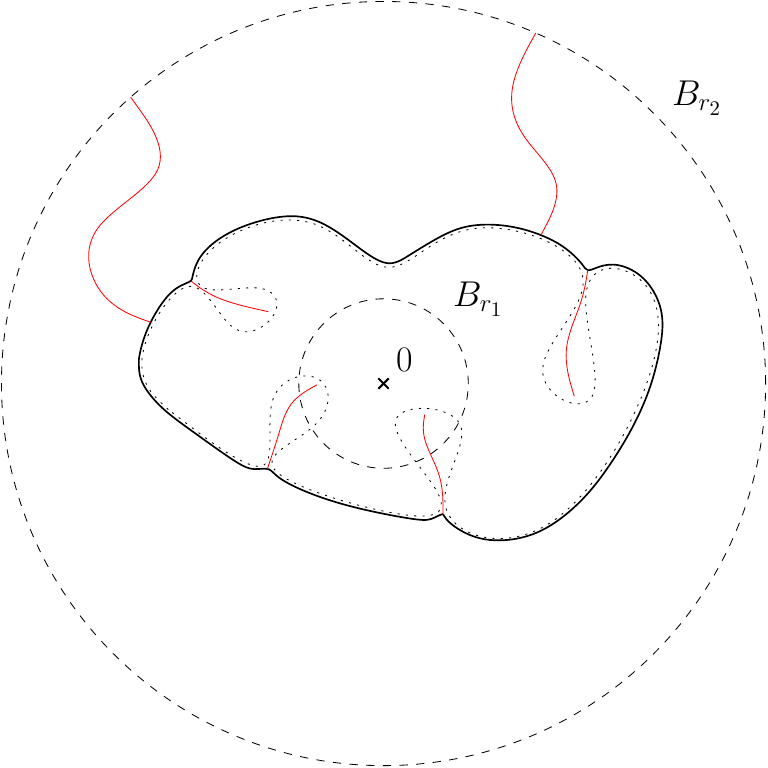}
\caption{Black paths emanating from the external frontier of a loop may ensure the occurrence of the monochromatic two-arm event in the annulus $A_{r_1,r_2}$, through passages of width $1$, even in the case when that frontier surrounds $B_{r_1}$.}
\label{fig:bubbles}
\end{figure}

Throughout this section, we consider (nested) ${\rm CLE}_6$ in the closed unit disk $\bbD$, which is a collection of continuous, non-simple loops. These loops have multiple points, and touch other loops, but no crossing can occur -- between two different loops, or within the same loop. Among these loops, we let $\gamma$ be the outermost loop (so in the non-nested ${\rm CLE}_6$) disconnecting $0$ from $\partial \bbD$, and we let $\ep{\gamma}$ be the outer boundary of $\gamma$, i.e. corresponding to the boundary of the connected component of $\bbC \setminus \gamma$ intersecting $\partial \bbD$. It is a simple loop, which may touch $\partial \bbD$ (since $\gamma$ itself may intersect $\partial \bbD$). We introduce, for all $r \in (0,1)$, the event
$$E_r := \{ \ep{\gamma} \cap \cB_r \neq \emptyset\}.$$
Note that $E_r$ does not correspond directly, in the continuum, to the existence of two disjoint black arms for the discrete model, and we will need to compare these two events. Let
\begin{equation}
f(r) := \bbP(E_r) \quad  \textrm{for }   r \in (0,1).
\end{equation}
In Section~\ref{subsec:link_bubble}, we will prove purely from the continuum the following proposition, which relates the asymptotic behavior of $f(r)$ as $r \searrow 0$  to the ${\rm SLE}_6$ bubble.
\begin{proposition} \label{prop:est_bb}
Recall the exponent $\xi(\kappa)$ defined by \eqref{eq:exponent-def}. We have
\begin{equation} \label{eq:bb_cont}
\lim_{r \searrow 0} \frac{\log f(r)}{\log r} = \xi(6).
\end{equation}
\end{proposition}

Proposition~\ref{prop:est_bb} means in particular that $\lim_{r \searrow 0} \frac{\log f(r)}{\log r}$ exists.  In the rest of this subsection, assuming Proposition~\ref{prop:est_bb}, we show Theorem~\ref{thm:bb_exp}, which states that this limit equals the discrete backbone exponent $\xi$ describing $\pi(0,n)$.
In order to do so, we first need to introduce analogous events for Bernoulli percolation. For any $r_2 \geq 1$, consider the lattice ball $B_{r_2}$. Similarly to the continuous situation above, we can introduce, in this discrete setting, the set of all edge boundaries of clusters, with suitable boundary conditions so that we get a collection of dual loops. This collection of discrete loops converges, in a sense made precise below, to nested ${\rm CLE}_6$. We can then consider the outermost such loop $\Gamma$, which is a boundary of a white cluster and such that its external frontier $\ep{\Gamma}$ (which is thus black) disconnects $0$ from $\dout B_{r_2}$. Note that this loop may not exist, but this event has a vanishing probability as $r_2 \to \infty$, so we do not need to worry about it in the upcoming reasonings.

Still in the ball $B_{r_2}$, for $0 < r_1 < r_2$, we introduce the discrete event corresponding to $E_r$, that is,
$$\tilde{E}_{r_1,r_2} := \{ \ep{\Gamma} \cap (B_{r_1} \cup (\dout B_{r_1})) \neq \emptyset\},$$
and we let
$$\tilde{f}(r_1,r_2) := \bbP_{p_c}(E_{r_1,r_2}).$$
From the connection to ${\rm CLE}_6$ in the scaling limit, it is of course natural to expect $\tilde{f}(r_1,r_2) \simeq f(\frac{r_1}{r_2})$. We will rely on the following intermediate result.
\begin{lemma} \label{lem:scaling_limit}
Let $R > 1$. For all $\eps \in (0,1)$,
\begin{equation} \label{eq:limit_proba}
\liminf_{n \to \infty} \tilde{f}(n,Rn) \geq f(R^{-1-\eps}) \quad \text{and} \quad \limsup_{n \to \infty} \tilde{f}(n,Rn) \leq f(R^{-1+\eps}).
\end{equation}
\end{lemma}
In principle, it should be true that $\lim_{n \to \infty} \tilde{f}(n,Rn) = f(R^{-1})$. However, proving it would require a little bit more care, e.g. to show that $f$ is continuous. Here, we are content with this weaker version, which is good enough to derive Theorem~\ref{thm:bb_exp}.

In order to discuss convergence of collections of loops, we use the setting of \cite{AB99}. Curves in $\bbD$ are defined as continuous functions $\gamma : [0,1] \to \bbD$, up to strictly monotone reparametrization. That is, we consider equivalence classes for the relation: $\gamma \sim \gamma'$ if and only if there exists a bijection $\varphi: [0,1] \to [0,1]$, either strictly increasing or strictly decreasing, such that $\gamma = \gamma' \circ \varphi$. We then consider uniform convergence up to reparametrization: the distance between two curves $\gamma$ and $\gamma'$ is defined as
$$d(\gamma,\gamma') = \inf_{\varphi} \sup_{t \in [0,1]} \|\gamma(t) - \gamma'(\varphi(t))\|,$$
where the infimum is taken over all strictly monotone bijections $\varphi$ from $[0,1]$ onto itself. Sets of curves are then simply equipped with the Hausdorff distance induced by the distance $d$.

\begin{proof}[Proof of Lemma~\ref{lem:scaling_limit}]
In the whole proof, we consider Bernoulli percolation in $B_{R n}$, for large $n$. We can choose a suitable coupling to ensure that $\frac{1}{Rn} \Gamma$ converges to the continuous loop $\gamma$ (from Skorokhod's theorem). Note that in the following, we are interested in the outermost loop $\gamma$ in ${\rm CLE}_6$ such that $\ep{\gamma}$ disconnects $0$ from $\partial \bbD$, in particular whether the loop $\ep{\gamma}$ intersects $\cB_{R^{-1}}$. We observe that if $\diam(\gamma) \leq R^{-1}$, then $\gamma$ would need to be included in $\cB_{R^{-1}}$, so in particular $\ep{\gamma}$ would intersect that ball. Hence, only ``macroscopic'' loops, with a diameter larger than $R^{-1}$, need to be examined. We can then use that for any $\eps > 0$, a.s. there are only finitely loops in ${\rm CLE}_6$ with a diameter $> \eps$.

Let $\eta > 0$. We claim that for all $n$ large enough, we have the following: with probability at least $1 - \eta$, the rescaled external frontier $\frac{1}{Rn} \ep{\Gamma}$ lies within distance $\delta := \frac{1}{2}(R^{-1} - R^{-1-\eps})$ ($\leq \frac{1}{2}(R^{-1+\eps} - R^{-1})$) from $\ep{\gamma}$. Then if we denote by $F$ the corresponding event, which satisfies $\bbP(F^c) \leq \eta$, we have the inclusion of events
$$F \cap \{ \ep{\gamma} \cap \cB_{R^{-1-\eps}} \neq \emptyset \} \subseteq F \cap \{ \ep{\Gamma} \cap (B_n \cup (\dout B_n)) \neq \emptyset \} \subseteq F \cap \{ \ep{\gamma} \cap \cB_{R^{-1+\eps}} \neq \emptyset \}.$$
This would then imply immediately the inequalities
$$f(R^{-1-\eps}) - \eta \leq \tilde{f}(n,Rn) \leq f(R^{-1+\eps}) + \eta,$$
for all $n$ large enough, and so the desired conclusion \eqref{eq:limit_proba}.

We now prove the claim. If $\gamma$ was a simple curve, this would be obvious. It is clear that $\tilde{\Gamma} = \frac{1}{Rn} \ep{\Gamma}$ cannot wander too much from $\ep{\gamma}$ to the outside (since $\ep{\gamma}$ is known to be a simple curve). However, $\gamma$ has double points and we have to take care of these ``pinchings'', to show that $\tilde{\Gamma}$ cannot penetrate too deeply by making use of them (which would well be conceivable, if passages of width $2, 3, \ldots$ existed on the discrete picture). This issue can be handled using properties of Bernoulli percolation (similarly to Section~4 of \cite{BN11}). For this, we use the a-priori bound \eqref{eq:6_arm} to ensure that in the discrete model, it is not possible for microscopic bottlenecks to arise. More specifically, on a suitable event with probability at least $1 - \frac{\eta}{10}$, there is no fjord which has both a diameter $\geq \frac{\delta}{4} R n$ and a passage width $\geq 2$.

Let $\bar{\delta} > 0$. First, we can ensure that with probability at least $1 - \frac{\eta}{2}$, $\frac{1}{Rn} \Gamma$ is within a distance $\bar{\delta}$ of $\gamma$ (we explain later how to choose $\bar{\delta}$). We denote by $F_1$ the corresponding event. Now, let
$$F_2 := \{ \exists z \in \cB_{Rn} \: : \: \cA_6(A_{\bar{\delta} R n, \frac{\delta}{4} R n}(z)) \text{ occurs}\},$$
where for simplicity, we write $\cA_6 = \cA_{BWWBWW}$. We can restrict to $z$ on the lattice $\bar{\delta} n \bbZ^2$, i.e. consider the event
$$\tilde{F}_2 := \{ \exists z \in \cB_{Rn} \cap (\bar{\delta} n \bbZ^2) \: : \: \cA_6(A_{2 \bar{\delta} R n, \frac{\delta}{8} R n}(z)) \text{ occurs}\},$$
and clearly, $\tilde{F}_2 \supseteq F_2$. Here, $R$ is given, and for some universal constant $c$, $|\cB_{Rn} \cap (\bar{\delta} n \bbZ^2)| \leq c R^2 (\bar{\delta})^{-2}$. We deduce from the union bound and \eqref{eq:6_arm} that
$$\bbP(\tilde{F}_2) \leq c R^2 (\bar{\delta})^{-2} \cdot \bigg( \frac{2 \bar{\delta} R n}{\frac{\delta}{8} R n} \bigg)^{2 + \beta} = c' (\bar{\delta})^{\beta},$$
for some $c' = c'(R,\delta)$. Hence, we can ensure that $\bbP(\tilde{F}_2) \leq \frac{\eta}{10}$, and thus $\bbP(F_2) \leq \frac{\eta}{10}$, by choosing $\bar{\delta}$ small enough. Then on the event $F_1 \cap F_2^c$, we are sure that any bubble with diameter at least $\frac{\delta}{2} R n$ in the continuum leads to a pinching in the discrete as well. Indeed, if $\ep{\Gamma}$ created a passage of width $\geq 2$, then we would get $6$ arms around the pinching point $z$, from a distance $\bar{\delta} Rn$ to distance $\frac{\delta}{4} R n$.

This completes the proof of the claim, up to some technical issue: the loop $\Gamma$ may touch the boundary. In this situation, instead of six arms, we would only get three arms locally. Nevertheless, we can still conclude in this case that $\frac{1}{Rn} \ep{\Gamma}$ and $\ep{\gamma}$ are close, by using that these three arms lie in a half-plane and that the corresponding exponent is $>1$.
\end{proof}

We now establish the following result, which immediately implies Theorems~\ref{thm:kappa=6} and~\ref{thm:solution}.

\begin{theorem} \label{thm:bb_exp}
Recall $\xi(6)$ from Proposition~\ref{prop:est_bb}. We have
\begin{equation}
\lim_{n \to \infty} \frac{\log \pi(0,n)}{\log n} = - \xi(6).
\end{equation}
Hence, the backbone exponent $\xi$ is equal to $\xi(6)$.
\end{theorem}

\begin{proof}
Consider an arbitrary $\eps > 0$. From \eqref{eq:bb_cont}, we have, for some $r_0 > 0$: for all $r \leq r_0$,
\begin{equation} \label{eq:bb_exp_pf1}
r^{\xi(6) + \frac{\eps}{10}} \leq f(r) \leq r^{\xi(6) - \frac{\eps}{10}}.
\end{equation}
In addition, we may assume that $r_0$ is chosen so small that for the (universal) constants $c_1$ and $c_2$ appearing in \eqref{eq:qm},
\begin{equation} \label{eq:bb_exp_pf2}
c_1 \geq r_0^{\frac{\eps}{10}} \quad \text{and} \quad c_2 \leq r_0^{-\frac{\eps}{10}}.
\end{equation}
We now fix a large integer $i \geq 2$, that we explain how to choose later. Let $R = r_0^{-1}$. Applying Lemma~\ref{lem:scaling_limit} twice, we obtain immediately that for all $n \geq n_0$, we have: for $j=0,1$,
\begin{equation}
R^{(-\xi(6) - \frac{2 \eps}{5})(i-j)} \leq (R^{i-j})^{-\frac{\eps}{10}} f((r_0^{i-j})^{1+\frac{\eps}{5}}) \leq \tilde{f}(R^j n,R^i n) \leq (R^{i-j})^{\frac{\eps}{10}} f((r_0^{i-j})^{1-\frac{\eps}{5}}) \leq R^{(-\xi(6) + \frac{2 \eps}{5})(i-j)}.
\end{equation}
Clearly,
\begin{equation} \label{eq:bb_exp_pf3}
\pi(n,R^i n) \geq \tilde{f}(n, R^i n) \geq R^{(-\xi(6) - \frac{2 \eps}{5})i}.
\end{equation}
In the other direction, we can consider the configuration in $B_{R^i n}$ and decompose the event $\cA_{BB}(A_{n,R^i n})$ according to the innermost annulus of the form $A_{R^{j-1} n, R^j n}$, $1 \leq j \leq i$, that $\ep{\Gamma}$ intersects. Note that $\ep{\Gamma}$ is a black circuit disconnecting $0$ from $\dout B_{R^i n}$.

In the case when $j \geq 2$, we make the following observation, assuming the existence of two black arms from $\din B_{R^i n}$ to $\dout B_n$.
\begin{itemize}
\item In particular, the event $\cA_{BB}(A_{R^{j-1} n,R^i n})$ occurs, and in the same annulus $A_{R^{j-1} n,R^i n}$, there exists a black circuit which surrounds $B_{R^{j-1} n}$ and intersects $B_{R^j n}$. We denote this latter event by $\cC(A_{R^{j-1} n,R^i n};B_{R^j n})$.

\item In addition, the event $\cA_{BWW}(A_{n,R^{j-1} n})$ occurs.
\end{itemize}

Hence, the union bound gives
\begin{align*}
\pi(n,R^i n) & = \bbP_{p_c}(\cA_{BB}(A_{n,R^i n}))\\
& \leq \bbP_{p_c}(\cA_{BWW}(A_{n,R^{i-1} n}))+ \tilde{f}(R n, R^i n)\\
& \hspace{1cm} + \sum_{j=2}^{i-1} \bbP_{p_c}(\cA_{BWW}(A_{n,R^{j-1} n}) \cap \cA_{BB}(A_{R^{j-1} n,R^i n}) \cap \cC(A_{R^{j-1} n,R^i n};B_{R^j n})).
\end{align*}
Clearly, the two events $\cA_{BWW}(A_{n,R^{j-1} n})$ and $\cA_{BB}(A_{R^{j-1} n,R^i n}) \cap \cC(A_{R^{j-1} n,R^i n};B_{R^j n})$ are independent, since they involve disjoint subsets of vertices. Furthermore, $\cA_{BB}(A_{R^{j-1} n,R^i n}) \cap \cC(A_{R^{j-1} n,R^i n};B_{R^j n})$ implies that $\ep{\Gamma} \cap B_{R^j n} \neq \emptyset$. We deduce
\begin{align*}
\pi(n,R^i n) & \leq \pi_3(n,R^{i-1} n) + \sum_{j=2}^{i-1} \pi_3(n,R^{j-1} n) \cdot \tilde{f}(R^j n, R^i n) + \tilde{f}(R n, R^i n)\\
& \leq (R^{-(i-1)})^{\beta} \pi(n,R^{i-1} n) + \sum_{j=2}^{i-1} (R^{-(j-1)})^{\beta} \pi(n,R^{j-1} n) \cdot \pi(R^j n, R^i n) + \tilde{f}(R n, R^i n)\\
& \leq c_3 R^{- \beta} \pi(n,R^i n) + \tilde{f}(R n, R^i n),
\end{align*}
for some universal constant $c_3$. Here, we used \eqref{eq:3_arm} and \eqref{eq:bb_exp_pf3} for the second inequality, and then \eqref{eq:qm} for the third inequality. Hence, it suffices to choose, in the beginning, $i$ sufficiently large so that $\frac{i-1}{i} \geq 1 - \frac{\eps}{10}$. In addition, we can assume that $(1 - c_3 R^{- \beta})^{-1} \leq R^{\frac{\eps}{10}}$, which yields
\begin{equation} \label{eq:bb_exp_pf4}
\pi(n,R^i n) \leq R^{(-\xi(6) + \frac{4 \eps}{5})i}.
\end{equation}
We can now conclude the proof by applying repeatedly Lemma~\ref{lem:qm}, combined with \eqref{eq:bb_exp_pf2}, \eqref{eq:bb_exp_pf3} and \eqref{eq:bb_exp_pf4}.
\end{proof}

\subsection{Link with ${\rm SLE}_\kappa$ bubbles} \label{subsec:link_bubble}
Throughout this section, we assume that $\kappa\in (4,8)$, and sometimes omit $\kappa$ in the notations. We also extend certain notations in Section~\ref{subsec:bb_exp} to all $\kappa\in(4,8)$.
Hence, let $\gamma$ be the outermost loop in CLE$_\kappa$ that disconnects $0$ from $\infty$, and let $\ep\gamma$ be the outer boundary of $\gamma$. For $r\in (0,1)$, we still consider $E_r := \{ \ep{\gamma} \cap \cB_r \neq \emptyset\}$, and denote $f(r)=\bbP(E_r)$. 

The goal of this section is to prove the following proposition, which contains Proposition~\ref{prop:est_bb} as the special case $\kappa=6$. The proof  uses Theorem~\ref{thm:moment} as an input, which will be proved in later sections.
\begin{proposition}\label{prop:backbone}
For $\kappa\in(4,8)$, let $\xi(\kappa)$ be the exponent defined by \eqref{eq:exponent-def}, then
\begin{equation} \label{eq:bb_cont1}
\lim_{r \searrow 0} \frac{\log f(r)}{\log r} = \xi(\kappa).
\end{equation}
\end{proposition}
Let us recall some background leading to the definition of $\xi(\kappa)$.
The ${\rm SLE}_\kappa$ bubble measure in the upper half plane $\bbH$ rooted at 0  is defined by
\begin{align}\label{eq:sle_bubble}
\mu_\kappa: =c(\kappa) \lim_{\eps\to 0} \eps^{-1+8/\kappa}\mu_\kappa^\eps,
\end{align}
where $\mu_\kappa^\eps$ is the probability law of a chordal ${\rm SLE}_\kappa$ from $0$ to $-\eps$ in $\bbH$, and $c(\kappa)>0$ is a renormalizing constant. We choose the constant $c(\kappa)$ in the following way. Suppose that $e$ is a non-simple loop rooted at 0. Let $\ep e$ be the outer boundary of $e$, and $D_{\ep e}$ be the domain enclosed by $\ep e$. 
We chose $c(\kappa)$ such that $\mu_\kappa[ i \in D_{\ep e}] = 1$. We define $\nu_\kappa$ as the restriction of $\mu_\kappa$ to the event $\{ i \not \in D_{\ep e} \}$. If $i \not \in D_{\ep e}$, we define $D_i$ as the simply connected component of $\bbH \setminus e$ that contains $i$. Let $\psi$ be a conformal map from $(\mathbb{H},i)$ to $(D_i, i)$. It follows that $|\psi'(i)|<1$, and this quantity does not depend on the choice of $\psi$. Equation \eqref{eq:exponent-def} defines $\xi(\kappa)$ as the unique solution in the interval $(0, 1-2/\kappa)$ to 
\begin{equation}
\label{eq:backbone}
\nu_\kappa[|\psi'(i)|^{-\xi(\kappa)} -1  ]=1\,.
\end{equation}

\medbreak

\noindent\textbf{Some background on $\mathrm{CLE}_\kappa$.}
Let us now give more background on CLE$_\kappa$, which is first constructed by Sheffield in \cite{MR2494457} using the \emph{continuum exploration tree}. The definition in \cite{MR2494457} allows one to choose a parameter $\beta$ which consists of a random swapping of the orientations of the loops, but we will only look at the simplest case $\beta=1$, where all the loops are traced in the  counterclockwise direction.
For $\kappa\in(4,8)$, the SLE$_\kappa(\kappa-6)$ process (see Appendix~\ref{app:descrip} for a definition) is  target-invariant \cite{MR2188260,MR2494457}, and proved to be a continuous curve in  \cite{MR3477777}. Let $a_1, a_2, \ldots$ be a countable dense set of points in $\bbH$. We run a branching SLE$_\kappa(\kappa-6)$ process in $\bbH$ from $0$ with marked point $0^-$, targeting all $a_1, a_2, \ldots$ simultaneously. For example, the SLE$_\kappa(\kappa-6)$ curve targeting $a_i$ and $a_j$ is the same until the first time that it disconnects $a_i$ from $a_j$ in $\bbH$, and then it branches into two curves which continue in the two connected components containing $a_1$ and $a_2$ respectively. 
For each $a_j\in \bbH$, the unique loop $\ell$ in the non-nested CLE$_\kappa$ which encircles $a_j$ (i.e., winds around $a_j$ in the counterclockwise direction) is constructed as follows (see Figure~\ref{fig:cle} Left):
\begin{itemize}
\item Let $\eta$ be a radial SLE$_\kappa(\kappa-6)$  in $\bbH$ from $0$ with marked point $0^-$, targeting $a_j$. Let $\sigma$ be the first time that $\eta$ makes a closed loop around $a_j$ in the counterclockwise direction. 
\item For $t\ge 0$, let $z_t$ be the marked point of $\eta$ at time $t$. More precisely, $z_t$ is a prime end in the connected component of $\bbH \backslash \eta([0,t])$ that contains $a_j$. Let $z_{\sigma^-}$ be the marked point of $\eta$ just before time $\sigma$. Let $t_0$ be the last time before $\sigma$ that $\eta$ visits $z_{\sigma^-}$.
\item Let $\wt\eta$ be a chordal SLE$_\kappa$ from $\eta(\sigma)$ to $z_{\sigma^-}$, in the connected component of $\bbH\setminus \eta([0,\sigma])$ which contains $z_{\sigma^-}$. The concatenation of $\eta([t_0, \sigma])$ and $\wt\eta$ constructs the loop $\ell$.
\end{itemize}
This fully describes the construction of the non-nested CLE$_\kappa$. For each $a_i$, if we continue $\eta$ beyond the stopping time $\sigma$, then we can also construct the nested  CLE$_\kappa$. However, we are only interested in the non-nested CLE$_\kappa$ in this article.
\begin{figure}[t]
\centering
\includegraphics[width=.9\textwidth]{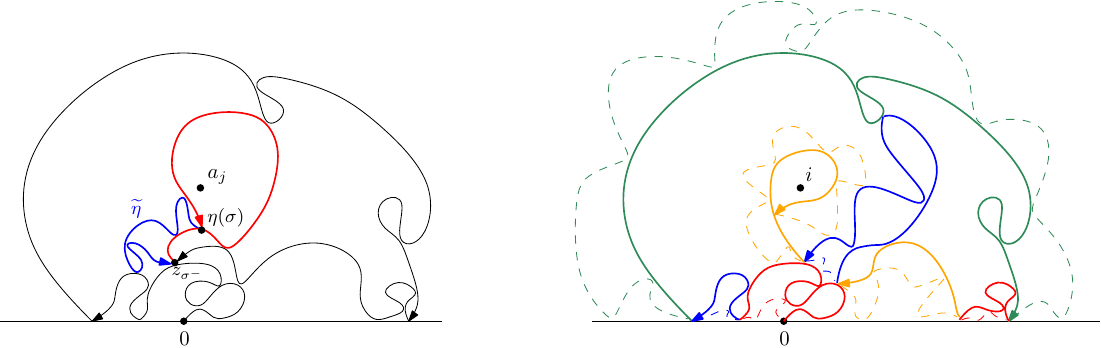}
\caption{\textbf{Left:} Construction of CLE$_\kappa$ out of a branching SLE$_\kappa(\kappa-6)$ process. The black curve followed by the red curve is a radial SLE$_\kappa(\kappa-6)$ from $0$, with marked point $0^-$, targeting $a_j$, stopped at the first time $\sigma$ that $\eta$ makes a counterclockwise loop around $a_j$. The concatenation of the red and blue curves form the unique loop $\ell$ in the non-nested CLE$_\kappa$ which surrounds $a_j$. \textbf{Right:} The radial SLE$_\kappa(\kappa-6)$ process $\eta$ from $0$, with marked point $0^-$, targeting $i$ is drawn in plain curve. The successive loops in the CLE$_\kappa$ that it traces are drawn in the order red, blue, green, red, yellow, blue, yellow. The parts of the loops which are not traced by $\eta$ are drawn in dashed lines.}
\label{fig:cle}
\end{figure}

\medbreak

Let us now relate this non-nested CLE$_\kappa$ to a Poisson point process of SLE$_\kappa$ bubbles, in Description \ref{descrip} below.
This description is analogous to the one given in \cite[Section 6.1]{MR3708206}, where the loops in the CLE$_\kappa$ were explored in a certain order, along a \emph{chordal} SLE$_\kappa(\kappa-6)$ process targeting $\infty$, inducing a Poisson point process of SLE$_\kappa$ bubbles. For our purpose, we will explore along a \emph{radial} SLE$_\kappa(\kappa-6)$ process targeting $i$.
We will  provide a proof of Description \ref{descrip} in Appendix~\ref{app:descrip}.

\begin{descrip}\label{descrip}
\textbf{$\mathrm{CLE}_\kappa$ and $\mathrm{SLE}_\kappa$ bubbles.} Let $\Gamma$ be a non-nested CLE$_\kappa$ in $\bbH$, constructed using a branching SLE$_\kappa(\kappa-6)$ exploration tree. 
Let $\eta$ be a radial SLE$_\kappa(\kappa-6)$  in $\bbH$ from $0$ with marked point $0^-$, targeting $i$.
Note that $\eta$ forms one branch of the exploration tree, and traces a sequence of loops in $\Gamma$ in a certain order. 
To be precise, $\eta$ only traces a portion of each loop that it visits (the remaining parts of the loops are traced by other branches in the exploration tree).
We only look at $\eta$ up to some stopping time that we will specify later. This stopping time will be chosen in such a way that $\eta$ has only explored the non-nested CLE$_\kappa$ (i.e., $\Gamma$) until that time.

There is a time-indexed Poisson point process with intensity $\mu_\kappa$, denoted by $\{e_s, s\in I\}$, where $I\subset \bbR^+$ is the set of times $s$ at which an SLE$_\kappa$ bubble $e_s$ appears. 
Let $\tau>0$ be the first time that a bubble $e_\tau$ appears in the Poisson point process so that $\ep e_\tau$ encircles $i$. There is a one-to-one correspondence between the set of bubbles $\{e_s, s\in I, s\le \tau\}$ and the ordered set of loops that $\eta$ visits in $\Gamma$. We  denote the latter set by  $\{\ell_s, s\in I, s\le \tau \}$, so that the loop $\ell_s$ corresponds to the bubble $e_s$ for each $s\in I$. We remind that $\eta$ is stopped at a certain time that we have not specified yet, but it will be easier to define this stopping time once we describe the correspondence between $\{e_s, s\in I, s\le \tau\}$ and $\{\ell_s, s\in I, s\le \tau \}$. We also emphasize that $I$ is the set of time indices for the Poisson point process, and does not correspond to the time parametrization for the curve $\eta$ (by the usual radial capacity). We parametrize the set $\{\ell_s, s\in I, s\le \tau \}$ by $I$, only to indicate the order in which the loops are discovered by $\eta$.

The exact correspondence between $\{e_s, s\in I, s\le \tau\}$ and $\{\ell_s, s\in I, s\le \tau \}$ is described via the following conformal maps.
For each bubble $e$ rooted at $0$ such that $\ep e$ does not encircle $i$, let $D_{\ep e}$ and $D_i$ be defined as earlier. Let $a(e)$ be the starting point of the arc $\partial D_i \cap \bbH$, if we go along $\partial D_i$ in the clockwise direction. 
Let $\psi_{e}$ be the conformal map  from the connected component containing $i$ of $\bbH\setminus \ep e$ onto $\bbH$ with $\psi_{e}(i)=i$ and $\psi_{e}(a(e))=0$.
For all $u\in (0, \tau]$, let $\Psi_u$ be the composition of the conformal maps $\psi_{e_s}$ for all $s\in I$ with $s\le u$, in the order of their appearance. 
We also define $\Psi_{u^-}$ to be the composition of  $\psi_{e_s}$  for $s\in I, s<u$ in the order of their appearance.  For $u\in I$, we have $\Psi_{u} = \psi_{e_u} \circ \Psi_{u^{-}}.$  
The correspondence between $\{e_s, s\in I, s\le \tau\}$ and $\{\ell_s, s\in I, s\le \tau \}$ is determined by the fact that for each $s\in I$,  $\ell_s$ is equal to the image under $\Psi_{s^-}^{-1}$ of $e_s$. 
In particular, for all $u\in I$, $\Psi_u$ (resp.\ $\Psi_{u^-}$) is a conformal map from $O_u$ (resp.\ $O_{u^-}$) onto $\bbH$, where $O_u$ (resp.\ $O_{u^-}$) is  the connected component containing $i$ of  $\bbH\setminus \cup_{s\in I, s\le u} \ep\ell_s$ (resp.\ $\bbH\setminus \cup_{s\in I, s< u} \ep\ell_s$).
It is clear from this description that we should stop $\eta$ at the time that it has traced the loop $\ell_\tau$, or more precisely, at the time that $\eta$ has finished tracing the portion of $\ell_\tau$ that belongs to $\eta$.

This correspondence is also illustrated in Figure~\ref{fig:cle} Right. The curve $\eta$ makes ``bridges'' (see Definition~\ref{def:bridge}) that are marked by different colors. Each bridge corresponds to a loop $\ell_s$ (i.e., it is a portion of $\ell_s$). Conditionally on all the loops (or equivalently, bridges) that $\eta$ has already traced until a certain time, if we map the unexplored connected component containing $i$ conformally back onto $\bbH$, then the image of the ``next'' loop is an independent SLE$_\kappa$ bubble in $\bbH$. Note that it does not really make sense to talk about the next loop, because for any two instances $s_1, s_2\in I$, there are infinitely many intermediate ones in $(s_1,s_2)$. However, such a description still provides a meaningful intuition, and it could be made rigorous by first considering (the finite set of) loops with diameter greater than $\eps$, and then leting $\eps\to 0$.
\end{descrip}

Using the notations in Description \ref{descrip}, let us state the following lemma, which will be a key input in the proof of Proposition~\ref{prop:backbone}.

\begin{lemma}\label{lem:er1}
 Let $\xi(\kappa)$ be the solution to \eqref{eq:backbone}. We have 
 \begin{equation*}
\lim_{x\to\infty} \frac{\log \bbP(|\Psi_{\tau^-}'(i)|>x)}{\log x} = - \xi(\kappa).
\end{equation*}
\end{lemma}
Before proving Lemma~\ref{lem:er1}, let us explain how to complete the proof of Proposition~\ref{prop:backbone} assuming Lemma~\ref{lem:er1}. Another input in the proof of Proposition~\ref{prop:backbone} is the following lemma.
 \begin{lemma}\label{lem:er2}
 Let $\alpha_{3}(\kappa):=1-2/\kappa$. For all $\zeta<\alpha_{3}(\kappa)$, we have 
 \begin{align*}
\bbE [|\psi_{e_\tau}'(i)|^\zeta]<\infty.
\end{align*}
 \end{lemma}
Note that $e_\tau$ is distributed as an SLE$_\kappa$ bubble sampled according to $\mu_\kappa$ restricted to $i \in D_{\ep e_\tau}$ (this is a probability measure). 
The exponent $\alpha_{3}(\kappa)=1-2/\kappa$ is the exponent for an SLE$_{16/\kappa}$ to get $\eps$-close to an interior point, see \cite{MR1740371,MR2435854}. By duality, $\ep e_\tau$  locally looks like an SLE$_{16/\kappa}$ process (we prove the duality for $\ep e_\tau$ in Appendix~\ref{sec:app}). 
A rigorous proof  of Lemma~\ref{lem:er2} can possibly be made based on this, but in fact, a stronger result is obtained in \cite[Proposition 7.12]{Wu23}, where the exact formula for $\bbE [|\psi_{e_\tau}'(i)|^\zeta]$ is given for all $\zeta<\alpha_{3}(\kappa)$. A more transparent formula was derived in~\cite{ARSZ23}, which we recall in Remark~\ref{rmk:CR-formula} in Appendix~\ref{sec:app}.
For simplicity, we cite this as a proof of Lemma~\ref{lem:er2}.

Let us first prove Proposition~\ref{prop:backbone} assuming Lemma~\ref{lem:er1}. 
The proof  relies on the inequality $\alpha_{3}(\kappa)>\xi(\kappa)$, which follows from Theorem~\ref{thm:moment}, where the exact value of $\xi(\kappa)$ is computed.

\begin{proof}[Proof of Proposition~\ref{prop:backbone}]
Let $R$ be the Euclidean distance from $i$ to $\ep e_\tau$. The event $E_r$ is equal to $\{R \le r\}$. By Koebe $1/4$ theorem, we have 
\begin{align}\label{eq:koebe}
(2|\Psi_{\tau}'(i)|)^{-1} \le R \le 2 |\Psi_{\tau}'(i)|^{-1}, \qquad
\bbP(|\Psi_{\tau}'(i)| \ge 2/r) \le \bbP(E_r) \le  \bbP(|\Psi_{\tau}'(i)| \ge 1/(2 r)).
\end{align}
Therefore, in order to prove Proposition~\ref{prop:backbone}, it is enough to prove
\begin{align}\label{eq:enough_limit}
\lim_{x\to\infty}\frac{\log \bbP(|\Psi_\tau'(i)| > x)}{\log x} = -\xi(\kappa).
\end{align}
Note that $|\psi_{e_\tau}'(i)|>1$, hence $|\Psi_\tau'(i)|=|\psi_{e_\tau}'(i)|| \Psi_{\tau^-}'(i)|>|\Psi_{\tau^-}'(i)|$. Therefore, $\bbP(|\Psi_\tau'(i)| > x) \ge \bbP(|\Psi_{\tau^-}'(i)| > x)$. Thus the $\ge$ direction of \eqref{eq:enough_limit} follows from Lemma~\ref{lem:er1}. 

It remains to prove the $\le$ direction of \eqref{eq:enough_limit}.
By Lemma~\ref{lem:er1}, we know that for all $\eps>0$, there exists $x_0>0$, such that for all $x\ge x_0$,
\begin{align}\label{eq:>=x}
\bbP(|\Psi_{\tau^-}'(i)|>x) \le x^{-\xi(\kappa)+\eps}.
\end{align}
By Theorem~\ref{thm:moment}, we know $\alpha(\kappa)>\xi(\kappa)$. Let $\delta(\kappa)$ be a fixed number such that $\xi(\kappa)/\alpha(\kappa) < \delta(\kappa)<1$. 
For all $x\ge x_0^{1/(1-\delta(\kappa))}$, we have
\begin{align}
\notag
&\bbP(|\Psi_\tau'(i)| > x)\\
\notag
=&\bbP(|\psi_{e_\tau}'(i)| |\Psi_{\tau^-}'(i)|>x)
\le \bbP\big(|\psi_{e_\tau}'(i)| >x^{\delta(\kappa)}\big)+ \bbP\big(|\psi_{e_\tau}'(i)| \le x^{\delta(\kappa)}, \,|\Psi_{\tau^-}'(i)|>x |\psi_{e_\tau}'(i)|^{-1}\big)\\
\notag
\le &x^{-\xi(\kappa)} \bbE\big[|\psi_{e_\tau}'(i)|^{ \xi(\kappa)/\delta(\kappa)}\big] + \bbE\left[  \one_{|\psi_{e_\tau}'(i)| \le x^{\delta(\kappa)}} \bbP\left(|\Psi_{\tau^-}'(i)|>x |\psi_{e_\tau}'(i)|^{-1} \mid |\psi_{e_\tau}'(i)|\right)\right] \\
\label{eq:c2}
\le & c_4 x^{-\xi(\kappa)}  + \bbE\big[ (x |\psi_{e_\tau}'(i)|^{-1})^{-\xi(\kappa)+\eps}  \big] =c_4 x^{-\xi(\kappa)} + c_5 x^{-\xi(\kappa)+\eps} ,
\end{align}
where $c_4 =  \bbE\big[\psi_{e_\tau}'(i)^{\xi(\kappa)/\delta(\kappa)}\big]$ and $c_5 =  \bbE\big[  \psi_{e_\tau}'(i)^{ \xi(\kappa)-\eps}  \big]$ are finite, due to Lemma~\ref{lem:er2}, and because  $\xi(\kappa)/\delta(\kappa)<\alpha_{3}(\kappa)$ and $\xi(\kappa)-\eps<\alpha_{3}(\kappa)$.
The last inequality in \eqref{eq:c2} follows from the independence between $|\psi_{e_\tau}'(i)|$ and $|\Psi_{\tau^-}'(i)|$, \eqref{eq:>=x} and the fact that $x |\psi_{e_\tau}'(i)|^{-1} \ge x^{1-\delta(\kappa)}\ge x_0$ on the event $|\psi_{e_\tau}'(i)| \le x^{\delta(\kappa)}$. This implies that
\begin{align*}
\lim_{x\to\infty}\frac{\log \bbP(|\Psi_\tau'(i)| \ge x)}{\log x} \le -\xi(\kappa)+\eps.
\end{align*}
Letting $\eps$ tend to $0$, we obtain the $\le$ direction of \eqref{eq:enough_limit}. This completes the proof of \eqref{eq:enough_limit} and the proposition.
\end{proof}

Let us now prove Lemma~\ref{lem:er1}. For this purpose, we need to use a version of Tauberian theorem, due to \cite{MR2417689}.
Before stating the theorem, let us define the notion of the abscissa of convergence. Suppose that $\nu$ is a measure on $[0,\infty)$, then it is known (see e.g.\ \cite[p37]{MR0005923}) that there exists $s_0\in\bbR$ such that the integral
\begin{align*}
f(z)=\int_0^\infty e^{-z x} \nu(dx)
\end{align*}
converges for $\Re z> s_0$, diverges for $\Re z<s_0$ and has a singularity at $s_0$. The number $s_0$ is called the abscissa of convergence of $f$.
The following theorem gives the decay rate of the tail probability of a random variable, under certain conditions on its Laplace transform.
\begin{theorem}[Theorem 3, \cite{MR2417689}]\label{thm:tauberian}
Let $X$ be a non-negative random variable. Let 
$\varphi(s):= \bbE[\exp(-s X)]$ be the Laplace transform of $X$.
Suppose that $s_0 < 0$ is the abscissa of convergence of $\varphi$. If $s_0$ is a pole of $\varphi$, then we have
\begin{align*}
\lim_{x\to \infty} x^{-1} \log \bbP(X>x) =s_0.
\end{align*}
\end{theorem}

We are now ready to prove Lemma~\ref{lem:er1}.
\begin{proof}[Proof of Lemma~\ref{lem:er1}]
Note that
\begin{align*}
\log|\Psi_{\tau^-}'(i)|= \sum_{s\in I, s< \tau}\log |\psi_{e_s}'(i)|.
\end{align*}
We can view $\{e_s, s\in I, s < \tau\}$ as the bubbles from a Poisson point process with intensity $\nu_\kappa$ (instead of $\mu_\kappa$) stopped at an exponential random time $\tau$ with parameter $\mu_\kappa[ i \in D_{\ep e}] = 1$, where $\tau$ is independent from the Poisson point process. Then, Campbell's theorem for Poisson point process implies that
\begin{align*}
F(\lambda):=\bbE \left[  |\Psi_{\tau^{-}}'(i)|^{-\lambda} \right]= \bbE \left[ \exp (-\lambda \log |\Psi_{\tau^{-}}'(i)|)\right]=\bbE\left[ \exp\left(\nu_\kappa [e^{\lambda \log |\psi'(i)|} -1] \tau\right)\right] =\bbE\left[\exp\left(\beta(\lambda) \tau\right) \right],
\end{align*}
where
$\beta(\lambda):= \nu_\kappa [|\psi'(i)|^{\lambda} -1]$.
Let us first consider the function $F$ above as a real function. 
Since $\tau$ is an exponential random variable with parameter $1$, we have
$F(\lambda)<\infty$ if and only if $\beta(\lambda)<1$. 
Let $\xi(\kappa)$ be the solution to \eqref{eq:backbone}, so that $\beta(-\xi(\kappa))=1$.
By Theorem~\ref{thm:moment},  $\beta(\lambda)<1$ if and only if $\lambda> -\xi(\kappa)$.
This implies that $F(\lambda)<\infty$ if $\lambda > -\xi(\kappa)$ and $F(\lambda)=\infty$ if $\lambda\le -\xi(\kappa)$.
Therefore, $-\xi(\kappa)$ is the abscissa of convergence of $F$. 
For $\lambda\in \bbC$ with $\Re \lambda > -\xi(\kappa)$,  we further have
\begin{align}\label{eq:CR-good}
F(\lambda) = \frac{\kappa\cos(\frac{4\pi}{\kappa}) \Gamma(\frac{8}{\kappa}-1) \sin(\frac{4\pi  \theta}{\kappa})}{2   \Gamma(\frac{4(1-\theta)}{\kappa})\Gamma(\frac{4(1+\theta)}{\kappa}) }
   \Big( \theta \sin(\frac{8\pi}{\kappa})-\sin(\frac{8\pi  \theta}{\kappa}) \Big)^{-1}, 
\end{align}
where  $\theta^2= (1-\frac{\kappa}{4})^2-\frac{\kappa \lambda}{2}$.
Let $\hat G(\theta):=\left(\theta \sin(\frac{8\pi}{\kappa}) - \sin(\frac{8\pi \theta}{\kappa})\right)\left( \sin(\frac{4\pi  \theta}{\kappa})\right)^{-1}$.  Then $\hat G(\theta)$ is meromorphic in $\theta$ on the entire complex plane $\bbC$. 
Since $\hat G(\theta)=\hat G(-\theta)$,
the function \(G(\lambda):= \hat G(\sqrt{(1-\frac{\kappa}{4})^2-\frac{\kappa \lambda}{2}}) \)
is meromorphic in  $\lambda$ on $\bbC$ and has a zero at $\lambda=-\xi(\kappa)$. Therefore, $\lambda=-\xi(\kappa)$ is a pole of $G(\lambda)^{-1}$. On the other hand, the function 
\(\lambda \mapsto  \frac{\kappa\cos(\frac{4\pi}{\kappa}) \Gamma(\frac{8}{\kappa}-1)}
{2   \Gamma(\frac{4(1-\theta(\lambda))}{\kappa})\Gamma(\frac{4(1+\theta(\lambda))}{\kappa})  }
\)
 is analytic in an neighborhood of $\lambda=-\xi(\kappa)$, and is not zero at $\lambda=-\xi(\kappa)$. This proves that $\lambda=-\xi(\kappa)$ is a pole of $F(\lambda)$. Applying Theorem~\ref{thm:tauberian} to $\log  |\Psi_{\tau^{-}}'(i)|$, we conclude that
\begin{align*}
\lim_{x\to\infty} x^{-1}\log \bbP(|\Psi_{\tau^{-}}'(i)| >x) =-\xi(\kappa). 
\end{align*}
This completes the proof of the lemma.
\end{proof}

\subsection{Concluding remarks}
We conclude this section with a few remarks.
\begin{remark} [Comparison between exponents]\label{rmk:alpha}
The proof of Proposition~\ref{prop:backbone} relies on the inequality $\alpha_3(\kappa)>\xi(\kappa)$ which follows from Theorem~\ref{thm:moment}. The rigorous proof of Theorem~\ref{thm:moment} will be given in the following sections, based on a series of explicit computations which are rather convoluted, but the inequality just mentioned can also be informally seen in various ways. For example, in the discrete: $\alpha_{3}(\kappa)$ corresponds to the three-arm polychromatic exponent, which should be greater than the two-arm monochromatic exponent (e.g.\ from Russo-Seymour-Welsh-type bounds). In the continuum: we believe that $2-\xi(\kappa)$ should be the dimension of the thin gasket of CLE$_\kappa$ (in the same way that the dimension of the gasket of CLE$_\kappa$ is given by $2-\alpha_1(\kappa)$ \cite{MSWCLEgasket}, where $\alpha_1(\kappa)$ is the one-arm exponent). On the other hand, $2-\alpha_{3}(\kappa)$ is the dimension of the outer boundary of a single loop in CLE$_\kappa$. Intuitively, the latter should be strictly thinner than the former, which would yield $2-\alpha_{3}(\kappa)<2-\xi(\kappa)$.
\end{remark}

\begin{figure}
	\centering
	\subfigure{\includegraphics[width=.3\textwidth]{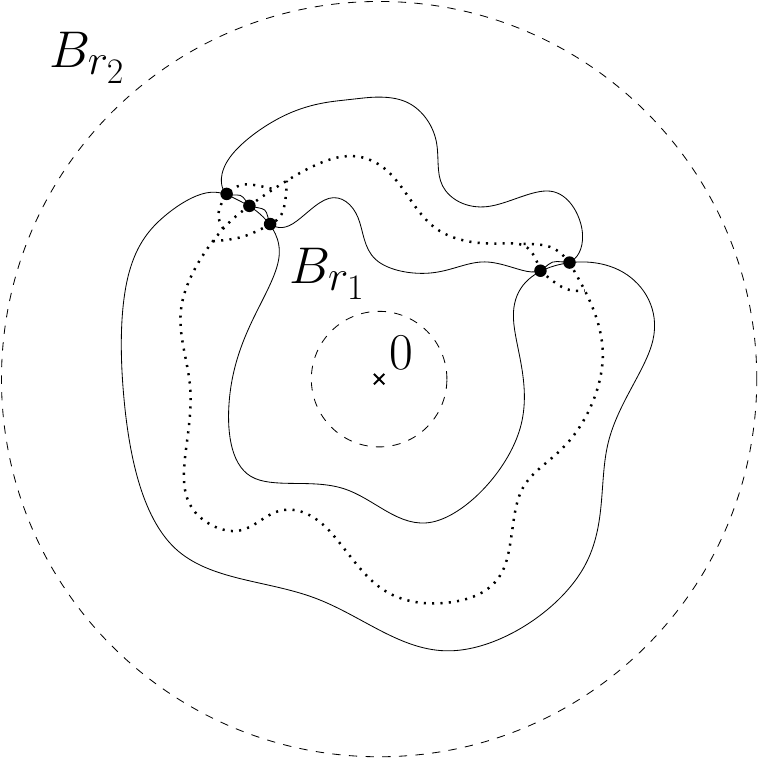}}
	\hspace{0.5cm}
	\subfigure{\includegraphics[width=.3\textwidth]{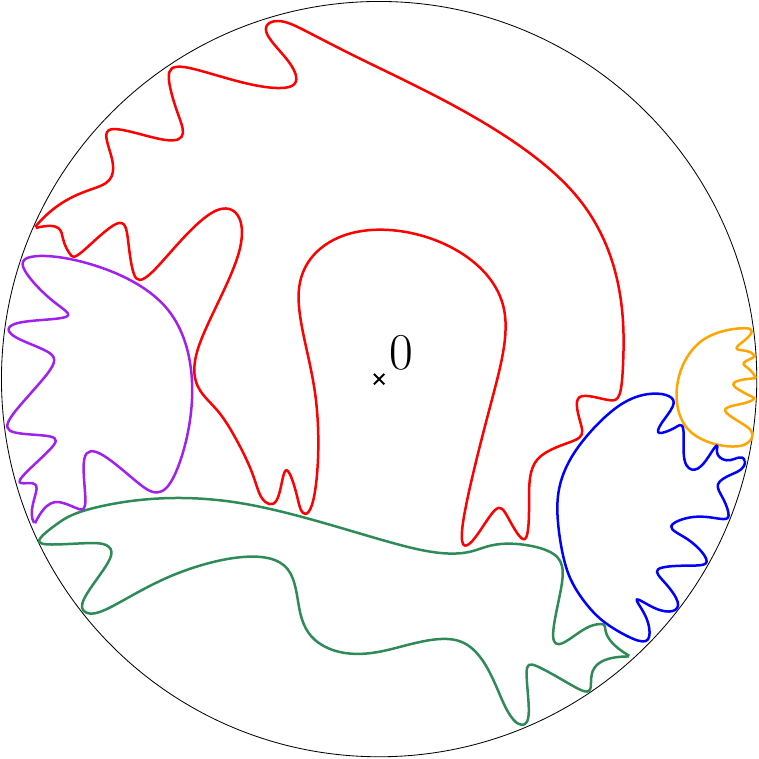}}
		\hspace{0.5cm}
	\subfigure{\includegraphics[width=.3\textwidth]{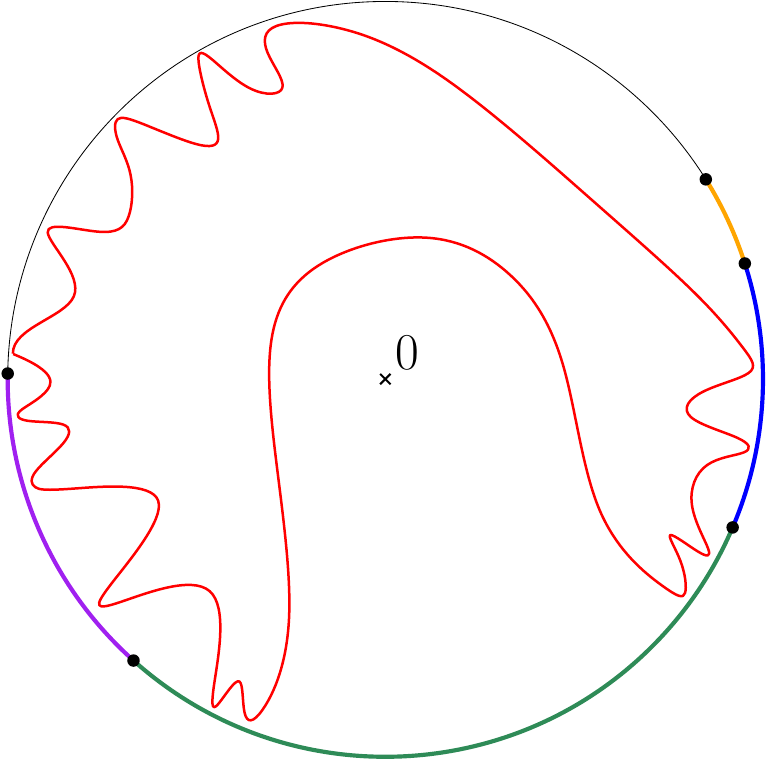}}
	\caption{Consider the monochromatic $3$-arm event in $A_{r_1,r_2}$, $0 \leq r_1 < r_2$. Its complement can be expressed in terms of the existence of a circuit around $B_{r_1}$ so that all its vertices, except at most two, are white. This leads to consider three subcases, in particular the one shown on the left: the two white clusters are so that the external frontier of their union is a black circuit surrounding $B_{r_1}$. In the middle, we show a sequence of outer boundaries of the loops that are discovered in the CLE$_\kappa$ (in the order yellow, blue, green, purple and red), according to the exploration process in Description~\ref{descrip} (one should discover countably many loops, but we only draw finitely many for illustration). The union of the red and green loops disconnect $0$ from $\infty$.
	On the right,  we map out all the loops discovered before the red one, which become boundary arcs of respective colors. The image of the red loop is an SLE$_\kappa$ bubble, whose union with the green arc disconnects $0$ from $\infty$.}
	\label{fig:three_arms}
\end{figure}

\begin{remark}[Other values of cluster weight $q$]\label{rmk:Ising}
In our proofs, we decided to focus on the case of Bernoulli percolation, that is, the random-cluster process with $q=1$. However, we expect similar reasonings to be applicable for all $q \in [1,4)$, as long as the convergence of the discrete process to ${\rm CLE}_{\kappa}$ is known rigorously, which so far is the case only for $q=2$ (corresponding to $\kappa=\frac{16}{3}$), i.e. the FK Ising process, following the breakthrough \cite{Sm10}. Indeed, our proofs in this Section are based on Russo-Seymour-Welsh type crossing estimates and a-priori bounds on arm events, which are known for the FK Ising model \cite{DCHN11,CDCH16}. Hence, we believe that all ingredients are available in the literature for this case. Moreover, such estimates have now been established for all values of $q$ between $1$ and $4$, see \cite{DCST17, DCMT21}, so the beginning of the argument would also apply.
\end{remark}

\begin{remark}[Difficulty for more than two arms]\label{rmk:three-arm}
In principle, one can try to follow a similar approach to compute all $j$-arm monochromatic exponents, $j \geq 3$. However, already for $j = 3$, the situation becomes more complicated: in order to fix ideas, let us discuss this case. Using again Menger's theorem as a starting point, the complement of the existence of $3$ disjoint black arms in an annulus $A_{r_1,r_2}$ is the event that one can find a quasi-white circuit surrounding $B_{r_1}$, containing at most two black sites. This leads in particular to analyze the subevent that there exist two distinct white clusters, which together surround $B_{r_1}$ as shown on Figure~\ref{fig:three_arms} (Left). There should not be any particular difficulty to relate this discrete event to the corresponding continuum event for ${\rm CLE}_6$, again relying on a-priori arm events. Hence, it should be possible to repeat the beginning of the story, up to minor adaptations. However, that continuum event turns out to be much more convoluted to estimate, due to the fact that it involves two distinct loops (instead of a single one). If we explore the loops in the CLE$_\kappa$ using Description~\ref{descrip}, through a Poisson point process of bubbles, then the stopping time for this exploration process is no longer independent from the process: The previously discovered loops are mapped out to some boundary arcs of the domain, see Figure~\ref{fig:three_arms}. We need to keep track of a countable collection of boundary arcs $\{l_j\}_{j\in J}$ from the past, in order to determine the first time $\tau$ that a bubble $e_\tau$ occurs, so that there exists $j\in J$ such that $e_\tau \cup l_j$ disconnects $0$ from $\infty$. It seems difficult to determine the law of this stopping time, so we have not been able to derive the corresponding exponent with the methods of the present paper.
\end{remark}

\begin{figure}[h!]
    \centering
    \includegraphics[width=0.7\textwidth]{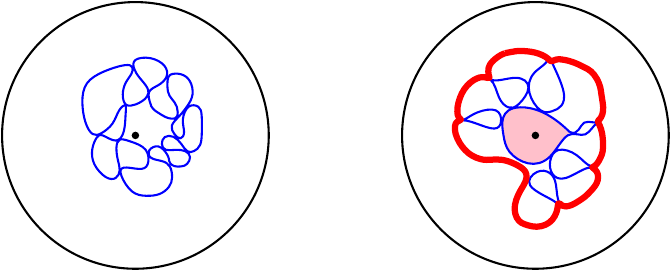}
    \caption{${\rm CLE}_6$ loops related to the one-arm and the backbone exponents. \emph{Left:} The outermost ${\rm CLE}_6$ loop $\gamma_o$ that surrounds the origin,  with inner boundary $\mathcal{L}_o$. The one-arm exponent captures the probability tail of $\mathrm{CR}(0,\mathcal L_o)$. \emph{Right:} The outermost ${\rm CLE}_6$ loop  $\gamma$ whose outer boundary $\ep\gamma$ (in bold red) surrounds the origin. The backbone exponent $\xi$ captures the probability tail of $\mathrm{CR}(0,\ep\gamma)$.  In our figure, the origin is contained in the pink domain, which is not surrounded by  $\gamma$. 
    }%
    \label{fig:explore-CR}%
\end{figure}

\begin{remark}[Connection with the conformal radius]
We notice that both the one-arm and the backbone exponents can be expressed in terms of the conformal radii of certain CLE$_6$ related domains, as poles of their moment generating functions. The one-arm exponent $\alpha_1=\frac{5}{48}$ was computed \cite{LSW02} before CLE was introduced, but it can be related to the following quantity in CLE$_6$.
Let $\gamma_o$ be outermost loop in CLE$_6$ that surrounds the origin.
Let $\mathcal L_o$ be the boundary of the connected component containing $0$ of $\mathbb{C}\setminus \gamma_o$; see Figure~\ref{fig:explore-CR} (left). Then 
\begin{equation}\label{eq:CR1}
\alpha_1=\inf \{x>0:\mathbb E[\mathrm{CR}(0,\mathcal L_o)^{-x}]=\infty \}
\end{equation}
where $\mathrm{CR}(0,\mathcal L_o)$ is the conformal radius of $\mathcal L_o$ viewed from the origin. In~\cite{SSW09}, $\mathbb E[\mathrm{CR}(0,\mathcal L_o)^{-x}]$ was exactly computed. For the backbone exponent $\xi$, we have the following counterpart for the loop $\ep \gamma$ defined in Section~\ref{subsec:bb_exp}, 
see Figure~\ref{fig:explore-CR} (right):
\begin{equation}\label{eq:CR}
 \xi=\inf \{x>0:\mathbb E[\mathrm{CR}(0,\ep\gamma)^{-x}] =\infty \}.
\end{equation}
Indeed, in Section~\ref{subsec:link_bubble}, we have shown $\mathrm{CR}(0,\ep\gamma) =  |\Psi_\tau'(i)|^{-1} = |\Psi_{\tau^-}'(i)|^{-1} |\psi_{e_\tau}'(i)|^{-1}$, and that
\begin{align}\label{eq:CR2}
\xi=\inf\{x>0: \mathbb E[|\Psi_{\tau^-}'(i)|^x]=\infty\}.
\end{align}
Since $ |\Psi_{\tau^-}'(i)|$ and $|\psi_{e_\tau}'(i)| $ are independent, by Lemma~\ref{lem:er2}  and Theorem~\ref{thm:moment} (see Remark~\ref{rmk:alpha})
we have that \eqref{eq:CR} and \eqref{eq:CR2} yield the same exponent.
Although not needed, we can also derive an exact formula for $\mathbb{E}[\mathrm{CR}(0,\ep\gamma)^{-x}]$, as the product of the moment of $|\Psi_{\tau^-}'(i)|$ explicitly given by~\eqref{eq:CR-good} and the  moment of $|\psi_{e_\tau}'(i)| $. The latter was obtained in~\cite{Wu23}, as explained in Remark~\ref{rmk:CR-formula}.

\end{remark}

\section{SLE bubble measure and conformal welding of LQG surfaces}\label{sec:welding}
Recall the setting of Theorem~\ref{thm:moment}. Let $\eta$ be a sample from the SLE bubble measure $\mu_\kappa$ on $\bbH$ rooted at 0. Define $D_i$ as the connected component of $\bbH\setminus \eta$ containing $i\in \bbH$, and let $E_i$ be the event that $\partial D_i\cap \bbR$ is empty. Define $D_{\ep \eta}$ as the domain enclosed by the outer boundary of $\eta$, and let $D_\infty = \mathbb{H} \backslash \overline{D}_{\ep \eta}$ be the complement of $D_{\ep \eta}$, as depicted in Figure~\ref{fig:sle-bubble}. Then $E_i$ is equivalent to the event that $i \in D_{\ep \eta}$ almost surely. 
We normalize $\mu_\kappa$ such that $\mu_\kappa(E_i)=1$ and let $\nu_\kappa$ be the restriction of $\mu_\kappa$ to the complementary event $E^c_i$ of $E_i$.
Next, let $\psi$ be a conformal map from $\mathbb{H}$ to $D_i$ such that $\psi(i)=i$. Theorem~\ref{thm:moment} concerns the law of $|\psi'(i)|$ under the infinite measure $\nu_\kappa$. Since $|\psi'(i)|$ does not depend on the choice of $\psi$, we fix it by requiring that $\psi(0) = b$ and define the point $x = \psi^{-1}(a)$, where $a$ and $b$ are the two endpoints of $\partial D_i \cap \mathbb{R}$ such that $a, 0, b$ are aligned counterclockwise. See Figure~\ref{fig:psi} for an illustration. In addition, on the event $E_i^c$, we introduce $\varphi: \mathbb{H} \rightarrow  \mathbb{H} \backslash \overline{D}_i$ as the conformal map that maps the boundary points $0, 1, \infty$ to $a, 0, b$, respectively. In this section, we derive conformal welding results for LQG surfaces that allow us to access the joint moment of $|\psi'(i)|$ and $|\varphi'(1)|$ under  $\nu_\kappa$, which will lead to the proof of Theorem~\ref{thm:moment}  in Section~\ref{sec:solve}.

\begin{figure}[H]
\centering
\includegraphics[width=8cm]{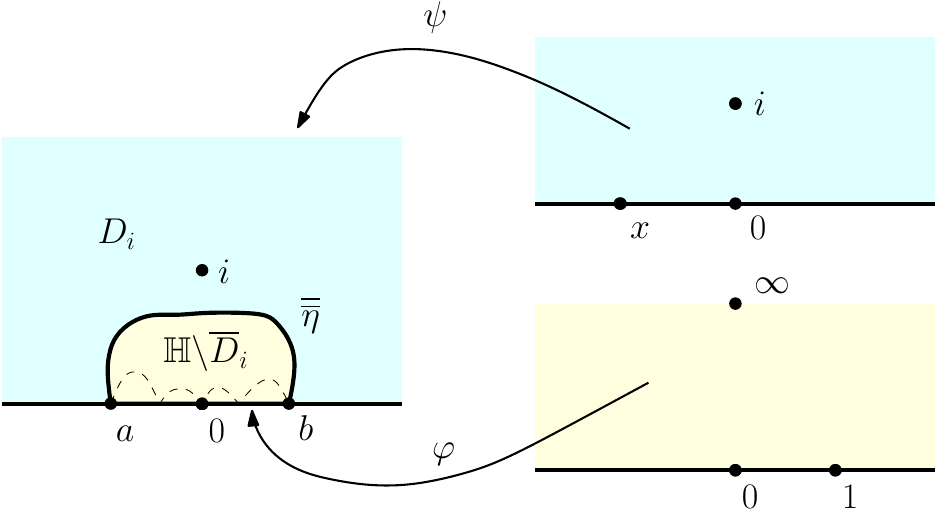}
\caption{Illustration of $\psi$ and $\varphi$. On the left-hand side, the union of the solid line separating the light cyan and light yellow regions and the dashed line represents $\ep \eta$, the outer boundary of $\eta$. The light cyan region corresponds to $D_i$, and the light yellow region corresponds to $\mathbb{H} \backslash \overline{D}_i$. In addition, we have two conformal maps, $\psi : (\mathbb{H}, i,0,x) \rightarrow (D_i, i,b,a)$ and $\varphi: ( \mathbb{H}, 0, 1 ,\infty) \rightarrow (\mathbb{H} \backslash \overline{D}_i, a,0,b)$.}
\label{fig:psi}
\end{figure}
The key result in this section is Proposition~\ref{prop:reweight}, which is the base of our calculations in Section~\ref{sec:solve}. After providing preliminaries about ${\rm LQG}$ surfaces and Liouville fields in 
Sections~\ref{sec:QD} and~\ref{subsub:LF}, we first prove some intermediate conformal welding results in Sections~\ref{sec:bubble-weld}---\ref{sec:weld-lcft}. The proof are based on earlier results from~\cite{Wu23} and~\cite{ASY22}, and on techniques developed in~\cite{AHS21,ARS21}. However, for calculations in Section~\ref{sec:solve} we must  keep track of constants that are implicit in those works. We finally prove Proposition~\ref{prop:reweight} in Section~\ref{sec:reweight} using the strategy from~\cite{AHS21}.

We fix some conventions used in Sections~\ref{sec:welding} and~\ref{sec:solve}. 
We will use the following SLE and LQG parameters:
\begin{equation}\label{eq:parameter}
 \gamma\in (\sqrt2,2), \quad    \kappa=16/\gamma^2\in (4,8),\quad \textrm{and}\quad Q = \frac{2}{\gamma}+\frac{\gamma}{2}. 
\end{equation}

We  will frequently consider infinite measures and extend the probability terminology to this setting. 
In particular, suppose $M$ is a $\sigma$-finite measure on a measurable space $(\Omega, \mathcal F)$.
Suppose  $X:(\Omega,\mathcal F)\rightarrow (E,\mathcal E)$ is an $\mathcal F$-measurable function taking values in $(E,\mathcal E)$.
Then we say that  $X$ is a random variable on $(\Omega,\mathcal F)$
and call  the pushforward measure $M_X = X_*M$ on $(E,\sigma(X))$ the \emph{law} of $X$.
We say that $X$ is \emph{sampled} from $M_X$. We also write the integral $\int X(\omega) M(d\omega)$ as $M[X]$  for simplicity. For a finite measure $M$, we write $|M|$ as its total mass  
and write $M^{\#}=|M|^{-1}M$ as the probability measure proportional to $M$.
 
We also need the concept of disintegration.	Let $M$ be a measure on  a measurable space $(\Omega, \mathcal F)$. 
	Let $X:\Omega\rightarrow \mathbb R^n$ be a measurable function with respect to $\mathcal F$, where $\mathbb R^n$ is endowed with the Borel  $\sigma$-algebra. 
	A family of measures $\{M_x: x\in \mathbb R^n \}$ on $(\Omega, \mathcal F)$ 
	is called a disintegration of $M$ over $X$ if for each set $A\in\mathcal F$, the function $x\mapsto M_x(A)$ is Borel measurable, and 
\begin{equation}\label{eq:disint}
	\int_{A} f(X) \, dM= \int_{\mathbb R^n}   f(x)  M_x(A) \,d^nx \textrm{ for each non-negative measurable function } f \textrm{ on } \mathbb R^n. 
	\end{equation}
When~\eqref{eq:disint} holds, we simply write $M=\int_{\mathbb R^n} M_x \, d^nx$.

\subsection {Quantum disks in LQG}\label{sec:QD}
For $a \in \mathbb{R}$, we write $a_+ := \max \{a,1\}$. Let $\mathcal{S}$ be the horizontal strip $\mathbb{R} \times (0, \pi)$, and $\mathbb{H}$ be the upper half plane $\{z : {\rm Im} z>0 \}$. Let ${\rm exp}:\mathcal{S} \rightarrow \mathbb{H}$ be the exponential map $z \rightarrow e^z$. All random functions or distributions on $\mathbb{H}$ or $\mathcal{S}$ considered in this section will belong to the Sobolev space with index $-1$.

We will consider the Gaussian free field $h_D$ with free boundary conditions defined on two different domains $D = \mathcal{S}$ or $D = \mathbb{H}$. The resulting $h_D$ are random functions modulo an additive constant. We fix the additive constant such that the average of $h_{\mathcal{S}}$ over the vertical segment $\{ 0  \} \times (0,\pi)$ is zero, and the average of $h_{\mathbb{H}}$ over the semicircle $\{ e^{i \theta} : \theta \in (0,\pi)\}$ is zero. In other words, $h_D$ is the centered Gaussian process on $D=\mathcal{S}$ or $D = \mathbb{H}$ with the covariance kernel $G_D(z,w):=\mathbb{E}[h_D(z)h_D(w)]$ given by
\begin{equation}
\begin{aligned}
\label{eq:def-gff}
&G_{\mathbb S}(z,w) = -\log |e^z -e^w| - \log |e^z - e^{\bar w }| + 2 \max \{ {\rm Re}z,0\} + 2\max \{ {\rm Re}w,0\}\,,\\
&G_{\mathbb H}(z,w) = G_{\mathbb S}(e^z,e^w) = -\log |z-w|  - \log |z-\bar w|+2\log|z|_+ +2\log |w|_+\,.
\end{aligned}
\end{equation}

We now review the concept of quantum surfaces. For $n \in \mathbb{N}$, we consider tuples $(D,h,z_1,\ldots,z_n)$, where $D$ is a simply connected domain on $\mathbb{C}$, $h$ is a distribution on $D$ and $z_i$ are marked point on the bulk and the boundary of $D$. Let $(\widetilde D,\widetilde h,\widetilde z_1,\ldots,\widetilde z_n)$ be another tuple. We say that
$$
(D,h,z_1,\ldots,z_n) \sim_\gamma (\widetilde D,\widetilde h,\widetilde z_1,\ldots,\widetilde z_n)
$$
if there is a conformal mapping $f:D \rightarrow \widetilde D$ such that 
\begin{equation}\label{def:push}
    \widetilde h = f \bullet_\gamma h := h\circ f^{-1} + Q \log |(f^{-1})'|
\end{equation} 
and $f(z_i) = \widetilde z_i$ for each $1 \leq i \leq n$. We call each tuple $(D, h, z_1, \ldots, z_m)$ modulo the equivalence relation $\sim_\gamma$ a ($\gamma$-)quantum surface. We can similarly define a quantum surface decorated with curves. For a quantum surface, we define the \textit{quantum area measure} $\mu_h :=\lim_{\epsilon \rightarrow 0} \epsilon^{\gamma^2/2} e^{\gamma h_\epsilon(z)} d^2z$, where $d^2z$ is the Lebesgue measure and $h_\epsilon(z)$ is the circle average of $h$ over $\partial B_\epsilon(z)$. When $D = \mathbb{H}$, we can define the \textit{quantum boundary length measure} $\nu_h :=\lim_{\epsilon \rightarrow 0} \epsilon^{\gamma^2/4} e^{\frac{\gamma}{2} h_\epsilon(z)} dz$, where $dz$ is the Lebesgue measure on $\mathbb{R} = \partial \mathbb{H}$ and $h_\epsilon(z)$ is the semicircle average of $h$ over $\{z + \epsilon e^{i\theta} :\theta \in (0,
\pi) \}$. We can extend this definition to any simply connected domain by using conformal maps. The quantum area measure and quantum boundary length measure are well-defined under the equivalence relation $\sim_\gamma$.

Now we define the notion of push-forward measure that we will use in this paper. Let $f:D \rightarrow \widetilde D$ be a conformal map. For a measure $M$ on the space of random functions on the domain $D$, we define $f_* M$ as the push-forward measure of $M$ under the mapping $h \rightarrow f \bullet_\gamma h$. For a measure $m$ on random curves on the domain $D$, we define $f_* m$ as the push-forward measure of $m$ under the mapping $\eta \rightarrow f(\eta)$.

The two-pointed (thick) quantum disk was introduced in \cite{DMS14}, which we recall now.

\begin{definition}\label{def:thick}
    Fix $W \geq \frac{\gamma^2}{2}$, and let $\beta = \gamma + \frac{2-W}{\gamma} < Q$. Let $h = h^1 + h^2 + \bm{c}$ where $h^1, h^2$ are random distributions on $\mathcal{S}$, $\bm{c} \in \mathbb{R}$ and they are independent such that
    \begin{enumerate}
        \item $h^1(z) = X_{{\rm Re} z}$ for each $z \in \mathcal{S}$, and
        \begin{equation*}
        X_t \overset{d}{=} \begin{cases}
        B_{2t} - (Q -\beta)t \quad  \mbox{  if }t \geq 0 \\
        \widetilde B_{-2t} + (Q-\beta) t \quad \mbox{if }t < 0
        \end{cases},
        \end{equation*}
        where $(B_t)_{t \geq 0}$ and $(\widetilde B_t)_{t \geq 0}$ are two independent standard Brownian motions conditioned on $B_{2t} - (Q-\beta)t<0$ and $\widetilde B_{2t} - (Q-\beta)t<0$ for all $t > 0$.
        \item $h^2$ has the same law as the lateral part of $h_{\mathcal{S}}$, i.e., the projection of $h_{\mathcal{S}}$ onto the subspace of functions which have mean zero on each vertical segment $\{ t \} \times (0,\pi)$ for all $t \in \mathbb{R}$ in the Hilbert space endowed with the Dirichlet inner product.
        \item $\bm c$ is sampled from the infinite measure $\frac{\gamma}{2} e^{(\beta-Q)c} dc$.
    \end{enumerate}
    Let $\mathcal{M}^{\rm disk}_2(W)$ be the infinite measure describing the law of $(\mathcal{S},h,-\infty,\infty) /\sim_\gamma$. We call a sample from $\mathcal{M}^{\rm disk}_2(W)$ a (two-pointed) quantum disk of weight $W$.
\end{definition}

When $W \in (0,\frac{\gamma^2}{2})$, the two-pointed (thin) quantum disk was defined in \cite{AHS20} as follows.
\begin{definition}
\label{def:thin-disk}
    For $W \in (0,\frac{\gamma^2}{2})$, the infinite measure $\mathcal{M}^{\rm disk}_2(W)$ on two-pointed beaded quantum surfaces is defined as follows. First, sample $T$ from $(1-\frac{2}{\gamma^2} W)^{-2} {\rm Leb}_{\mathbb{R}_+}$. Then, sample a Poisson point process $\{(u, \mathcal{D}_u)\}$ from the measure ${\rm Leb}_{\mathbb{R}_+} \times \mathcal{M}^{\rm disk}_2(\gamma^2-W)$, and finally concatenate the disks $\{\mathcal{D}_u\}$ for $u\leq T$ according to the ordering induced by $u$. The left (resp.\ right) boundary length of $\mathcal{M}^{\rm disk}_2(W)$ is defined as the total sum of the left (resp.\ right) boundary lengths of all the $\mathcal{D}_u$'s before time $T$.
\end{definition}

We now review the definition of quantum disks, either unmarked or marked by $m$ bulk points and $n$ boundary points, see \cite[Definition 2.2]{AHS21}.
\begin{definition}
\label{def:QD}
    Let $(\mathcal{S}, h, -\infty, \infty)/\sim_\gamma$ be a sample from $\mathcal{M}^{\rm disk}_2(2)$. Let ${\rm QD}$ be the law of $(\mathcal{S}, h) / \sim_\gamma$
    under the reweighted measure $\nu_h(\partial \mathcal{S})^{-2} \mathcal{M}^{\rm disk}_2(2)$.
    Given integers $m,n \geq 0$, let $(\mathcal{S}, h)$ be a sample from $\mu_h(\mathcal{S})^m \nu_h(\partial \mathcal{S})^n {\rm QD}$. Next, sample $z_1,\ldots, z_m$ and $w_1,\ldots w_n$ independently according to the probability measures $\mu_h^\#$ and $\nu_h^\#$, respectively. Let ${\rm QD}_{m,n}$ be the law of $$(\mathcal{S},h,z_1,\ldots,z_m,w_1,\ldots ,w_n)/\sim_\gamma.
    $$
    Let  $\widetilde {\rm QD}_{m,n}$ be the restriction of ${\rm QD}_{m,n}$ on the event that $w_1,\ldots,w_n$ are in counterclockwise order.
\end{definition}

Next, we define a two-pointed quantum disk with a quantum typical bulk marked point. 
\begin{definition}\label{def:M2bulk}
For $W \geq \frac{\gamma^2}{2}$, let $(\mathcal{S}, h, -\infty,\infty)$ be a sample from the reweighted measure $\mu_h(\mathcal{S}) \mathcal{M}_{2}^{\rm disk}(W)$, and independently sample a point $z$ on $\mathcal{S}$ from the probability measure $\mu_h^\#$. We define $\mathcal{M}_{2,\bullet}^{\rm disk}(W)$ as the law of $(\mathcal{S},h,-\infty,\infty,z)/\sim_\gamma$. Similarly, for $W \in (0, \frac{\gamma^2}{2})$, $\mathcal{M}_{2,\bullet}^{\rm disk}(W)$ is defined as the law on three-pointed beaded quantum surfaces with two boundary points and one bulk point. This measure is obtained by first sampling a two-pointed beaded quantum surface from $\mathcal{M}_{2}^{\rm disk}(W)$ reweighted by its quantum area measure, and then marking a bulk point independently according to a probability measure proportional to its quantum area measure.
\end{definition}

The following lemma is a direct consequence of Definition~\ref{def:thin-disk}. It follows verbatim from Lemma 4.1 and Proposition 4.4 of \cite{AHS20}, and therefore, we omit the proof here.

\begin{lemma}
\label{lem:thin-thick}
For $W \in (0,\frac{\gamma^2}{2})$, we have
$$
\mathcal{M}_{2, \bullet}^{\rm disk}(W) = \big(1 - \frac{2W}{\gamma^2}\big)^2 \mathcal{M}_{2}^{\rm disk}(W) \times \mathcal{M}_{2,\bullet}^{\rm disk}(\gamma^2-W) \times \mathcal{M}_{2}^{\rm disk}(W)\,,
$$
where the right-hand side denotes the infinite measure on ordered collections of quantum surfaces obtained
by concatenating samples from the three independent measures.
\end{lemma}

\subsection{Liouville fields and quantum surfaces}\label{subsub:LF}

We first recall the definition of Liouville field on the upper half plane. We write $P_{\mathbb{H}}$ as the law of the Gaussian free field $h_{\mathbb{H}}$ considered at the beginning of Section~\ref{sec:QD}. 
 
\begin{definition}
    Sample $(h,\bm c)$ from the infinite measure $P_{\mathbb{H}} \times [e^{-Qc} dc]$ and let $\phi(z) = h(z) - 2Q \log|z|_+ + \bm c$. We say $\phi$ is a Liouville field on $\mathbb{H}$ and denote its law by ${\rm LF}_{\mathbb{H}}$.
\end{definition}

We also need Liouville fields with one bulk and multiple boundary insertions, which we define below. The same description without bulk insertion was given in~\cite[Section 2]{AHS21}, and the case involving only bulk insertion was given in~\cite[Section 2]{ARS21}.

\begin{definition}
    \label{def:LF}
    Let $(\alpha, u) \in \mathbb{R} \times \mathbb{H}$ and $(\beta_i,s_i) \in \mathbb{R} \times (\partial \mathbb{H} \cup \{\infty \})$ for $1 \leq i \leq m$ where $m \geq 0$ and all $s_i$'s are distinct. We also assume that $s_i \neq \infty$ for $i \geq 2$. Let the constant
    \begin{equation*}
    \begin{aligned}
    &\quad C_{\mathbb{H}}^{(\alpha,u),(\beta_i,s_i)_i} := \\
    & \quad \begin{cases}
        (2 {\rm Im} u)^{-\frac{\alpha^2}{2}} |u|_+^{-2\alpha(Q-\alpha)} \prod_{i=1}^m |s_i|_+^{-\beta_i(Q-\frac{\beta_i}{2})} \times e^{ \sum_{ 1 \leq i < j \leq m } \frac{\beta_i\beta_j}{4} G_{\mathbb{H}}(s_i,s_j) +  \sum_{i=1}^m \frac{\alpha\beta_i}{2} G_{\mathbb{H}}(u,s_i)}\quad  \mbox{  if }s_1 \neq \infty \\
        (2 {\rm Im} u)^{-\frac{\alpha^2}{2}} |u|_+^{-2\alpha(Q-\alpha)} \prod_{i=2}^m |s_i|_+^{-\beta_i(Q-\frac{\beta_i}{2})} \times e^{ \sum_{ 1 \leq i < j \leq m} \frac{\beta_i\beta_j}{4} G_{\mathbb{H}}(s_i,s_j) +  \sum_{i=1}^m \frac{\alpha\beta_i}{2} G_{\mathbb{H}}(u,s_i)}\quad  \mbox{  if }s_1 = \infty.
    \end{cases}
    \end{aligned}
    \end{equation*}
    Here, we use the convention that $G_{\mathbb{H}}(z,\infty) := \lim_{w \rightarrow \infty} G_{\mathbb{H}} (z,w) = 2 \log|z|_+$. 
    
    Sample $(h, \bm c)$ from $C_{\mathbb{H}}^{(\alpha,u),(\beta_i,s_i)_i} P_{\mathbb{H}} \times [e^{(\frac{1}{2} \sum_{i=1}^m \beta_i + \alpha - Q)c} dc]$, and let $\phi(z) = h(z) - 2Q\log|z|_+ +\frac{1}{2} \sum_{i=1}^m \beta_i G_{\mathbb{H}} (z,s_i) + \alpha G_{\mathbb{H}}(z,u) + \bm c$. Then we define ${\rm LF}_{\mathbb{H}}^{(\alpha,u),(\beta_i,s_i)_i}$ as the law of $\phi$. When $\alpha = 0$, we simply write it as ${\rm LF}_{\mathbb{H}}^{(\beta_i,s_i)_i}$.
\end{definition}

The following lemma shows that the definition mentioned above corresponds to a field obtained by formally reweighting the law ${\rm LF}_{\mathbb H}$ through adding a bulk insertion $e^{\alpha \phi(u)}$ or boundary insertions $e^{ \frac{\beta_i}{2} \phi_\epsilon(s_i)}$.

\begin{lemma}
\label{lem:def-LF}
    With notation as in Definition~\ref{def:LF}, the following equations hold under vague topology:
    \begin{align*}
        &{\rm LF}_{\mathbb{H}}^{(\alpha,u),(\beta_i,s_i)_i} =\lim_{\epsilon \rightarrow 0}\epsilon^{\frac{1}{2}\alpha^2 + \frac{1}{4} \sum_{i=1}^m \beta_i^2 }e^{\alpha \phi_\epsilon(u) + \sum_{i=1}^m \frac{\beta_i}{2} \phi_\epsilon(s_i) } {\rm LF}_{\mathbb H}(d \phi)  \quad\qquad \mbox{  if }s_1 \neq \infty \\
        &{\rm LF}_{\mathbb{H}}^{(\alpha,u),(\beta_i,s_i)_i} =\lim_{\epsilon \rightarrow 0}\epsilon^{\frac{1}{2}\alpha^2 - Q \beta_1 +  \frac{1}{4} \sum_{i=1}^m \beta_i^2 }e^{\alpha \phi_\epsilon(u) + \sum_{i=1}^m \frac{\beta_i}{2} \phi_\epsilon(s_i) } {\rm LF}_{\mathbb H}(d \phi) \quad  \mbox{  if }s_1 = \infty,
    \end{align*} where $\phi_\epsilon(z)$ is the average of $\phi$ over $\{ z + \epsilon e^{i \theta} : \theta \in (0,\pi)\}$ for $z \in \partial \mathbb{H}$ and $\phi_\epsilon(\infty)$ is defined to be $\phi_{1/\epsilon}(0)$.
\end{lemma}

\begin{proof}
    We refer to~\cite[Lemma 2.8]{SY23} for the case of $\alpha = 0$, i.e., there is no bulk insertion, and \cite[Lemma 2.2]{ARS21} for the case with only bulk insertion. The result follows by combining their arguments.
\end{proof}

We now describe how Liouville fields change under conformal maps. 
Recall the notation $f\bullet_\gamma $ from~\eqref{def:push} and the notion of pushforward of field measures above Definition~\ref{def:thick}.
For $\alpha\in\mathbb R$, we define the corresponding scaling exponent by
\begin{equation}\label{eq:KPZ}
   \Delta_\alpha:= \frac{\alpha}{2}(Q-\frac{\alpha}{2}).
\end{equation} 

\begin{lemma}
    \label{lem:coordinate-change}
   Fix $(\alpha,u) \in \mathbb{R} \times \mathbb{H}$ and $(\beta_i,s_i) \in \mathbb{R} \times \partial \mathbb{H}$ for $1 \leq i \leq m$ where $m \geq 0$ and all $s_i$'s are distinct. Suppose $f:\mathbb{H} \rightarrow \mathbb{H}$ is a conformal map such that $f(s_i) \neq \infty$ for each $i$. Then, we have
    $$
    {\rm LF}_{\mathbb{H}}^{(\alpha,f(u)),(\beta_i,f(s_i))_i} = |f'(u)|^{-2\Delta_\alpha} \prod_{i=1}^m |f'(s_i)|^{-\Delta_{\beta_i}} f_*{\rm LF}_{\mathbb{H}}^{(\alpha,u),(\beta_i,s_i)_i}\,.
    $$
\end{lemma}
\begin{proof}
    Taking $\mu = \mu_\partial = 0$ in \cite[Theorem 3.5]{HRV18} yields the result. Although their statement is for the unit disk, the argument can be adapted to $\mathbb{H}$ by using their Proposition 3.7. 
\end{proof}

We need an instance of the following principle: adding a quantum typical point in the bulk or the boundary means adding a $\gamma$-singularity to the Liouville field.
\begin{lemma}\label{lem:typ-pt}
    For any $(\alpha,u) \in \mathbb{R} \times \mathbb{H}$ and $(\beta_i,s_i) \in \mathbb{R} \times \partial \mathbb{H}$, we have
    \begin{align*}
    &{\rm LF}_{\mathbb{H}}^{(\alpha,u),(\beta_i,s_i)_i}(d\phi) \nu_\phi(dq) = {\rm LF}_{\mathbb{H}}^{(\alpha,u),(\beta_i,s_i)_i,(\gamma,q)}(d\phi) dq,\\
    &{\rm LF}_{\mathbb{H}}^{(\beta_i,s_i)_i}(d\phi) \mu_\phi(d^2p) = {\rm LF}_{\mathbb{H}}^{(\gamma,p),(\beta_i,s_i)_i}(d\phi) d^2p.
    \end{align*}
    Here, $d^2p$ is the Lebesgue measure on $\mathbb{H}$ and $dq$ is the Lebesgue measure on $\partial \mathbb{H}$.
\end{lemma}

\begin{proof}
    The first identity follows from \cite[Lemma 2.13]{SY23}, where they derived the case with only boundary insertions. This can be extended to our case by adding one bulk insertion in the end and then applying Lemma~\ref{lem:def-LF}. The second identity follows from \cite[Equation (3.5)]{ARS21} where they derived the result for ${\rm LF}_{\mathbb H}$. We can extend this to our case by adding multiple boundary insertions to their equation and then applying Lemma~\ref{lem:def-LF}.
\end{proof}

We will now introduce the two quantum surfaces that will be used in Proposition~\ref{prop:reweight}. They are both defined by Liouville fields. Under specialization of certain insertion parameters, they are related to some aforementioned quantum surfaces up to an explicit constant. The first surface has a free boundary parameter.
\begin{definition}
\label{def:M12}
Fix $\alpha \in \mathbb{R}$. Sample $(x,\phi)$ from 
${\rm LF}_\mathbb{H}^{(\alpha, i), (\gamma,0), (\gamma,x)}(d \phi) d x$, 
where $d x$ is the Lebesgue measure on $\mathbb R$. Let $\mathcal{M}_{1,2}^{\rm disk}(\alpha)$ be the law of the marked quantum surface $(\mathbb{H}, \phi, i, 0, x)/{\sim_\gamma}$.
\end{definition}
The parameters $W$ in $\mathcal{M}^{\rm disk}_{2,\bullet}(W)$ and $\alpha$ in $\mathcal{M}^{\rm disk}_{1,2}(\alpha)$ have different meanings. The former comes from $\mathcal{M}^{\rm disk}_{2}(W)$ and denotes the weight of a two-pointed disk, while the latter means the magnitude of the bulk logarithmic singularity. The measures $\mathcal{M}^{\rm disk}_{2,\bullet}(2)$ and $\mathcal{M}^{\rm disk}_{1,2}(\gamma)$ are equal up to an explicit constant.  
\begin{lemma}
\label{lem:embedding-2}
Recall $\mathcal{M}^{\rm disk}_{2,\bullet}(2)$ from Definition~\ref{def:M2bulk}. When viewed as  laws of quantum surfaces, we have
\begin{equation}
\label{eq:embedding-2}
\mathcal{M}_{2,\bullet}^{\rm disk}(2) = \frac{\gamma}{2(Q-\gamma)^2} \mathcal{M}^{\rm disk}_{1,2}(\gamma)\,.
\end{equation}
\end{lemma}

\begin{proof} 

We first show that, when viewed as laws of quantum surfaces,
\begin{equation}
\label{eq:lem-embed-2-1}
{\rm QD}_{1,1} = \frac{\gamma}{2(Q-\gamma)^2} {\rm LF}_\mathbb{H}^{(\gamma, i), (\gamma,0)}.
\end{equation}
This result follows from \cite[Proposition 3.9]{ARS21}, where the explicit value of the constant has not been derived. However, following verbatim their argument, we can obtain the constant as elaborated below. By \cite[Theorem 3.4]{ARS21}, we have ${\rm QD}_{1,0} =  \frac{\gamma}{2 \pi (Q-\gamma)^2} {\rm LF}_\mathbb{H}^{(\gamma,i)}$. Adding a quantum typical boundary marked point to both sides of this equation and then applying Lemma~\ref{lem:typ-pt}, we obtain that
$$
{\rm QD}_{1,1} =  \frac{\gamma}{2 \pi (Q-\gamma)^2} {\rm LF}_\mathbb{H}^{(\gamma,i),(\gamma,x)}(d\phi)dx\,.
$$
For each fixed $x$, applying the conformal map $f_x(z) = \frac{z-x}{xz+1}$ which maps $(\mathbb{H}, i, x)$ to $(\mathbb{H}, i ,0)$ to the Liouville field on the right-hand side, and then using the coordinate change formula (Lemma~\ref{lem:coordinate-change}) yields
\begin{align*}
{\rm QD}_{1,1} &= \frac{\gamma}{2 \pi (Q-\gamma)^2} \int_\mathbb{R} \frac{1}{x^2+1} dx \cdot {\rm LF}_\mathbb{H}^{(\gamma,i),(\gamma,0)}(d\phi)\,.
\end{align*}
Combining with $\int_\mathbb{R} \frac{1}{x^2+1} dx = \pi$ yields~\eqref{eq:lem-embed-2-1}. Finally, by adding a quantum typical boundary marked point to both sides of~\eqref{eq:lem-embed-2-1} yields the desired lemma: the left-hand side becomes ${\rm QD}_{1,2}$, which equals to $\mathcal{M}^{\rm disk}_{2,\bullet}(2)$ by Definitions~\ref{def:QD} and \ref{def:M2bulk}, and the right-hand side becomes $\mathcal{M}^{\rm disk}_{1,2}(\gamma)$ by Lemma~\ref{lem:typ-pt}. \qedhere

\end{proof}

The second quantum surface we need is the quantum triangle introduced in \cite{ASY22}.  
\begin{definition}
\label{def:QT}
Let $W_1, W_2$, and $W_3$ be fixed with $ W_1,W_2,W_3 > \frac{\gamma^2}{2}$. We set $\beta_i = \gamma + \frac{2-W_i}{\gamma}< Q$ for $i=1,2,3$. Sample $\phi$ from $\frac{1}{(Q-\beta_1)(Q-\beta_2)(Q-\beta_3)}{\rm LF}_\mathbb{H}^{(\beta_1,0),(\beta_2,1),(\beta_3,\infty)}$ and define the infinite measure ${\rm QT}(W_1, W_2, W_3)$ as the law of $(\mathbb{H}, \phi, 0,1,\infty) / \sim_\gamma$. We call a sample from ${\rm QT}(W_1, W_2, W_3)$ a quantum triangle of weights $W_1$, $W_2$, and $W_3$.
\end{definition}

Definition~\ref{def:QT} is equivalent to \cite[Definition 2.17]{ASY22} by their Lemma 2.9. The quantum triangle ${\rm QT}(W_1, W_2, W_3)$ was defined for all $W_i>0$  in \cite{ASY22}, but we will focus on the case of $W_1,W_2,W_3>\frac{\gamma^2}{2}$. Recall $\widetilde {\rm QD}_{0,3}$ from Definition~\ref{def:QD}. Then, as shown in the lemma below, the measures $\widetilde {\rm QD}_{0,3}$ and ${\rm QT}(2,2,2)$ are equal up to a constant. This is a special case of \cite[Proposition 2.18]{AHS21}, as explained in Remark 2.19 there.
\begin{lemma}
\label{lem:QT} 
    There exists a constant $C \in (0,\infty)$ such that when viewed as laws of quantum surfaces with marked points, we have
    $$
    \widetilde {\rm QD}_{0,3}= C {\rm QT}(2,2,2)\,.
    $$
\end{lemma}

\subsection{The outer boundary of the SLE$_\kappa$ bubble via conformal welding}\label{sec:bubble-weld}

For a measure $m$ on curves, we use $\overline{m}$ to denote the law of the outer boundary of a curve sampled from $m$. For example, 
Let $\eta$ be a sample from the SLE bubble measure $\mu_\kappa$ whose outer boundary is $\ep \eta$. Then
$\epmu$ is the law of $\ep \eta$ under $\mu_\kappa$. Recall that $\mu_\kappa$ is normalized such that $\mu_\kappa[E_i]=1$, where $E_i$ is event the $i$ is not contained in the unbounded connected component of $\bbH\setminus \ep\eta$.
Let $\mumu $ be the restriction of $\mu_\kappa$ to the event $E_i$. Namely, $\mumu = \mu_\kappa - \nu_\kappa$. In this subsection we present our first conformal welding result, which produces $\epmumu$ using conformal welding of quantum surfaces. 
In Section~\ref{subsec:relocate}, we do the same for $\epnu$.

We start with the definition of the disintegration of a measure on a quantum surface or a Liouville field over the boundary lengths. For more details, we refer to \cite[Section 2.6]{AHS20}. Suppose $M$ is the law of a quantum surface or a Liouville field. For a sample from $M$, let $\mathcal{L}$ be the quantum length of a specific boundary arc. Then, we can construct a family of measures $\{M(\ell)\}_{\ell > 0}$ such that $M(\ell)$ is supported on quantum surfaces with $\mathcal{L} = l$, and 
\begin{equation}
\label{eq:def-disintegration}
    M = \int_0^\infty M(\ell) d\ell\,.
\end{equation} We will clarify the specific boundary arc that we are considering in the context. See ${\rm QD}_{1,1}(\ell)$ and $\mathcal{M}_2^{\rm disk}(\gamma^2-2)(\ell)$ defined below for an example. Sometimes we will also consider several boundary arcs $\mathcal L_1, \mathcal L_2,\ldots$, and define the disintegration measure $\{M(\ell_1,\ell_2,\ldots)\}_{\ell_1>0,\ell_2>0,\ldots}$ with respect to them.

We now introduce the notion of conformal welding. Suppose $M_1$ and $M_2$ are the laws of a quantum surface or a Liouville field. We specify one boundary arc for each of their samples and consider the disintegration $\{M_1(\ell)\}_{\ell>0}$ and $\{M_2(\ell)\}_{\ell>0}$ as defined in \eqref{eq:def-disintegration}. Let $\ell>0$. Given a pair of quantum surface sampled from $M_1(\ell) \times M_2(\ell)$, we conformally weld them along their length $\ell$ boundary arcs according to the quantum length and obtain a quantum surface decorated with a curve. We denote the law of the resulting quantum surface as ${\rm Weld}(M_1(\ell),M_2(\ell))$. We define the law of the quantum surface ${\rm Weld}(M_1, M_2)$ as:
\begin{equation}
\label{eq:conf-weld}
{\rm Weld}(M_1, M_2) = \int_0^\infty {\rm Weld}(M_1(\ell),M_2(\ell)) d \ell\,.
\end{equation}
We refer to ${\rm Weld}(M_1, M_2)$ as the conformal welding of $M_1$ and $M_2$, and we will clarify the specific boundary arcs that we are considering in the context. See ${\rm Weld}(\mathcal{M}_{2}^{\rm disk}(\gamma^2-2), {\rm QD}_{1,1})$ defined below for an example. Sometimes, we will also consider the conformal welding of more than two measures of quantum surfaces: ${\rm Weld}(M_1,M_2,\ldots)$, which can be defined inductively. See ${\rm Weld}(\mathcal{M}_{2}^{\rm disk}(\gamma^2-2), \mathcal{M}_{2, \bullet}^{\rm disk}(2), \mathcal{M}_{2}^{\rm disk}(\gamma^2-2), \widetilde {\rm QD}_{0,3})$ introduced in Section~\ref{sec:weld-lcft} for an example.

We now give an example of conformal welding: ${\rm Weld}( \mathcal{M}_{2}^{\rm disk}(\gamma^2-2), {\rm QD}_{1,1})$ to clarify the definition of disintegration and conformal welding. This welded surface will be used in Proposition~\ref{prop:weld6}.

Recall from Definition~\ref{def:QD} that ${\rm QD}_{1,1}$ is a measure on a quantum surface with one bulk and one boundary marked point. Let $\mathcal{L}$ be the total quantum boundary length of a sample from this measure. Then, as in \eqref{eq:def-disintegration}, we can construct a family of measures $\{{\rm QD}_{1,1}(\ell)\}_{\ell > 0}$, where ${\rm QD}_{1,1}(\ell)$ is supported on quantum surfaces with $\mathcal{L} = \ell$, and we have ${\rm QD}_{1,1} = \int_0^\infty {\rm QD}_{1,1}(\ell) d\ell$. Also, recall from Definition~\ref{def:thin-disk} that $\mathcal{M}_{2}^{\rm disk}(\gamma^2-2)$ is a measure on two-pointed beaded quantum surfaces. As in \eqref{eq:def-disintegration}, we can construct the disintegration measure over its right quantum boundary length: $\{\mathcal{M}_{2}^{\rm disk}(\gamma^2-2)(\ell)\}_{\ell > 0}$. Then we have $\mathcal{M}_{2}^{\rm disk}(\gamma^2-2) = \int_0^\infty \mathcal{M}_{2}^{\rm disk}(\gamma^2-2)(\ell) d\ell$.

We now define ${\rm Weld}( \mathcal{M}_{2}^{\rm disk}(\gamma^2-2), {\rm QD}_{1,1})$ as the law on quantum surfaces obtained by conformally welding the right boundary arc of a sample from $\mathcal{M}_{2}^{\rm disk}(\gamma^2-2)$ with the entire boundary arc of a sample from ${\rm QD}_{1,1}$ conditioned that they have the same quantum length. We refer to Figure~\ref{fig:weld6} for an illustration. To be more precise, for $\ell>0$, we independently sample two quantum surfaces from $\mathcal{M}_{2}^{\rm disk}(\gamma^2-2)(\ell)$ and ${\rm QD}_{1,1}(\ell)$, respectively. These surfaces are then conformally welded along their length $\ell$ boundary arcs to obtain a quantum surface decorated with a curve. We denote the law of the resulting quantum surface as ${\rm Weld}(\mathcal{M}_{2}^{\rm disk}(\gamma^2-2)(\ell), {\rm QD}_{1,1}(\ell))$. Finally, as in~\eqref{eq:conf-weld}, ${\rm Weld}( \mathcal{M}_{2}^{\rm disk}(\gamma^2-2), {\rm QD}_{1,1})$ is defined as $\int_0^\infty {\rm Weld}( \mathcal{M}_{2}^{\rm disk}(\gamma^2-2)(\ell), {\rm QD}_{1,1}(\ell)) d \ell$.

We are ready to state our first conformal welding result.
\begin{proposition}
\label{prop:weld6}
There exists a constant $C_1  \in (0,\infty)$ such that
when viewed as laws of quantum surfaces decorated with a curve and marked points, we have: 
\begin{equation} 
\label{eq:forget-1}
{\rm LF}_\mathbb{H}^{(\gamma,i),(\frac{4}{\gamma}-\gamma,0)} \times \epmumu = C_1  {\rm Weld}(\mathcal{M}_{2}^{\rm disk}(\gamma^2-2), {\rm QD}_{1,1}) \,.
\end{equation}
That is, when $(\phi, \eta)$ is sampled from ${\rm LF}_\mathbb{H}^{(\gamma,i),(\frac{4}{\gamma}-\gamma,0)} (d\phi) \times \mumu(d\eta)$, the law of $(\mathbb{H}, \phi, \ep{\eta}, i, 0)$ viewed as a quantum surface decorated with a curve and marked points equals $C_1  {\rm Weld}(\mathcal{M}_2^{\rm disk}(\gamma^2-2), {\rm QD}_{1,1})$. Here, $(D_\infty,0,0)$ corresponds to $\mathcal{M}_2^{\rm disk}(\gamma^2-2)$, and $(D_{\ep \eta},0,i)$ corresponds to ${\rm QD}_{1,1}$; see also the following figure.
\begin{figure}[H]
\centering
\includegraphics[width=8cm]{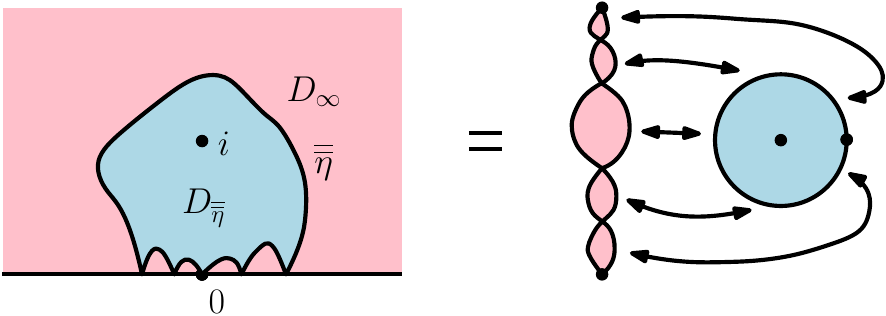}
\caption{Illustration of Proposition~\ref{prop:weld6}. The blue region corresponds to $D_{\ep \eta}$ and the pink region corresponds to $D_\infty$.} 
\label{fig:weld6}
\end{figure}

\end{proposition}

The paper~\cite{Wu23} proves the conformal welding result for the so-called $\SLE_{\gamma^2}(\rho)$ bubble measure. 
When $\rho=\gamma^2-4$, it is equivalent to Proposition~\ref{prop:weld6}  with $\epmumu$ replaced by
the $\SLE_{\gamma^2}(\gamma^2-4)$ bubble measure. The proof for this special case is a straightforward adaption of an argument in~\cite[Section 4]{ARS21} with an additional input from~\cite{ASY22} on the conformal welding of quantum triangles. For completeness, we recall this proof in Appendix~\ref{sec:app}. Given this,  Proposition~\ref{prop:weld6}  follows from the equivalence of the $\SLE_{\gamma^2}(\gamma^2-4)$ bubble measure and $\epmumu$, which is an instance of the SLE duality. We prove this duality hence Proposition~\ref{prop:weld6} in Appendix~\ref{sec:app}.
We will explicitly compute the welding constant $C_1$  in Section~\ref{sec:sol-C1}, which is an important step in the proof of Theorem~\ref{thm:moment}.

\subsection{Relocate the bulk marked point}\label{subsec:relocate}
In this subsection, we present a conformal welding result in Proposition~\ref{prop:weld3}, which produces the measure $\epnu$. Recall from Definitions~\ref{def:QD} and \ref{def:M2bulk} that ${\rm QD}_{0,1}$ is the law on quantum surfaces with one boundary marked point, and $\mathcal{M}_{2,\bullet}^{\rm disk}(\gamma^2-2)$ is the law on two-pointed beaded quantum surfaces with one bulk marked point. Similarly to the conformal welding ${\rm Weld}(\mathcal{M}_{2}^{\rm disk}(\gamma^2-2), {\rm QD}_{1,1})$, we can define the conformal welding ${\rm Weld}(\mathcal{M}_{2,\bullet}^{\rm disk}(\gamma^2-2), {\rm QD}_{0,1})$ obtained by gluing the right boundary arc of a sample from the first measure to the entire boundary arc of a sample from the second measure, see Figure~\ref{fig:weld3} for an illustration.
 
We need the concept of uniform embedding first considered in \cite{AHS21}, which was used to prove similar statements as Proposition~\ref{prop:weld3}.
Sample $(\bm{{\rm p}},\bm{{\rm q}},\bm{{\rm r}})$ from the measure $|(p-q)(q-r)(r-p)^{-1}| dpdqdr$ restricted to triples $(p,q,r)$ which are oriented counterclockwise on $\mathbb{R} = \partial \mathbb{H}$. Let $\mathfrak g$ be the conformal map such that $\mathfrak g(0) = \bm{{\rm p}}$, $\mathfrak g(1) = \bm{{\rm q}}$, and $\mathfrak g(2) = \bm{{\rm r}}$, and define $\mathbf{m}_\mathbb{H}$ to be the law of $\mathfrak g$. As explained in Section 2.5 of \cite{AHS21}, $\mathbf{m}_{\mathbb{H}}$ is the Haar measure on the conformal automorphism group ${\rm conf}(\mathbb{H})$ of $\mathbb{H}$, which is both left and right invariant. 
Recall from~\eqref{def:push} the definition of push-forward of random functions. Suppose $(D,h,z_1,\ldots, z_m)$ is an embedding of a quantum surface such that there exists a conformal map $f: D \rightarrow \mathbb{H}$. We denote by $\mathbf{m}_\mathbb{H} \ltimes (D,h,z_1,\ldots, z_m)$ the law of $(\mathbb{H}, (\mathfrak g \circ f) \bullet_\gamma  h, \mathfrak g \circ f ( z_1), \ldots, \mathfrak g \circ f (z_m))$, and call it a uniform embedding of $(D,h,z_1,\ldots, z_m)$ onto $\mathbb{H}$, where $\mathfrak g$ is sampled independently from $\mathbf{m}_{\mathbb{H}}$ and $\mathfrak g \circ f$ denotes the composition of $\mathfrak g$ and $f$. This definition does not depend on the choice of the embedding and the map $f$, given the quantum surface. In addition, we can adapt this notion to quantum surfaces decorated with curves.

We will use the following crucial observation from~\cite{AHS21}: if two quantum surfaces have the same law under the uniform embedding, then their law must be the same. We now use it to  prove the following proposition based on  Proposition~\ref{prop:weld6}. We emphasize that the welding constant $C_1$ remains the same as in Proposition~\ref{prop:weld6}, which will be computed in Section~\ref{sec:sol-C1}.

\begin{proposition}
\label{prop:weld3}
When viewed as laws of quantum surfaces decorated with a curve and marked points,
$$
{\rm LF}_\mathbb{H}^{(\gamma,i),(\frac{4}{\gamma}-\gamma,0)}  \times \epnu = C_1 {\rm Weld}(\mathcal{M}_{2, \bullet}^{\rm disk}(\gamma^2-2), {\rm QD}_{0,1}) \,\quad \textrm{with } C_1 \textrm{ from }\eqref{eq:forget-1}. $$
That is, when $(\phi, \eta)$ is sampled from ${\rm LF}_\mathbb{H}^{(\gamma,i),(\frac{4}{\gamma}-\gamma,0)} (d\phi) \times \nu_\kappa(d\eta)$, the law of $(\mathbb{H}, \phi, \ep{\eta}, i, 0)$ viewed as a quantum surface decorated with a curve and marked points equals $C_1  {\rm Weld}(\mathcal{M}_{2, \bullet}^{\rm disk}(\gamma^2-2), {\rm QD}_{0,1})$. Here, $(D_\infty,i,0,0)$ corresponds to $\mathcal{M}_{2, \bullet}^{\rm disk}(\gamma^2-2)$, and $(D_{\ep \eta},0)$ corresponds to ${\rm QD}_{0,1}$, as shown in the figure below.
\begin{figure}[H]
\centering
\includegraphics[width=8cm]{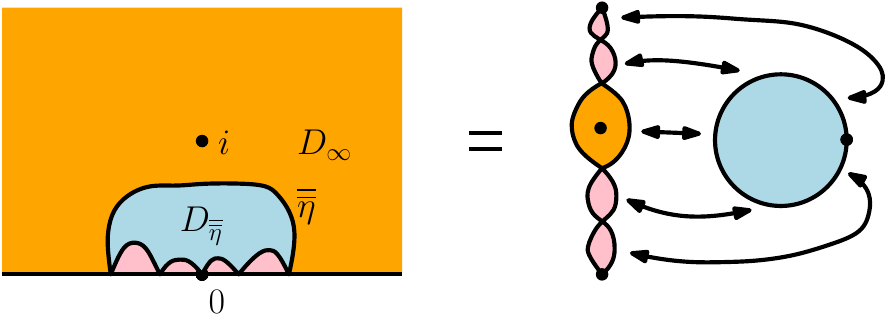}
\caption{Illustration of Proposition~\ref{prop:weld3}. The union of orange and pink regions corresponds to $D_\infty$, and the blue region corresponds to $D_{\ep \eta}$.}
\label{fig:weld3}
\end{figure}

\end{proposition}

\begin{proof}
Applying the uniform embedding to both sides of \eqref{eq:forget-1} in Proposition~\ref{prop:weld6}, we get 
\begin{equation}
\label{eq:forget-2}
\mathfrak{m}_\mathbb{H} \ltimes \big[ {\rm LF}_\mathbb{H}^{(\gamma,i),(\frac{4}{\gamma}-\gamma,0)} \times  \epmumu \big] = C_1  \mathfrak{m}_\mathbb{H} \ltimes [{\rm Weld}(\mathcal{M}_{2}^{\rm disk}(\gamma^2-2), {\rm QD}_{1,1})].
\end{equation}
Recall the definition of push-forward of random curves before Definition~\ref{def:thick}. For $q \in \mathbb{R} = \partial \mathbb{H}$, define $\mu_\kappa^q$ as the ${\rm SLE}_\kappa$ bubble measure rooted at $q$, i.e., the push-forward measure $f_* \mu_\kappa$ of $\mu_\kappa$ under the map $f(z) = z+q$. We use $\overline{\mu}_\kappa^q$ be denote the law of the outer boundary of a loop sampled from $\mu_\kappa^q$. Similar to $D_{\ep \eta}$ and $D_\infty$, we define the domains $D_{\ep \eta}^q$ and $D_\infty^q$ for a non-simple loop $\eta$ rooted at $q$. We claim that:
\begin{equation}
\label{eq:forget-4}
\begin{aligned}
\mathbf{m}_\mathbb{H} \ltimes  \big[ {\rm LF}_\mathbb{H}^{(\gamma,i),(\frac{4}{\gamma}-\gamma,0)} \times \epmumu \big]= \mu_\phi(d^2 p) \mathbbm{1}\{ p \in D_{\ep \eta}^q \} \cdot  {\rm LF}_\mathbb{H}^{(\frac{4}{\gamma}-\gamma,q)}(d \phi) \times  \epmu^q(d\eta) \cdot dq .
\end{aligned}
\end{equation}
The right-hand side corresponds to the measure obtained by first sampling $q$ from the Lebesgue measure on $\mathbb{R}$, then sampling $(\phi,\eta)$ from ${\rm LF}_\mathbb{H}^{(\frac{4}{\gamma}-\gamma,q)}(d \phi) \times  \epmu^q(d\eta)$, and finally sampling $p$ from $\mu_\phi(d^2 p) \mathbbm{1}\{ p \in D_{\ep \eta}^q \}$. 
To prove~\eqref{eq:forget-4}, we must analyze $\mathfrak{m}_\mathbb{H} \ltimes \big[ {\rm LF}_\mathbb{H}^{(\gamma,i),(\frac{4}{\gamma}-\gamma,0)} \times  \epmumu \big]$. Viewing ${\rm LF}_\mathbb{H}^{(\gamma,i),(\frac{4}{\gamma}-\gamma,0)} \times  \epmumu$ as the law of an embedded curve-decorated surface, we apply a conformal map $\mathfrak g$ sampled independently from $\mathfrak{m}_\mathbb{H}$. One way to sample $\mathfrak g$ is the following. Sample $(p,q)$ according to $\frac{1}{{\rm Im} p \cdot |p-q|^2} d^2p dq$ and let $\mathfrak g$ be the conformal automorphism of $\bbH$ such that $(p,q) = (\mathfrak g(i),\mathfrak g(0))$; see e.g.\ Lemmas 5.2 and 5.3 of \cite{Wu23} for a proof of this fact.
Now we have
\begin{equation*}
\mathbf{m}_\mathbb{H} \ltimes  \big[ {\rm LF}_\mathbb{H}^{(\gamma,i),(\frac{4}{\gamma}-\gamma,0)} \times \epmumu \big] =   \mathfrak g_* {\rm LF}_\mathbb{H}^{(\gamma,i),(\frac{4}{\gamma}-\gamma,0)} \times    \mathfrak g_*\epmumu \cdot \frac{1}{{\rm Im} p \cdot |p-q|^2} d^2p dq .  
\end{equation*}
On the right-hand side, both $\mathfrak g_* {\rm LF}_\mathbb{H}^{(\gamma,i),(\frac{4}{\gamma}-\gamma,0)}(d\phi)$ and   $\mathfrak g_*\epmumu(d\eta)$ depend on $(p,q)$ since $\mathfrak g$ is sampled accordingly. In particular,  $|\mathfrak g'(i)| = {\rm Im} p $ and $|\mathfrak g'(0)| = \frac{|p-q|^2}{{\rm Im} p}$. By \cite{MR2979861} and \cite[Theorem 3.7]{Z22}, $\mu_\kappa$ is conformally covariant, namely for any conformal map $f$ from $\mathbb{H}$ onto $\mathbb{H}$,
\begin{align}\label{eq:bubble_cov}
f_*\mu_\kappa=f'(0)^{(8-\kappa)/\kappa}\mu_\kappa.
\end{align}
Combined with Lemma~\ref{lem:coordinate-change}, we have
\begin{equation}
\label{eq:covariance}
\mathfrak g_* {\rm LF}_\mathbb{H}^{(\gamma,i),(\frac{4}{\gamma}-\gamma,0)} = |\mathfrak g'(i)|^{2 \Delta_\gamma} |\mathfrak g'(0)|^{\Delta_{\frac{4}{\gamma} - \gamma}}   {\rm LF}_\mathbb{H}^{(\gamma,p),(\frac{4}{\gamma}-\gamma,q)}  \quad \mbox{and} \quad \mathfrak g_* \mu_\kappa = |\mathfrak g'(0)|^{\frac{8-\kappa}{\kappa}} \mu_\kappa^{q}.
\end{equation}
Since  $\mumu = \mu_\kappa \mathbbm{1}\{ i \in D_{\ep \eta} \}$, we see that $\mathbf{m}_\mathbb{H} \ltimes  \big[ {\rm LF}_\mathbb{H}^{(\gamma,i),(\frac{4}{\gamma}-\gamma,0)} \times \epmumu \big]$ equals
\begin{equation*}
|\mathfrak g'(i)|^{2 \Delta_\gamma} |\mathfrak g'(0)|^{\Delta_{\frac{4}{\gamma} - \gamma}}    {\rm LF}_\mathbb{H}^{(\gamma,p),(\frac{4}{\gamma}-\gamma,q)} \times   |\mathfrak g'(0)|^{\frac{8-\kappa}{\kappa}} \epmu^q(d\eta) \mathbbm{1}\{ p \in D_{\ep \eta}^q \}  \cdot \frac{1}{{\rm Im} p \cdot |p-q|^2} d^2p dq,
\end{equation*}
which further equals $ {\rm LF}_\mathbb{H}^{(\gamma,p),(\frac{4}{\gamma}-\gamma,q)} \times  \epmu^q(d\eta) \mathbbm{1}\{ p \in D_{\ep \eta}^q \}  \cdot  d^2p dq$ since 
$|\mathfrak g'(i)| = {\rm Im} p \mbox{ and }|\mathfrak g'(0)| = \frac{|p-q|^2}{{\rm Im} p}$.
By Lemma~\ref{lem:typ-pt},  we have ${\rm LF}_\mathbb{H}^{(\gamma,p),(\frac{4}{\gamma}-\gamma,q)} \cdot  d^2p=\mu_\phi(d^2 p)  {\rm LF}_\mathbb{H}^{(\frac{4}{\gamma}-\gamma,q)}(d \phi)$, which yields~\eqref{eq:forget-4}.

Combining \eqref{eq:forget-2} and~\eqref{eq:forget-4}, we obtain that
\begin{align}\label{eq:UEQD11}
&\mu_\phi(d^2 p) \mathbbm{1}\{ p \in D_{\ep \eta}^q \} \cdot  {\rm LF}_\mathbb{H}^{(\frac{4}{\gamma}-\gamma,q)}(d \phi) \times  \epmu^q(d\eta) \cdot dq = C_1 \mathbf{m}_\mathbb{H} \ltimes \left[ {\rm Weld}(\mathcal{M}_{2}^{\rm disk}(\gamma^2-2), {\rm QD}_{1,1})  \right].
\end{align}
This equation is equivalent to 
\begin{equation}\label{eq:UEQD10}
 {\rm LF}_\mathbb{H}^{(\frac{4}{\gamma}-\gamma,q)}(d \phi) \times  \epmu^q(d\eta) \cdot dq = C_1    \mathbf{m}_\mathbb{H} \ltimes \left[  {\rm Weld}(\mathcal{M}_{2}^{\rm disk}(\gamma^2-2), {\rm QD}_{0,1})  \right]
\end{equation}
as the left-hand side of \eqref{eq:UEQD11} can be obtained from the left-hand side of~\eqref{eq:UEQD10} by adding a bulk marked point according to the quantum area measure inside the domain $D_{\ep \eta}^q$. Now on both sides of~\eqref{eq:UEQD10}, adding a point to the domain $D_\infty^q$ sampled from its quantum area measure yields that
\begin{align*}
& \mu_\phi(d^2 p) \mathbbm{1}\{ p \in D_{\infty}^q \} \cdot  {\rm LF}_\mathbb{H}^{(\frac{4}{\gamma}-\gamma,q)}(d \phi) \times  \epmu^q(d\eta) \cdot dq = C_1    \mathbf{m}_\mathbb{H} \ltimes \left[  {\rm Weld}(\mathcal{M}_{2,\bullet}^{\rm disk}(\gamma^2-2), {\rm QD}_{0,1})  \right].
\end{align*}
Similar to \eqref{eq:forget-4}, we can show that the left-hand side above equals $\mathbf{m}_\mathbb{H} \ltimes  [{\rm LF}_\mathbb{H}^{(\gamma,i),(\frac{4}{\gamma}-\gamma,0)} \times \epnu ]$ hence 
$\mathbf{m}_\mathbb{H} \ltimes  [{\rm LF}_\mathbb{H}^{(\gamma,i),(\frac{4}{\gamma}-\gamma,0)} \times \epnu ]=C_1    \mathbf{m}_\mathbb{H} \ltimes \left[  {\rm Weld}(\mathcal{M}_{2,\bullet}^{\rm disk}(\gamma^2-2), {\rm QD}_{0,1})  \right]$.
Disintegrating with respect to the Haar measure $\mathbf{m}_\mathbb{H}$ concludes the proof.
\end{proof}

\subsection{Re-welding of quantum surfaces in Proposition~\ref{prop:weld3}}
\label{sec:weld-lcft}

In this section, we decompose and  rearrange the conformal welding picture  from Proposition~\ref{prop:weld3}  to prove Proposition~\ref{prop:weld5}, which will produce the following measure  $\nu_\kappa^{\gamma,\beta_0}$ on simple curves; see Figure~\ref{fig:weld5} (left).

\begin{definition}\label{def:simple}
Let $\beta_0 = \frac{4}{\gamma}-\gamma$.
Let $\eta$ be a sample from  $\nu_\kappa$ and $D_i$ be the connected component of $\bbH\setminus \eta$ containing $i\in \bbH$. Let  $a$ and $b$ represent the two endpoints of $\partial D_i \cap \mathbb{R}$ in such a way that $a,0,b$ are in counterclockwise order. Let $\nu_\kappa^{\gamma,\beta_0}$ be the law of the simple curve with endpoints $a$ and $b$ that separates the domains $D_i$ and $\mathbb{H} \backslash \overline{D}_i$.
\end{definition}

We use the notation $\nu_\kappa^{\gamma,\beta_0}$ because in Section~\ref{sec:reweight} we will extend to $\nu_\kappa^{\alpha,\beta}$ for general $\beta$. This extension captures the joint moment of  $|\psi'(i)|, |\varphi'(1)|$ mentioned at the beginning of Section~\ref{sec:welding}, which is our target. Proposition~\ref{prop:weld5} concerns the curve decorated quantum surface ${\rm LF}_\mathbb{H}^{(\gamma,i),(\beta_0,0)} \times \nu_\kappa^{\gamma,\beta_0}$. This topological configuration requires us to decompose $\mathcal{M}_{2, \bullet}^{\rm disk}(\gamma^2-2)$ into three pieces according to Lemma~\ref{lem:thin-thick}, then glue them to a sample from $\rm QD_{0,1}$. We prove in Lemma~\ref{lem:weld-decompose} that up to an explicit constant, the end result is equivalent to the conformal welding of four measures: ${\rm Weld}(\mathcal{M}_{2}^{\rm disk}(\gamma^2-2), \mathcal{M}_{2, \bullet}^{\rm disk}(2), \mathcal{M}_{2}^{\rm disk}(\gamma^2-2), \widetilde {\rm QD}_{0,3})$, as depicted in Figure~\ref{fig:weld-decompose}.   This measure is obtained by gluing the three boundary arcs of a sample from $\widetilde {\rm QD}_{0,3}$ with the right boundary arcs of samples from $\mathcal{M}_{2}^{\rm disk}(\gamma^2-2)$, $\mathcal{M}_{2, \bullet}^{\rm disk}(2)$, and $\mathcal{M}_{2}^{\rm disk}(\gamma^2-2)$. This welding process can be thought of as first sampling the four pieces from their product measure and then restrict to the event that their shared boundary lengths agree. This can be achieved rigorously by disintegrating the shared boundary lengths and gluing them one by one sequentially as described in detail at the beginning of Section~\ref{sec:bubble-weld}. The order of gluing does not affect the final outcome. 

\begin{lemma}
\label{lem:weld-decompose}
    When viewed as laws of quantum surfaces decorated with curves and marked points,
    $$
    {\rm Weld}(\mathcal{M}_{2, \bullet}^{\rm disk}(\gamma^2-2), {\rm QD}_{0,1}) = \big(\frac{4}{\gamma^2}-1\big)^2 {\rm Weld}(\mathcal{M}_{2}^{\rm disk}(\gamma^2-2), \mathcal{M}_{2, \bullet}^{\rm disk}(2), \mathcal{M}_{2}^{\rm disk}(\gamma^2-2), \widetilde {\rm QD}_{0,3})\,. 
    $$
\begin{figure}[H]
\centering
\includegraphics[width=8cm]{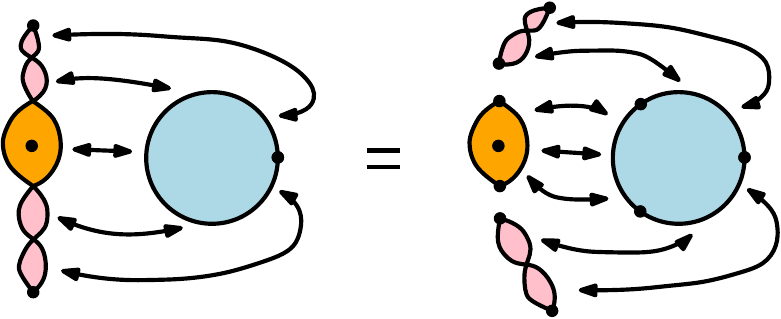}
\caption{Illustration of Lemma~\ref{lem:weld-decompose}. On the left-hand side, the union of the orange and pink regions corresponds to $\mathcal{M}_{2, \bullet}^{\rm disk}(\gamma^2-2)$ and the blue region corresponds to ${\rm QD}_{0,1}$. On the right-hand side, the pink regions correspond to $\mathcal{M}_{2}^{\rm disk}(\gamma^2-2)$, the orange region corresponds to $\mathcal{M}_{2, \bullet}^{\rm disk}(2)$, and the blue region corresponds to $\widetilde {\rm QD}_{0,3}$.}
\label{fig:weld-decompose}
\end{figure}
\end{lemma}
\begin{proof}
Using Lemma~\ref{lem:thin-thick} with $W=\gamma^2-2$, we have that ${\rm Weld}(\mathcal{M}_{2, \bullet}^{\rm disk}(\gamma^2-2), {\rm QD}_{0,1}) $ equals
\begin{equation*}
(\frac{4}{\gamma^2}-1)^2\int_{\mathbb{R}_+^3} {\rm Weld}(\mathcal{M}_{2}^{\rm disk}(\gamma^2-2;\ell_1) \times \mathcal{M}_{2, \bullet}^{\rm disk}(2;\ell_2) \times \mathcal{M}_{2}^{\rm disk}(\gamma^2-2;\ell_3), {\rm QD}_{0,1}(\ell_1 + \ell_2 + \ell_3)) d\ell_1 d\ell_2 d\ell_3\,.
\end{equation*}
For $\ell_1,\ell_2,\ell_3>0$, by Definition~\ref{def:QD}, the measure $\widetilde {\rm QD}_{0,3}(\ell_1, \ell_2, \ell_3)$ can be obtained from ${\rm QD}_{0,1}(\ell_1 + \ell_2 + \ell_3)$ by adding two marked points such that the three boundary arcs have quantum length $\ell_1,\ell_2,\ell_3$ counterclockwise. Combining this with the above identity and then applying the definition of conformal welding yields the desired lemma. \qedhere
\end{proof}
The following lemma describes the Liouville field for ${\rm Weld}(\mathcal{M}_2^{\rm disk}(\gamma^2-2), \widetilde {\rm QD}_{0,3}, \mathcal{M}_2^{\rm disk}(\gamma^2-2))$, where the welded surface is obtained by gluing the right boundary arcs of two samples from $\mathcal{M}_2^{\rm disk}(\gamma^2-2)$ with two boundary arcs of a sample from $\widetilde {\rm QD}_{0,3}$. From the proof we also get the precise law of  interfaces $\hat m$ but we will not need it.
\begin{lemma}
\label{lem:weld4}
There exists a constant $C_2 \in (0,\infty)$ such that when viewed as laws of quantum surfaces decorated with curves and marked points, we have:
    $$
   {\rm LF}_\mathbb{H}^{(\frac{2}{\gamma},0),(\frac{4}{\gamma}-\gamma,1), (\frac{2}{\gamma},\infty)} \times \widehat{m} = C_2 {\rm Weld}(\mathcal{M}_2^{\rm disk}(\gamma^2-2), \widetilde {\rm QD}_{0,3}, \mathcal{M}_2^{\rm disk}(\gamma^2-2))\,.
    $$
    Here, $\widehat{m}$ is a probability measure on two non-intersecting curves that connects $0, 1$, and $1, \infty$, respectively. The domains below the two curves correspond to $\mathcal{M}_2^{\rm disk}(\gamma^2-2)$, respectively, and the domain sandwiched between the two curves with $0,1,\infty$ on its boundary corresponds to $\widetilde {\rm QD}_{0,3}$; see Figure~\ref{fig:weld4}.
    \begin{figure}[H]
\centering
\includegraphics[width=8cm]{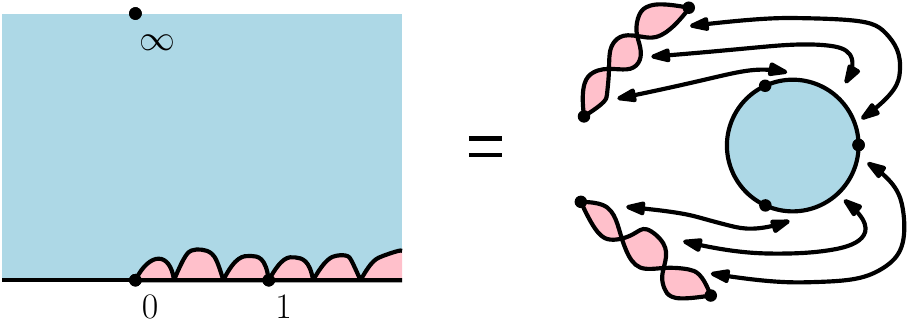}
\caption{Illustration of Lemma~\ref{lem:weld4}. The left-hand side shows the Liouville field with three boundary insertions, and the two random curves from $0$ to $1$, and $1$ to $\infty$, respectively. On the right-hand side, the two pink region corresponds to $\mathcal{M}_2^{\rm disk}(\gamma^2-2)$, and the blue region corresponds to $\widetilde {\rm QD}_{0,3}$.}
\label{fig:weld4}
\end{figure}
\end{lemma}

\begin{proof}
Lemma~\ref{lem:weld4} is a direct consequence of \cite[Theorem 1.1]{ASY22}. Recall  the definition of quantum triangle ${\rm QT}(W_1, W_2, W_3)$ from Definition~\ref{def:QT}.
For $W,W_1,W_2,W_3>0$ with $W_1 + 2 = W_2 + W_3$ and $\frac{\gamma^2}{2} \not \in  \{W_1,W_2,W_3,W+W_1,W+W_2 \}$, \cite[Theorem 1.1]{ASY22} asserts the following. There exists a constant $C(W,W_1,W_2,W_3) \in (0,\infty)$ such that
    $$
    {\rm QT}(W + W_1, W + W_2,W_3) \times m = C(W,W_1,W_2,W_3) {\rm Weld}(\mathcal{M}_2^{\rm disk}(W), {\rm QT}(W_1,W_2,W_3)).
    $$
    Here, $m$ is an explicit SLE type probability measure on a random curve from the $W+W_2$-weighted point to the $W+W_1$-weighted point. The welded surface is obtained by gluing the right boundary arc of a sample from $\mathcal{M}_2^{\rm disk}(W)$ with the boundary arc between the $W_1$- and $W_2$-weighted points of a sample from ${\rm QT}(W_1,W_2,W_3)$. In this proof we allow the constant $C$ to vary from place to place.

    Applying \cite[Theorem 1.1]{ASY22} twice, with $(W, W_1, W_2, W_3 ) = (\gamma^2-2,2,2,2)$ and $(\gamma^2-2, \gamma^2, 2, \gamma^2)$ respectively, we see that the law of the quantum surface for  ${\rm Weld}(\mathcal{M}_2^{\rm disk}(\gamma^2-2), {\rm QT}(2,2,2), \mathcal{M}_2^{\rm disk}(\gamma^2-2))$ equals $C{\rm QT}(\gamma^2, 2\gamma^2-2, \gamma^2)$.
 Recall from Definition~\ref{def:QT} and Lemma~\ref{lem:QT} that \begin{equation}
    \label{eq:weld5-12}
    {\rm LF}_\mathbb{H}^{(\frac{2}{\gamma},0),(\frac{4}{\gamma}-\gamma,1), (\frac{2}{\gamma},\infty)} = C {\rm QT}(\gamma^2, 2\gamma^2-2, \gamma^2) \quad \mbox{and} \quad \widetilde{\rm QD}_{0,3} = C {\rm QT}(2,2,2)\,.\end{equation}
Putting these together concludes the proof.
\end{proof}

We are now ready to describe the conformal welding for $ {\rm LF}_\mathbb{H}^{(\gamma,i),(\beta_0,0)} \times \nu_\kappa^{\gamma,\beta_0}$.
Define the conformal welding ${\rm Weld}({\rm LF}_\mathbb{H}^{(\frac{2}{\gamma},0), (\beta_0,1), (\frac{2}{\gamma} ,\infty)},\mathcal{M}_{1,2}^{\rm disk}(\gamma))$ as the law of the quantum surface obtained by gluing one boundary arc of a sample from $\mathcal{M}_{1,2}^{\rm disk}(\gamma)$ with the boundary arc between two $\frac{2}{\gamma}$-singularity points of a sample from ${\rm LF}_\mathbb{H}^{(\frac{2}{\gamma},0), (\beta_0,1), (\frac{2}{\gamma} ,\infty)}$. See the right hand side of Figure~\ref{fig:weld5} for an illustration.

\begin{proposition}
\label{prop:weld5}
    When viewed as laws of quantum surfaces decorated with a curve and marked points,$$
    {\rm LF}_\mathbb{H}^{(\gamma,i),(\beta_0,0)} \times \nu_\kappa^{\gamma,\beta_0} = \frac{2C_1}{\gamma C_2} {\rm Weld}({\rm LF}_\mathbb{H}^{(\frac{2}{\gamma},0), (\beta_0,1), (\frac{2}{\gamma} ,\infty)}, \mathcal{M}_{1,2}^{\rm disk}(\gamma))
    $$
    with $C_1$ from Proposition~\ref{prop:weld6} and $C_2$ from Lemma~\ref{lem:weld4}. Here $(D_i, i, b, a)$ corresponds to $\mathcal{M}_{1,2}^{\rm disk}(\gamma)$, and $(\mathbb{H} \backslash \overline{D}_i, a, 0, b)$ corresponds to ${\rm LF}_\mathbb{H}^{(\frac{2}{\gamma},0), (\beta_0,1), (\frac{2}{\gamma} ,\infty)}$, where $D_i, a,b$ are as in Definition~\ref{def:simple}. See Figure~\ref{fig:weld5}.
    \begin{figure}[H]
\centering
\includegraphics[width=8cm]{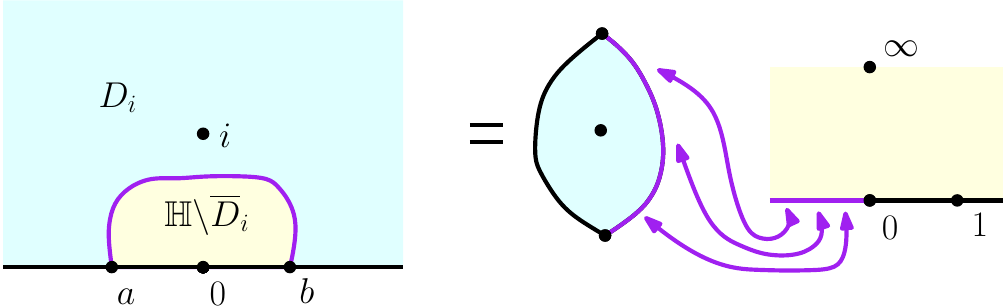}
\caption{Illustration of Proposition~\ref{prop:weld5}. The purple curve on the left-hand side corresponds to a sample from $\nu_\kappa^{\gamma,\beta_0}$. On the right-hand side, the light cyan region corresponds to $\mathcal{M}_{1,2}^{\rm disk}(\gamma)$, and the light yellow region corresponds to ${\rm LF}_\mathbb{H}^{(\frac{2}{\gamma},0), (\beta_0,1), (\frac{2}{\gamma} ,\infty)}$.}
\label{fig:weld5}
\end{figure}
\end{proposition}
\begin{proof}
    By applying Lemma~\ref{lem:weld4} to the welded surface in Lemma~\ref{lem:weld-decompose}, we obtain that
    $$
    {\rm Weld}(\mathcal{M}_{2, \bullet}^{\rm disk}(\gamma^2-2), {\rm QD}_{0,1}) = \frac{1}{C_2}\big(\frac{4}{\gamma^2}-1\big)^2 {\rm Weld}(\mathcal{M}_{2, \bullet}^{\rm disk}(2), {\rm LF}_\mathbb{H}^{(\frac{2}{\gamma},0),(\frac{4}{\gamma}-\gamma,1), (\frac{2}{\gamma},\infty)} \times \widehat{m})\,.  
    $$
    Combining with Proposition~\ref{prop:weld3} and Lemma~\ref{lem:embedding-2}, we have:
    $$
    {\rm LF}_\mathbb{H}^{(\gamma,i),(\beta_0,0)} \times \overline{\nu}_\kappa = \frac{2C_1}{\gamma C_2} {\rm Weld}(\mathcal{M}_{1,2}^{\rm disk}(\gamma), {\rm LF}_\mathbb{H}^{(\frac{2}{\gamma},0),(\frac{4}{\gamma}-\gamma,1), (\frac{2}{\gamma},\infty)} \times \widehat{m})\,.  
    $$
Forgetting $\widehat{m}$ yields the desired result, since by Definition~\ref{def:simple}, the measure $\nu_\kappa^{\gamma, \beta_0}$ is the marginal law of $\overline{\nu}_\kappa$ on the interface between the two welded surfaces. \qedhere
    \end{proof}

\subsection{Proof of Proposition~\ref{prop:reweight}: the final conformal welding result}
\label{sec:reweight}

Recall $D_i, a,b,$ and $\nu^{\gamma,\beta_0}_\kappa$ from Definition~\ref{def:simple}, where $\beta_0 = \frac{4}{\gamma} - \gamma$. Consider two conformal maps, $\psi : (\mathbb{H}, i,0,x) \rightarrow (D_i, i,b,a)$ and $\varphi: ( \mathbb{H}, 0, 1 ,\infty) \rightarrow (\mathbb{H} \backslash \overline{D}_i, a,0,b)$ as in Figure~\ref{fig:psi}. 
 Recall from~\eqref{eq:KPZ} the scaling exponent $\Delta_\alpha = \frac{\alpha}{2}(Q-\frac{\alpha}{2})$ for $\alpha \in \mathbb{R}$. For any $\alpha, \beta \in \mathbb{R}$, let 
\begin{equation}
\label{eq:def-beta}
\nu_\kappa^{\alpha, \beta}:= |\psi'(i)|^{2 \Delta_\alpha - 2 \Delta_\gamma} \times |\varphi'(1)|^{\Delta_\beta - \Delta_{\beta_0}} \cdot \nu_\kappa^{\gamma, \beta_0}.
\end{equation}
Since the moment $\nu_\kappa[|\psi'(i)|^a |\varphi'(1)|^b]$ equals $\nu_\kappa^{\gamma,\beta_0}[|\psi'(i)|^a |\varphi'(1)|^b]$, by setting $a=2 \Delta_\alpha - 2 \Delta_\gamma$ and $b=\Delta_\beta - \Delta_{\beta_0}$, we see that $\nu_\kappa[|\psi'(i)|^a |\varphi'(1)|^b]$  equals the total mass of $\nu_\kappa^{\alpha, \beta}$.
In this section, we prove the following conformal welding result producing $\nu_\kappa^{\alpha, \beta}$, based on which we will compute $\nu_\kappa[|\psi'(i)|^a |\varphi'(1)|^b]$  in Section~\ref{sec:solve} and prove Theorem~\ref{thm:moment}.

\begin{proposition}
\label{prop:reweight}
For any $\alpha,\beta \in \mathbb{R}$, Proposition~\ref{prop:weld5} holds with 
${\rm LF}_\mathbb{H}^{(\gamma,i),(\beta_0,0)}$,
$\nu_\kappa^{\gamma,\beta_0}$,
${\rm LF}_\mathbb{H}^{(\frac{2}{\gamma},0), (\beta_0,1), (\frac{2}{\gamma} ,\infty)}$,
 and
$\mathcal{M}_{1,2}^{\rm disk}(\gamma)$ 
replaced by ${\rm LF}_\mathbb{H}^{(\alpha,i),(\beta,0)}$, $\nu_\kappa^{\alpha,\beta}$, 
${\rm LF}_\mathbb{H}^{(\frac{2}{\gamma},0), (\beta,1), (\frac{2}{\gamma} ,\infty)}$, and
$\mathcal{M}_{1,2}^{\rm disk}(\alpha)$, respectively. Namely,
\begin{equation}\label{eq:key-welding}
    {\rm LF}_\mathbb{H}^{(\alpha,i),(\beta,0)} \times \nu_\kappa^{\alpha,\beta} = \frac{2C_1}{\gamma C_2} {\rm Weld}({\rm LF}_\mathbb{H}^{(\frac{2}{\gamma},0), (\beta,1), (\frac{2}{\gamma} ,\infty)}, \mathcal{M}_{1,2}^{\rm disk}(\alpha)) 
\end{equation}
with the constants $C_1$ from Proposition~\ref{prop:weld6} and $C_2$ from Lemma~\ref{lem:weld4}. 
\end{proposition}
 
Our proof follows a strategy originating from  \cite{AHS21} and refined in~\cite{ARS21}. We first prove Proposition~\ref{prop:reweight} for $\alpha=\gamma$ and a general $\beta$, based on Proposition~\ref{prop:weld5}. We need the following two lemmas.

\begin{lemma}
\label{lem:reweight1}
For $(\phi,\eta)$ sampled from ${\rm LF}_\mathbb{H}^{(\gamma,i),(\beta_0,0)} \times \nu_\kappa^{\gamma,\beta_0}$, let $X = (\varphi^{-1})_* \phi = \phi \circ \varphi + Q\log|\varphi '|$ denote the push-back of $\phi$ under the map $\varphi$ so that $(\mathbb{H}, X, 0, 1, \infty) / \sim_\gamma = (\mathbb{H} \backslash \overline{D}_i, \phi, a, 0, b) / \sim_\gamma$. Given $\epsilon\in(0,\frac12)$, let $X_\epsilon(1)$ be the average of $X$ over $\partial B_\epsilon(1) \cap \mathbb{H}$. For any $\beta \in \mathbb{R}$, we have
    \begin{align}\label{eq:Girsanov1}
    \iint F(\phi) \epsilon^{\frac{\beta^2 - \beta_0^2}{4}} e^{\frac{(\beta - \beta_0)}{2} X_\epsilon(1)} {\rm LF}_\mathbb{H}^{(\gamma,i),(\beta_0,0)} (d\phi) \nu_\kappa^{\gamma,\beta_0}(d\eta)
    = \iint F(\phi) {\rm LF}_\mathbb{H}^{(\gamma,i),(\beta,0)}(d\phi) \nu_\kappa^{\gamma,\beta}(d\eta)\,,
    \end{align}
where $F$ is any non-negative measurable function depending only on the restriction of $\phi$ to $\mathbb{H} \backslash \varphi(B_\epsilon(1) \cap \mathbb{H})$.
\end{lemma}
\begin{proof}
We first treat $\eta$ as a fixed curve and only deal with the integration with respect to the field $\phi$. Let $\theta_\epsilon$ be the uniform probability measure on $\partial B_\epsilon(1) \cap \mathbb{H}$ and $\widehat \theta_\epsilon = \varphi_* \theta_\epsilon$. Lemma~\ref{lem:reweight1} follows from the following two identities:
    \begin{equation}
    \label{eq:lem3.23-1}
    e^{\frac{\beta-\beta_0}{2}X_\epsilon(1)} = |\varphi'(1)|^{\frac{(\beta-\beta_0)Q}{2}} e^{\frac{\beta-\beta_0}{2} (\phi, \widehat \theta_\epsilon)};
    \end{equation}
    \begin{equation}
    \label{eq:lem3.23-2}
    \int F(\phi) e^{\frac{\beta-\beta_0}{2} (\phi, \widehat \theta_\epsilon)}
    {\rm LF}_\mathbb{H}^{(\gamma,i),(\beta_0,0)}(d\phi) = |\epsilon \varphi'(1)|^{-\frac{\beta^2-\beta_0^2}{4}} \int F(\phi){\rm LF}_\mathbb{H}^{(\gamma,i),(\beta,0)}(d\phi)\,.
    \end{equation}
    To prove \eqref{eq:lem3.23-1}, we observe that $\log |\varphi'|$ is harmonic, and therefore $(X,\theta_\epsilon) = (\phi \circ \varphi + Q\log|\varphi'|, \theta_\epsilon) = (\phi, \widehat \theta_\epsilon) + Q \log|\varphi'(1)|$. Next, we prove \eqref{eq:lem3.23-2}. Let $h$ be sampled from $P_{\mathbb{H}}$, $u(z) = - 2Q \log|z|_+ + \frac{\beta_0}{2} G_{\mathbb{H}}(z,0) + \gamma G_{\mathbb{H}}(z,i)$, and $\widetilde h(z) = h(z) +u(z)$. By Definition~\ref{def:LF} and Girsanov's theorem, we have
    \begin{align*}
        &\quad \int F(\phi) e^{\frac{\beta-\beta_0}{2}(\phi,\widehat \theta_\epsilon)} {\rm LF}_\mathbb{H}^{(\gamma,i),(\beta_0,0)}(d\phi) = C_\mathbb{H}^{(\gamma,i),(\beta_0,0)} \int F(\widetilde h + c)e^{\frac{\beta-\beta_0}{2}(\widetilde h + c, \widehat \theta_\epsilon)} e^{(\frac{\beta_0}{2} + \gamma-Q)c}dP_\mathbb{H}(h)dc  \\
        &=C_\mathbb{H}^{(\gamma,i),(\beta_0,0)} e^{\frac{\beta-\beta_0}{2}(u, \widehat \theta_\epsilon)} e^{\frac{(\beta- \beta_0)^2}{8} \mathbb{E}(h, \widehat \theta_\epsilon)^2 }\int F(\widetilde h + \frac{\beta-\beta_0}{2}\mathbb{E}[h(\cdot)(h, \widehat \theta_\epsilon)] + c)e^{(\frac{\beta}{2} + \gamma-Q)c}dP_\mathbb{H}(h)dc\,.
    \end{align*}
Let $\delta = (2\log|\cdot|_+, \widehat \theta_\epsilon)$. Using the identity that $\mathbb{E}[h(z) (h,\widehat \theta_\epsilon)] = G_{\mathbb{H}}(z,0) + \delta$ for $z \not \in \varphi(B_\epsilon(1) \cap \mathbb{H})$, and changing the variable $c' = \frac{\beta-\beta_0}{2}\delta +c$, we see that $\int F(\phi) e^{\frac{\beta-\beta_0}{2}(\phi,\widehat \theta_\epsilon)} {\rm LF}_\mathbb{H}^{(\gamma,i),(\beta_0,0)}(d\phi)$ equals
    \begin{align*}
        C_\mathbb{H}^{(\gamma,i),(\beta_0,0)} e^{\frac{\beta-\beta_0}{2}(u, \widehat \theta_\epsilon)} e^{\frac{(\beta-\beta_0)^2}{8} \mathbb{E}(h, \widehat \theta_\epsilon)^2 } e^{-(\frac{\beta}{2} + \gamma-Q)\frac{\beta-\beta_0}{2}\delta} \int F(\widetilde h + \frac{\beta-\beta_0}{2}G_\mathbb{H}(\cdot,i)  + c')e^{(\frac{\beta}{2} + \gamma-Q)c
        '}dP_\mathbb{H}(h)dc'.
    \end{align*}
Using the following identities
\begin{align*}
        &(G_{\mathbb{H}}(\cdot,0),\widehat \theta_\epsilon) = -2 \log|\epsilon \varphi'(1)| + \delta\,, \quad (G_{\mathbb{H}}(\cdot,i),\widehat \theta_\epsilon) = G_{\mathbb{H}}(0,i) +\delta\,,\\
        &\mathbb{E}(h,\widehat \theta_\epsilon)^2 = -2\log|\epsilon \varphi'(1)| + 2\delta\,,\qquad C_\mathbb{H}^{(\gamma,i),(\beta_0,0)}/C_\mathbb{H}^{(\gamma,i),(\beta,0)} = e^{\frac{\gamma(\beta_0-\beta)}{2} G_\mathbb{H}(0,i)},
\end{align*}
we can verify that the constant before the above integral equals to $|\epsilon \varphi'(1)|^{-\frac{\beta^2-\beta_0^2}{4}}   C_\mathbb{H}^{(\gamma,i),(\beta,0)}$. Using Definition~\ref{def:LF} again, we conclude~\eqref{eq:lem3.23-2}.

Combining \eqref{eq:lem3.23-1}, \eqref{eq:lem3.23-2}, and $|\varphi'(1)|^{\Delta_\beta -\Delta_{\beta_0}} \cdot \nu_\kappa^{\gamma,\beta_0}= \nu_\kappa^{\gamma,\beta}$, the left-hand side of~\eqref{eq:Girsanov1} equals
    $$
    \iint F(\phi)  |\varphi'(1)|^{\Delta_\beta -\Delta_{\beta_0}} {\rm LF}_\mathbb{H}^{(\gamma,i),(\beta,0)}(d\phi) \nu_\kappa^{\gamma,\beta_0}(d\eta) = \iint F(\phi) {\rm LF}_\mathbb{H}^{(\gamma,i),(\beta,0)}(d\phi) \nu_\kappa^{\gamma,\beta}(d\eta)\,.\qedhere
    $$
\end{proof}
\begin{lemma}
\label{lem:reweight2}
Let $\phi$ be a sample from ${\rm LF}_\mathbb{H}^{(\frac{2}{\gamma},0), (\beta_0,1), (\frac{2}{\gamma} ,\infty)}$. For any $\beta \in \mathbb{R}$, $\epsilon \in (0,\frac{1}{2})$, we have
$$
\int f(\phi) \times \epsilon^{\frac{1}{4}(\beta^2 - \beta_0^2)} e^{\frac{\beta - \beta_0}{2} \phi_\epsilon(1)} {\rm LF}_\mathbb{H}^{(\frac{2}{\gamma},0), (\beta_0,1), (\frac{2}{\gamma} ,\infty)}(d\phi) = \int f(\phi) {\rm LF}_\mathbb{H}^{(\frac{2}{\gamma},0), (\beta,1), (\frac{2}{\gamma} ,\infty)}(d\phi),
$$
where $f$ is any non-negative measurable function depending only on the restriction of $\phi$ to $\mathbb{H} \backslash B_\epsilon(1)$.
\end{lemma}
\begin{proof}
 The proof is a simpler application of Girsanov theorem than \eqref{eq:lem3.23-2}. We omit the detail.
\end{proof}
\begin{proof}[Proof of Proposition~\ref{prop:reweight} when $\alpha =\gamma$ and $\beta\in \mathbb R$]
Fix $\eps\in (0,\frac12)$. By Proposition~\ref{prop:weld5}, the left-hand side of~\eqref{eq:Girsanov1} equals
    \begin{align*}
     \frac{2C_1}{\gamma C_2} {\rm Weld}({\rm LF}_\mathbb{H}^{(\frac{2}{\gamma},0), (\beta_0,1), (\frac{2}{\gamma} ,\infty)}, \mathcal{M}_{1,2}^{\rm disk}(\gamma))[F(\phi) \epsilon^{\frac{\beta^2 - \beta_0^2}{4}} e^{\frac{(\beta - \beta_0)}{2} X_\epsilon(1)}]\,.
    \end{align*}
    By Lemma~\ref{lem:reweight2}, this equals
    $\frac{2C_1}{\gamma C_2} {\rm Weld}({\rm LF}_\mathbb{H}^{(\frac{2}{\gamma},0), (\beta,1), (\frac{2}{\gamma} ,\infty)}, \mathcal{M}_{1,2}^{\rm disk}(\gamma))[F(\phi) ]$. 
    Comparing this expression with the right-hand side of~\eqref{eq:Girsanov1}  and sending $\eps$ to 0, we get
    \begin{equation*}
    \label{eq:prop3.19-2}
    {\rm LF}_\mathbb{H}^{(\gamma,i),(\beta,0)}(d\phi) \times \nu_\kappa^{\gamma,\beta}(d\eta) = \frac{2C_1}{\gamma C_2} {\rm Weld}({\rm LF}_\mathbb{H}^{(\frac{2}{\gamma},0), (\beta,1), (\frac{2}{\gamma} ,\infty)}, \mathcal{M}_{1,2}^{\rm disk}(\gamma))\,. \qedhere
    \end{equation*}
\end{proof}

We now prove Proposition~\ref{prop:reweight} with a general $\alpha$ based on the special case $\alpha=\gamma$ proved above.
\begin{proof}[Proof of Proposition~\ref{prop:reweight}]
The remaining argument from $\alpha=\gamma$ to a general bulk insertion $\alpha$ is almost identical with the above argument from $\beta=\beta_0$  to a general boundary insertion $\beta$. So we only point out the difference. First
based on Girsanov theorem and $|\psi'(i)|^{2\Delta_\alpha -2\Delta_{\gamma}} \cdot \nu_\kappa^{\gamma,\beta}= \nu_\kappa^{\alpha,\beta}$, we have the following analog of~\eqref{eq:Girsanov1}.
For $(\phi,\eta)$ sampled from ${\rm LF}_\mathbb{H}^{(\gamma,i),(\beta,0)} \times \nu_\kappa^{\gamma,\beta}$, let $X = (\psi^{-1})_* \phi = \phi\circ \psi+ Q\log|\psi'|$ denote the push-back of $\phi$ under the map $\psi$ so that $(\mathbb{H}, X, i,0,x) / \sim_\gamma = (D_i, \phi, i,b,a) / \sim_\gamma$. Given $\epsilon\in(0,\frac12)$, let $X_\epsilon(i)$ be the average of $X$ over $\partial B_\epsilon(i)$. For any $\beta \in \mathbb{R}$, we have
    \begin{align}\label{eq:Girsanov2}
    \iint F(\phi)  \epsilon^{\frac{\alpha^2 - \gamma^2}{2}} e^{(\alpha - \gamma) X_\epsilon(i)}  {\rm LF}_\mathbb{H}^{(\gamma,i),(\beta,0)} (d\phi) \nu_\kappa^{\gamma,\beta}(d\eta)
    = \iint F(\phi)  {\rm LF}_\mathbb{H}^{(\alpha,i),(\beta,0)}(d\phi) \nu_\kappa^{\alpha,\beta}(d\eta)\,.
    \end{align}
where  $F$ is any non-negative measurable function depending only on the restriction of $\phi$ to $\mathbb{H} \backslash \psi(B_\epsilon(i))$.
By the case of Proposition~\ref{prop:reweight} with $\alpha=\gamma$ proved above, the left-hand side of~\eqref{eq:Girsanov2} equals 
\begin{align}
\frac{2C_1}{\gamma C_2} {\rm Weld}({\rm LF}_\mathbb{H}^{(\frac{2}{\gamma},0), (\beta,1), (\frac{2}{\gamma} ,\infty)}, \mathcal{M}_{1,2}^{\rm disk}(\gamma))[F(\phi) \epsilon^{\frac{\alpha^2 - \gamma^2}{2}} e^{(\alpha - \gamma) X_\epsilon(i)}]\,. \label{eq:eps-bulk}
    \end{align}
We analyze~\eqref{eq:eps-bulk} by the following analog of Lemma~\ref{lem:reweight2}.
For fixed $x \in \mathbb{R}$, we have
\begin{equation}\label{eq:bulk-Girsanov}
    \int f(\phi)\times \epsilon^{\frac{\alpha^2 - \gamma^2}{2}} e^{(\alpha - \gamma) \phi_\epsilon(i)} {\rm LF}_\mathbb{H}^{(\gamma, i), (\gamma,0), (\gamma,x)} (d\phi)  = \int f(\phi) {\rm LF}_\mathbb{H}^{(\alpha, i), (\gamma,0), (\gamma,x)} (d\phi) \,,
\end{equation}where $f$ is any non-negative measurable function that depends only on the restriction of $\phi$ to $\mathbb{H} \backslash B_\epsilon(i)$.
Recall the definition of $\mathcal{M}_{1,2}^{\rm disk}(\alpha)$ from Definition~\ref{def:M12}. 
By integrating over $x\in \mathbb R$ in~\eqref{eq:bulk-Girsanov},  we see that \eqref{eq:eps-bulk} equals
$\frac{2C_1}{\gamma C_2} {\rm Weld}({\rm LF}_\mathbb{H}^{(\frac{2}{\gamma},0), (\beta,1), (\frac{2}{\gamma} ,\infty)}, \mathcal{M}_{1,2}^{\rm disk}(\alpha))[F(\phi) ]$. Comparing this expression with the right-hand side of~\eqref{eq:Girsanov2} and then sending $\epsilon$ to 0 yields the desired result.
\end{proof}

\section{Derivation of the SLE bubble exponent}
\label{sec:solve}
In this section we prove Theorem~\ref{thm:moment} based on  Proposition~\ref{prop:reweight}. The calculation relies on  the boundary structure constants for Liouville CFT. We recall their definitions in Section~\ref{subsec:GH}. In Section~\ref{sec:sol-gab}, we express the joint moment $\nu_\kappa[|\psi'(i)|^a |\varphi'(1)|^b]$ in terms of these structure constants. In Section~\ref{subsec:QD-id} we present a few exact results on quantum disks. 
We then carry out some   intermediate calculations in Sections~\ref{sec:sol-C1} and ~\ref{subsec:H} based on the conformal welding results in Section~\ref{sec:welding}, and the exact solvability of both the structure constants and quantum disks.  Finally, we prove Theorem~\ref{thm:moment} in Section~\ref{sec:analytic}.

\subsection{Boundary structure constants for Liouville CFT}\label{subsec:GH}
The boundary structure constants for Liouville CFT describe the correlation functions of Liouville CFT on the upper half plane where there is one bulk and one boundary insertions, or there are three boundary insertions. The most general case gives the distributional information of both the quantum area and boundary length, which is determined in~\cite{ARS21,ARSZ23}. For our purpose it suffices to use those that only encode the boundary length information, which is determined in \cite{RZ22}. 

We will frequently use the following notation. Suppose $M$ is the law of a quantum surface or a Liouville field. Let $\mathcal L$ be the total quantum length of its sample. We will right $|M(\ell)|$ as the density of $\mathcal L$ under $M$.  Namely, $M[f(\mathcal L)] = \int_0^\infty f(\ell)|M(\ell)| d\ell$ for any positive measurable function. See $| {\rm LF}_\mathbb{H}^{(\alpha,i),(\beta,0)}(\ell)|$ in Lemma~\ref{lem:LF1} for an example. Sometimes we will specify a few boundary arcs  with boundary lengths  $\mathcal L_1,\mathcal L_2,\cdots$. We write $|M(\ell_1,\ell_2,\cdots)|$ as the  joint density of their boundary length. See  $|\mathcal{M}_{1,2}^{\rm disk}(\alpha; \ell_1, \ell_2)|$ in Lemma~\ref{lem:LF2} for an example. 

We first introduce the bulk-boundary structure constant.
\begin{lemma}
\label{lem:LF1}
For $\beta<Q$ and $\frac{\gamma}{2} - \alpha <\frac{\beta}{2}<\alpha$, let $\phi$ be a sample from the Liouville field ${\rm LF}_\mathbb{H}^{(\alpha,i),(\beta,0)}$, where there is one bulk and one boundary insertion. Let $| {\rm LF}_\mathbb{H}^{(\alpha,i),(\beta,0)}(\ell)|$ be the density of the total quantum boundary length $\mathcal L=\nu_\phi(-\infty,\infty)$. Namely
\begin{equation}\label{eq:density}
    {\rm LF}_\mathbb{H}^{(\alpha,i),(\beta,0)}[f(\mathcal L)] =\int_0^\infty | {\rm LF}_\mathbb{H}^{(\alpha,i),(\beta,0)}(\ell)  | f(\ell) d\ell\,.
\end{equation}
Then, there exists an explicit function $\overline{G}(\alpha,\beta)$ that can be found in \cite[Eq.\ (1.28)]{RZ22} such that  \begin{equation}
\label{eq:def-G}
|{\rm LF}_\mathbb{H}^{(\alpha,i),(\beta,0)}(\ell))| = \frac{2}{\gamma}2^{-\frac{\alpha^2}{2}}\overline{G}(\alpha,\beta) \ell^{\frac{2}{\gamma}(\frac{\beta}{2} + \alpha - Q)-1}.
\end{equation}
\end{lemma}
\begin{proof}
This is one of the main results in~\cite{RZ22}. This formulation is from \cite[Lemma 7.8]{Wu23}.
\end{proof}

By~\eqref{eq:def-G}, if $\beta<Q$, $\frac{\gamma}{2} - \alpha <\frac{\beta}{2} <\alpha$, and in addition $\frac{\beta}{2}+\alpha -Q>0$, then for $\mu>0$ we have
    \begin{equation}
        \label{eq:prop4.14-3}
        {\rm LF}_\mathbb{H}^{(\alpha,i),(\beta,0)}[e^{-\mu \mathcal L}] = \frac{2}{\gamma}2^{-\frac{\alpha^2}{2}}\overline{G}(\alpha,\beta) \Gamma \Big{(}\frac{2}{\gamma}(\frac{\beta }{2}+ \alpha - Q)\Big{)} \mu^{\frac{2}{\gamma}(Q - \alpha - \frac{\beta}{2})}.
    \end{equation} 
In the literature, $ {\rm LF}_\mathbb{H}^{(\alpha,i),(\beta,0)}[e^{-\mu \mathcal L}]$ is often referred as the one-bulk-one-boundary structure constant without bulk potential, and $\mu$ is the so-called cosmological constant. The function $\overline G$ is referred to as the normalized structure constant. See e.g.~\cite{RZ22}.
We do not recall the full expression of $\overline{G}(\alpha,\beta)$ but only collect the  cases that we need here. 
\begin{lemma}
\label{lem:LF4}
For any $\alpha \in (\frac{\gamma}{2},Q)$, we have
$$
\overline G(\alpha,0) = \bigg(\frac{2^{-\frac{\gamma\alpha}{2}} 2 \pi}{\Gamma(1-\frac{\gamma^2}{4})}\bigg)^\frac{2(Q-\alpha)}{\gamma} \Gamma(\frac{\gamma\alpha}{2}-\frac{\gamma^2}{4})\,, \qquad \overline G(\alpha,\gamma)=\frac{1}{\pi}\overline G(\alpha,0)\,, \quad \mbox{and}
$$
$$
\overline G(\alpha,\frac{4}{\gamma}-\gamma)=\frac{2^{3-\frac{\gamma^2}{2}-\frac{4}{\gamma^2}}}{\pi^{\frac{4}{\gamma^2}-1}} \frac{\Gamma(1-\frac{\gamma^2}{4})^{\frac{4}{\gamma^2}-1} \Gamma(\frac{\gamma^2}{2}-1)}{\Gamma(2-\frac{4}{\gamma^2})\Gamma(\frac{\gamma^2}{4})}\times \frac{\Gamma(\frac{2\alpha}{\gamma}-\frac{4}{\gamma^2}) \Gamma(\frac{\gamma\alpha}{2}+1-\frac{\gamma^2}{2})}{\Gamma(\frac{2\alpha}{\gamma}-1)\Gamma(\frac{\gamma\alpha}{2}-1)}\overline G(\alpha,0)\,.
$$  
\end{lemma}
\begin{proof}
This follows from Theorem 1.7 of \cite{RZ22}.  Indeed, $\overline G(\alpha,0)$ equals  $\overline U(\alpha)$ defined there whose expression is as given above.
     The values of $\overline G(\alpha,\beta)/\overline G(\alpha,0)$ for $\beta \in \{\frac{4}{\gamma}-\gamma,\gamma\}$ can be calculated using the shift equation for $\overline G(\alpha,\beta)$ as described in
     Proposition 2.1 of \cite{RZ22}.
\end{proof}

Recall $\mathcal{M}_{1,2}^{\rm disk}(\alpha)$ from Definition~\ref{def:M12}. A sample from $\mathcal{M}_{1,2}^{\rm disk}(\alpha)$ has one bulk marked point and two boundary ones. Let $|\mathcal{M}_{1,2}^{\rm disk}(\alpha; \ell_1, \ell_2)|$ be the joint density of the quantum lengths of the two boundary arcs. Namely, let $\mathcal L_1$ and $\mathcal L_2$ be the two boundary arc lengths of a sample from $\mathcal{M}_{1,2}^{\rm disk}(\alpha)$. Then
\begin{equation}\label{eq:def-M12-length}
    \mathcal{M}_{1,2}^{\rm disk}(\alpha) [f(\mathcal L_1,\mathcal L_2)] = \iint_0^\infty f(\ell_1,\ell_2)|\mathcal{M}_{1,2}^{\rm disk}(\alpha; \ell_1, \ell_2)|  d\ell_1d \ell_2
\end{equation}  
for any positive measurable function $f$. Our next lemma gives $|\mathcal{M}_{1,2}^{\rm disk}(\alpha; \ell_1, \ell_2)| $ in terms of $\overline G$.

\begin{lemma}
\label{lem:LF2}
For any $\alpha>\frac{\gamma}{2}$ and $\ell_1,\ell_2>0$, we have
$$
|\mathcal{M}_{1,2}^{\rm disk}(\alpha; \ell_1, \ell_2)| = \frac{2}{\gamma}2^{-\frac{\alpha^2}{2}}\overline{G}(\alpha,\gamma) (\ell_1+\ell_2)^{\frac{2(\alpha-Q)}{\gamma}}.
$$
\end{lemma}
\begin{proof}
By Lemma~\ref{lem:typ-pt}, we have ${\rm LF}_{\mathbb{H}}^{(\alpha,i),(\gamma,0)}(d\phi) \nu_\phi(dx) = {\rm LF}_{\mathbb{H}}^{(\alpha,i),(\gamma,0),(\gamma,x)}(d\phi) dx$. From Definition~\ref{def:M12}  of $\mathcal{M}_{1,2}^{\rm disk}(\alpha)$, we see that $\mathcal{M}_{1,2}^{\rm disk}(\alpha)$ can be obtained by first sampling $\nu_\phi(\mathbb{R}) {\rm LF}^{(\alpha, i),(\gamma,0)}_{\mathbb{H}}(d\phi)$ and then sampling the point $x$ independently according to a probability measure proportional to $\nu_\phi$.  Therefore, 
$
|\mathcal{M}_{1,2}^{\rm disk}(\alpha; \ell_1, \ell_2)| = (\ell_1+\ell_2) |{\rm LF}^{(\alpha, i),(\gamma,0)}_{\mathbb{H}}(\ell_1+\ell_2)| \times \frac{1}{\ell_1+\ell_2}$. By Lemma~\ref{lem:LF1}, we conclude.\qedhere 
\end{proof}

We now introduce the three-point boundary structure constant $\overline H$. It describes the joint law of the quantum lengths of the three boundary arcs for a Liouville field on $\bbH$ with three boundary insertions. We only need the case where the three insertions are 
$\frac{2}{\gamma}, \beta, \frac{2}{\gamma}$ located at $0,1,\infty$ where $\beta\in (0,\gamma)\cup (\gamma,Q)$.
Hence we give a formal description in this case.
\begin{lemma}
\label{lem:LF3}
 Given $\beta \in (0, \gamma) \cup (\gamma, Q)$, consider a sample  from ${\rm LF}_\mathbb{H}^{(\frac{2}{\gamma},0), (\beta,1), (\frac{2}{\gamma} ,\infty)}$. Let $\mathcal L_{13}$ represent the quantum length of $(-\infty,0)$, $\mathcal L_{12}$ represent the quantum length of $(0,1)$, and $\mathcal L_{23}$ represent the quantum length of $(1,\infty)$. There exists an explicit function $\overline H^{(\frac{2}{\gamma}, \beta, \frac{2}{\gamma})}_{(\mu_1,\mu_2,\mu_3)}$ whose value can be found in Equations (1.26) and (1.30) of \cite{RZ22} such that given $\mu_1,\mu_2,\mu_3 \geq 0$ with $\mu_1+\mu_2+\mu_3 >0$, we have:
\begin{align*}
    {\rm LF}_\mathbb{H}^{(\frac{2}{\gamma},0), (\beta,1), (\frac{2}{\gamma} ,\infty)}[e^{-\mu_1 \mathcal L_{13} - \mu_2 \mathcal L_{12} - \mu_3 \mathcal L_{23} }]&= \frac{2}{\gamma} \Gamma(\frac{\beta}{\gamma} - 1) \overline H^{(\frac{2}{\gamma}, \beta, \frac{2}{\gamma})}_{(\mu_1,\mu_2,\mu_3)}\qquad \textrm{for }\beta\in (\gamma,Q);\\
    {\rm LF}_\mathbb{H}^{(\frac{2}{\gamma},0), (\beta,1), (\frac{2}{\gamma} ,\infty)}[e^{-\mu_1 \mathcal L_{13} - \mu_2 \mathcal L_{12} - \mu_3 \mathcal L_{23} } - 1 ] &= \frac{2}{\gamma} \Gamma(\frac{\beta}{\gamma} - 1) \overline H^{(\frac{2}{\gamma}, \beta, \frac{2}{\gamma})}_{(\mu_1,\mu_2,\mu_3)} \qquad \textrm{for }\beta\in (0,\gamma).
\end{align*}
\end{lemma}

\begin{proof} 
    We claim that for any $\beta \in (0,Q)$, the law of $\mu_1 \mathcal L_{13} + \mu_2 \mathcal L_{12} + \mu_3 \mathcal L_{23}$ is  $ \frac{2}{\gamma} \overline H^{(\frac{2}{\gamma}, \beta, \frac{2}{\gamma})}_{(\mu_1,\mu_2,\mu_3)} l^{\frac{\beta}{\gamma} - 2} \mathbbm{1}_{l>0} dl$. 
    For the case $\mu_1 = \mu_3 = 0$ and $\mu_2 = 1$, this density formula is derived in \cite[Proposition 2.23]{ASY22}. The same gives that the density of $\mu_1 \mathcal L_{13} + \mu_2 \mathcal L_{12} + \mu_3 \mathcal L_{23}$ is given by: $\mathbb{E}[(\mu_1 \nu_{\phi_0}(-\infty,0) + \mu_2 \nu_{\phi_0}(0,1) +\mu_3 \nu_{\phi_0}(1,\infty))^{1-\frac{\beta}{\gamma}}]\frac{2}{\gamma} l^{\frac{\beta}{\gamma}-2} \mathbbm{1}_{l>0} dl$, where $\phi_0$ is a random distribution on $\mathbb{H}$ as defined in \cite[Eq. (2.22)]{ASY22}. The moment in the expression is the same as $\overline H^{(\frac{2}{\gamma}, \beta, \frac{2}{\gamma})}_{(\mu_1,\mu_2,\mu_3)}$ defined in \cite[Eq. (1.18)]{RZ22}. Hence, applying \cite[Theorem 1.8]{RZ22} gives the claim. Lemma~\ref{lem:LF3} now follows from integration.\end{proof}
We do not recall the exact expression of $\overline H$ in~\cite{RZ22} which is fairly hard to work with. As we will see in Section~\ref{sec:sol-gab}, 
we only need the integral analyzed in our next lemma.
\begin{lemma}
\label{lem:analy1}
    Fix $\alpha \in (\gamma,Q)$. There exists  a complex neighborhood $V$ of $(2(Q-\alpha) ,Q)$ such that for $\beta\in V$ the following integral is absolutely convergent and analytic in $\beta\in V$:
    \begin{equation}
    \label{eq:def-intg}
    \int_0^\infty \overline H^{(\frac{2}{\gamma}, \beta,\frac{2}{\gamma})}_{(s,1,1)}    \frac{s^{\frac{2}{\gamma}(Q-\alpha)-1}}{1+s} ds\,.
    \end{equation}
\end{lemma}

\begin{proof}  
By \cite[Equation (1.18)]{RZ22}, we can use the Gaussian multiplicative chaos to express $\overline H^{(\frac{2}{\gamma}, \beta,\frac{2}{\gamma})}_{(\mu_1,\mu_2,\mu_3)}$ as follows. For $\mu_1,\mu_2,\mu_3 \geq 0$ with $\mu_1+\mu_2+\mu_3 >0$,
       \begin{equation}
        \label{eq:def-threept}
        \overline H^{(\frac{2}{\gamma}, \beta,\frac{2}{\gamma})}_{(\mu_1,\mu_2,\mu_3)} = \mathbb{E} \Bigg[ \Big(\int_{\mathbb{R}} e^{\frac{\gamma}{2}h(x)} \frac{|x|_+^{\frac{\gamma\beta}{2}}}{|x|\cdot |x-1|^\frac{\gamma\beta}{2}}  d\mu(x) \Big)^{1-\frac{\beta}{\gamma}} \Bigg]\,,
        \end{equation}
where $h$ denotes the GFF on $\mathbb{H}$ defined above \eqref{eq:def-gff} and $e^{\frac{\gamma}{2}h(x)}d\mu(x) := \lim_{\epsilon \rightarrow 0} e^{\frac{\gamma}{2}h_\epsilon(x) - \frac{\gamma^2}{8}\mathbb{E} h_\epsilon(x)^2 }d\mu(x)$ with $d\mu(x) := \mu_1 \mathbbm{1}_{(-\infty,0)}(x) dx + \mu_2 \mathbbm{1}_{(0,1)}(x) dx + \mu_3 \mathbbm{1}_{(1,\infty)}(x) dx$. Following the notation in \cite{RZ22}, we define the fractional power using the argument within $(-\pi,\pi)$. 

Let $\epsilon>0$ be fixed. We first show that
there exists $\lambda = \lambda(\epsilon)>0$ such that for all $s\geq 0$ and $\beta \in U_{0,\epsilon} := \{z: 2(Q-\alpha)+\epsilon \leq {\rm Re} z  \leq Q-\epsilon, -\lambda < {\rm Im} z < \lambda \}$,
\begin{equation}\label{lem:int-bound}
    |\overline H^{(\frac{2}{\gamma}, \beta,\frac{2}{\gamma})}_{(s,1,1)}| \leq 2 \overline H^{(\frac{2}{\gamma}, {\rm Re} \beta,\frac{2}{\gamma})}_{(s,1,1)}.
\end{equation} 
By choosing a sufficiently small $\lambda$, we can ensure that ${\rm arg}((|x|_+/|x-1|)^\frac{\gamma\beta}{2}) \in (-\frac{\pi}{2}, \frac{\pi}{2})$ for all $x \in \mathbb{R}$ and $\beta \in U_{0,\epsilon}$. Therefore, the integral $\int_{\mathbb{R}} e^{\frac{\gamma}{2}h(x)} \frac{|x|_+^{\frac{\gamma\beta}{2}}}{|x|\cdot |x-1|^\frac{\gamma\beta}{2}}  d\mu(x)$ in \eqref{eq:def-threept} has arguments within $(-\frac{\pi}{2},\frac{\pi}{2})$. Using the inequality $|z^w| \leq e^{\frac{\pi}{2} |{\rm Im} w|} |z|^{{\rm Re}w}$ for all $z, w \in \mathbb{C}$ with $\arg z \in (-\frac{\pi}{2}, \frac{\pi}{2})$, we get~\eqref{lem:int-bound} by choosing a smaller $\lambda$ so that $\exp(\frac{\pi \lambda}{2 \gamma}) \leq 2$.

We now show that there exists a neighborhood $U_{1,\epsilon}$ of $[2(Q-\alpha)+\epsilon, Q-\epsilon]$ and a constant $C>0$ such that for all $s\geq 0$ and $\beta \in U_{1,\epsilon}$, \begin{equation}
    |\overline H^{(\frac{2}{\gamma},\beta,\frac{2}{\gamma})}_{(s,1,1)}| \leq C \min \{1, s^{1-\frac{{\rm Re} \beta}{\gamma}} \} .
    \label{eq:upper-bound-threept}
\end{equation}We first upper bound $|\overline H^{(\frac{2}{\gamma},\beta,\frac{2}{\gamma})}_{(s,1,1)}|$ when $s$ is away from zero. By~\eqref{eq:def-threept}, we have $\overline H^{(\frac{2}{\gamma}, \beta,\frac{2}{\gamma})}_{(s,1,1)} \leq \overline H^{(\frac{2}{\gamma}, \beta,\frac{2}{\gamma})}_{(s,0,0)} = \overline H^{(\frac{2}{\gamma}, \beta,\frac{2}{\gamma})}_{(1,0,0)} s^{1-\frac{\beta}{\gamma}}$ for $\beta \in (\gamma, Q)$; and $\overline H^{(\frac{2}{\gamma}, \beta,\frac{2}{\gamma})}_{(s,1,1)} \leq \overline H^{(\frac{2}{\gamma}, \beta,\frac{2}{\gamma})}_{(s,s,s)} = \overline H^{(\frac{2}{\gamma}, \beta,\frac{2}{\gamma})}_{(1,1,1)} s^{1-\frac{\beta}{\gamma}}$ for $\beta \in (0,\gamma)$;  and $\overline H^{(\frac{2}{\gamma}, \gamma,\frac{2}{\gamma})}_{(s,1,1)} = 1$. Combining these inequalities with \eqref{lem:int-bound}, we can choose a constant $C>0$ such that $|\overline H^{(\frac{2}{\gamma}, \beta,\frac{2}{\gamma})}_{(s,1,1)}| \leq C s^{1-\frac{{\rm Re} \beta}{\gamma}}$ for all $s>0 $ and $\beta \in U_{0,\epsilon}$. Next, we upper bound $|\overline H^{(\frac{2}{\gamma},\beta,\frac{2}{\gamma})}_{(s,1,1)}|$ when $s$ is close to zero. According to \cite[Theorem 1.8]{RZ22},  $0 < \overline H^{(\frac{2}{\gamma}, \beta,\frac{2}{\gamma})}_{(0,1,1)} < \infty$ when $\beta \in [2(Q-\alpha)+ \epsilon, Q-\epsilon]$ and $\overline H^{(\frac{2}{\gamma}, \beta,\frac{2}{\gamma})}_{(s,1,1)}$ is a continuous function in both $s\geq 0$ and $\beta$ within a small complex neighborhood of $[2(Q-\alpha)+ \epsilon, Q-\epsilon]$. Therefore, we can choose a sufficiently small constant $\delta = \delta(\epsilon)>0$ and a complex neighborhood $U_{1,\epsilon}$ of $[2(Q-\alpha)+ \epsilon, Q-\epsilon]$ such that $|\overline H^{(\frac{2}{\gamma}, \beta,\frac{2}{\gamma})}_{(s,1,1)}| < \frac{1}{\delta}$ for all $s \in [0,\delta)$ and $\beta \in U_{1,\epsilon}$. Combining the above arguments and enlarging the constant $C$, we prove~\eqref{eq:upper-bound-threept}.

  By \eqref{eq:upper-bound-threept} and the fact that $\alpha < Q$ and ${\rm Re} \beta \geq 2(Q-\alpha) + \epsilon$, we see that the integral in \eqref{eq:def-intg} is absolutely convergent and, moreover uniformly controlled for $\beta \in U_{1,\epsilon}$. By \cite[Theorem 1.8]{RZ22}, there exists another complex neighborhood $U_{2,\epsilon}$ of $[2(Q-\alpha)+ \epsilon, Q-\epsilon]$ such that for any fixed $s>0$, $\overline H^{(\frac{2}{\gamma}, \beta,\frac{2}{\gamma})}_{(s,1,1)}$ is analytic for $\beta \in U_{2,\epsilon}$. Combining these two properties and using Morera's theorem and Fubini's theorem, we see that the integral \eqref{eq:def-intg} is analytic in $U_\epsilon := U_{1,\epsilon} \cap U_{2,\epsilon}$. Therefore, it is absolutely convergent and analytic in $\cup_{\epsilon \leq 1} U_\epsilon$, which is a complex neighborhood of $(2(Q-\alpha),Q)$ as desired.
\end{proof}

\subsection{Expression of the joint moment via boundary structure constants} 
\label{sec:sol-gab}
Recall the measure $\nu_\kappa^{\alpha,\beta}$ from \eqref{eq:def-beta}, where $\beta_0 = \frac{4}{\gamma} - \gamma$  and the scaling exponent $\Delta_\alpha = \frac{\alpha}{2}(Q-\frac{\alpha}{2})$. 
For $\alpha, \beta \in \mathbb{R}$, let
\begin{equation}
\label{eq:def-gab}
g(\alpha,\beta) := \nu_\kappa\left[|\psi'(i)|^{2\Delta_\alpha - 2\Delta_\gamma} |\varphi'(1)|^{\Delta_\beta-\Delta_{\beta_0}}\right]\,.
\end{equation}
Namely, $g(\alpha,\beta)$ is the total mass of the measure $\nu_\kappa^{\alpha,\beta}$. Recall the conformal welding operation in Section~\ref{sec:welding}. It is of the form $\mathrm{Weld}(M_1,M_2,\cdots)$, where $M_1,M_2$ are laws of quantum surfaces with boundary. The quantum surfaces from the smaller surfaces are required to have equal boundary length on their interface but otherwise independent. This allows us to express the length distribution of boundary arcs of a sample from $\mathrm{Weld}(M_1,M_2,\cdots)$ in terms of the boundary length distributions under the individual laws $M_i$.
We now use this observation and a conformal welding result to  express $g(\alpha,\beta)$ via the boundary structure constants $\overline{G}$ and $\overline{H}$ in a certain parameter range. This idea will be used throughout this section.
\begin{proposition}
\label{prop:g(alpha,beta)}
For $ \gamma<\alpha<Q $ and $ \gamma \vee 2(Q-\alpha)  <\beta<Q$, we have
\begin{align}
g(\alpha,\beta) = \frac{4C_1\overline{G}(\alpha,\gamma)\Gamma(\frac{\beta}{\gamma}-1)}{\gamma^2C_2 \overline{G}(\alpha,\beta)\Gamma(\frac{2}{\gamma}(Q-\alpha))\Gamma (\frac{2}{\gamma}(\frac{\beta}{2} + \alpha - Q))}  \int_0^\infty \overline H^{(\frac{2}{\gamma}, \beta,\frac{2}{\gamma})}_{(s,1,1)}    \frac{s^{\frac{2}{\gamma}(Q-\alpha)-1}}{1+s} ds \,,  \label{eq:for-gab}
\end{align}
where  $C_1$ and $C_2$ are the constants defined in Proposition~\ref{prop:weld6} and Lemma~\ref{lem:weld4}, respectively.
\end{proposition}
\begin{proof}
Recall  the conformal welding identity from Proposition~\ref{prop:reweight} for ${\rm Weld}({\rm LF}_\mathbb{H}^{(\frac{2}{\gamma},0), (\beta,1), (\frac{2}{\gamma} ,\infty)}, \mathcal{M}_{1,2}^{\rm disk}(\alpha))$.
By Proposition~\ref{prop:reweight}, and the definition of $g(\alpha,\beta)$, for $\mu>0$ we have
\begin{equation}
    \label{eq:prop4.14-1}
    g(\alpha,\beta) {{\rm LF}_\mathbb{H}^{(\alpha,i),(\beta,0)}[e^{-\mu \mathcal L}]} = \frac{2C_1}{\gamma C_2}\cdot {\rm Weld}({\rm LF}_\mathbb{H}^{(\frac{2}{\gamma},0), (\beta,1), (\frac{2}{\gamma} ,\infty)}, \mathcal{M}_{1,2}^{\rm disk}(\alpha)) [e^{-\mu \mathcal L}]\,,
\end{equation}
where on both sides $\mathcal L$ represents the total  boundary length of the quantum surface.

Write $K(\mu):= {\rm Weld}({\rm LF}_\mathbb{H}^{(\frac{2}{\gamma},0), (\beta,1), (\frac{2}{\gamma} ,\infty)}, \mathcal{M}_{1,2}^{\rm disk}(\alpha)) [e^{-\mu \mathcal L}]$.
Let $|{\rm LF}_\mathbb{H}^{(\frac{2}{\gamma},0), (\beta,1), (\frac{2}{\gamma} ,\infty)}(L_{13},L_{12},L_{23})|$ be the joint density of $\mathcal L_{13},\mathcal L_{12},\mathcal L_{23}$ in Lemma~\ref{lem:LF3}. Recall $|\mathcal{M}_{1,2}^{\rm disk}(\alpha; \ell_1,\ell_2)|$ from Lemma~\ref{lem:LF2}, which is the joint density of the two boundary arcs of a sample from $\mathcal{M}_{1,2}^{\rm disk}(\alpha; \ell_1,\ell_2)$. Then we have
    \begin{align*}
        K(\mu) =\int_{\mathbb{R}_+^4} |{\rm LF}_\mathbb{H}^{(\frac{2}{\gamma}, 0),(\beta,1),(\frac{2}{\gamma},\infty)}(\ell,\ell_1,\ell_2)| \times |\mathcal{M}_{1,2}^{\rm disk}(\alpha; \ell,\ell_3)| \times e^{-\mu(\ell_1+\ell_2+\ell_3)} d\ell d\ell_1 d\ell_2 d\ell_3\,.
    \end{align*}
    
By  Lemma~\ref{lem:LF2} and the fact that $\alpha>\frac{\gamma}{2}$, we have:
    \begin{align*}
        K(\mu)=\frac{2}{\gamma} 2^{-\frac{\alpha^2}{2}} \overline G(\alpha,\gamma) \int_{\mathbb{R}_+^4} |{\rm LF}_\mathbb{H}^{(\frac{2}{\gamma}, 0),(\beta,1),(\frac{2}{\gamma},\infty)}(\ell, \ell_1, \ell_2)| \times (\ell + \ell_3)^{\frac{2}{\gamma}(\alpha-Q)}  e^{-\mu(\ell_1 + \ell_2 + \ell_3)} d\ell d\ell_1 d\ell_2 d\ell_3\,.
    \end{align*}
Since $(\ell + \ell_3)^{\frac{2}{\gamma}(\alpha-Q)} \Gamma(\frac{2}{\gamma}(Q-\alpha))=  \int_0^\infty e^{-s(\ell + \ell_3)} s^{\frac{2}{\gamma}(Q-\alpha)-1} ds$ for $\ell>0, \ell_3>0$ and $\alpha<Q$, we have:
    \begin{align*}
        K(\mu)&=\frac{2^{1-\frac{\alpha^2}{2}} \overline G(\alpha,\gamma)}{\gamma \Gamma(\frac{2}{\gamma}(Q-\alpha))} \int_{\mathbb{R}_+^5} |{\rm LF}_\mathbb{H}^{(\frac{2}{\gamma}, 0),(\beta,1),(\frac{2}{\gamma},\infty)}(\ell, \ell_1, \ell_2)| \times  e^{-s \ell- \mu \ell_1 -\mu \ell_2 - (\mu +s) \ell_3} s^{\frac{2}{\gamma}(Q-\alpha)-1} d\ell d\ell_1 d\ell_2 d\ell_3 ds\,.
    \end{align*}
Interchanging the order of integration and using Lemma~\ref{lem:LF3} with $\beta \in (\gamma, Q)$, we have: 
    \begin{align*}
        K(\mu)&=\frac{2^{1-\frac{\alpha^2}{2}} \overline G(\alpha,\gamma)}{\gamma \Gamma(\frac{2}{\gamma}(Q-\alpha))} \int_{\mathbb{R}_+^2} {\rm LF}_\mathbb{H}^{(\frac{2}{\gamma}, 0),(\beta,1),(\frac{2}{\gamma},\infty)}[e^{-s \mathcal L_{13} - \mu \mathcal L_{12} - \mu \mathcal L_{23}}]  \times  e^{- (\mu +s) \ell_3} s^{\frac{2}{\gamma}(Q-\alpha)-1} d\ell_3 ds \\
        &=\frac{2^{2-\frac{\alpha^2}{2}} \overline G(\alpha,\gamma) \Gamma(\frac{\beta}{\gamma}-1)}{\gamma^2 \Gamma(\frac{2}{\gamma}(Q-\alpha))} \int_{\mathbb{R}_+^2} \overline H^{(\frac{2}{\gamma},\beta,\frac{2}{\gamma})}_{(s,\mu,\mu)}   e^{- (\mu +s) \ell_3} s^{\frac{2}{\gamma}(Q-\alpha)-1} d\ell_3 ds\,.
    \end{align*}
By integrating with respect to $\ell_3$ first, we obtain:
    $$
    K(\mu) = \frac{2^{2-\frac{\alpha^2}{2}} \overline G(\alpha,\gamma) \Gamma(\frac{\beta}{\gamma}-1)}{\gamma^2 \Gamma(\frac{2}{\gamma}(Q-\alpha))} \int_0^\infty \overline H^{(\frac{2}{\gamma},\beta,\frac{2}{\gamma})}_{(s,\mu,\mu)}   \frac{s^{\frac{2}{\gamma}(Q-\alpha)-1}}{\mu+s} ds\,.
    $$
By \eqref{eq:def-threept}, we have $ \overline H^{(\frac{2}{\gamma},\beta,\frac{2}{\gamma})}_{(\mu s,\mu,\mu)} = \mu^{1-\frac{\beta}{\gamma}} \overline H^{(\frac{2}{\gamma},\beta,\frac{2}{\gamma})}_{(s,1,1)}$. Therefore,
    \begin{equation}
    \label{eq:prop4.14-2}
        K(\mu)=\frac{2^{2-\frac{\alpha^2}{2}} \overline G(\alpha,\gamma) \Gamma(\frac{\beta}{\gamma}-1)}{\gamma^2 \Gamma(\frac{2}{\gamma}(Q-\alpha))}\int_0^\infty \overline H^{(\frac{2}{\gamma},\beta,\frac{2}{\gamma})}_{(s,1,1)}   \frac{s^{\frac{2}{\gamma}(Q-\alpha)-1}}{1+s} ds \cdot \mu^{\frac{2}{\gamma}(Q-\alpha - \frac{\beta}{2})}.
    \end{equation}
Since $\gamma<\alpha<Q$ and $2(Q-\alpha)  <\beta<Q$, by Lemma~\ref{lem:analy1}, the integral in~\eqref{eq:prop4.14-2} is finite. By \eqref{eq:prop4.14-1} and the expression~\eqref{eq:prop4.14-3} of ${\rm LF}_\mathbb{H}^{(\alpha,i),(\beta,0)}[e^{-\mu \mathcal L }] $, we obtain~\eqref{eq:for-gab}.
\end{proof}

\subsection{Identities on the length distribution of quantum disks}
\label{subsec:QD-id}

 In Lemma~\ref{lem:LF2}, we have already given the boundary length distribution of the boundary arcs for $\mathcal{M}_{1,2}^{\rm disk}(\alpha)$. In this section, we give two lemmas concerning the boundary length distribution for various quantum disks defined in Section~\ref{sec:welding}. Together with Lemma~\ref{lem:LF2}, they serve as basic ingredients when we perform computations based on Proposition~\ref{prop:g(alpha,beta)}. Both lemmas are essentially extracted from the literature.

Recall the quantum disks $\widetilde {\rm QD}_{0,3}$ and ${\rm QD}_{1,1}$ from Definition~\ref{def:QD}. Let $|\widetilde {\rm QD}_{0,3}(\ell_1, \ell_2, \ell_3)|$ be the joint density of the quantum lengths of the three boundary arcs of a sample from $\widetilde {\rm QD}_{0,3}$. This is in the same sense as how $|\mathcal{M}_{1,2}^{\rm disk}(\alpha; \ell_1, \ell_2)|$ is defined in~\eqref{eq:def-M12-length}. Similarly, let $|{\rm QD}_{1,1}(\ell)|$ be the density of the quantum length of the entire boundary of a sample from ${\rm QD}_{1,1}$. The values of $|\widetilde {\rm QD}_{0,3}(\ell_1, \ell_2, \ell_3)|$ and $|{\rm QD}_{1,1}(\ell)|$ are given in the following lemma:

\begin{lemma}
    \label{lem:QD}
    For any $\ell, \ell_1, \ell_2, \ell_3>0$, we have
 \begin{equation}
        \label{eq:qd1}
            |\widetilde {\rm QD}_{0,3}(\ell_1, \ell_2, \ell_3)| = E_1 (\ell_1 + \ell_2 + \ell_3)^{-\frac{4}{\gamma^2}-1} \quad \mbox{and} \quad |{\rm QD}_{1,1}(\ell)| = E_2 \ell^{-\frac{4}{\gamma^2}+1},
    \end{equation}
    where 
    \begin{equation}
        \label{eq:val-e1-e2}
        E_1  = \frac{( 2 \pi)^{\frac{4}{\gamma^2}-1}}{(1-\frac{\gamma^2}{4})\Gamma(1-\frac{\gamma^2}{4})^{\frac{4}{\gamma^2}}} \quad \mbox{and} \quad E_2 = \frac{\Gamma(\frac{\gamma^2}{4})}{4 \pi (Q-\gamma)^2} \Big( \frac{2 \pi }{\Gamma(1-\frac{\gamma^2}{4})} \Big)^{\frac{4}{\gamma^2}-1}.
    \end{equation}
\end{lemma}
\begin{proof}
Recall Definition~\ref{def:QD}. By the definition of $\rm {QD},\rm {QD}_{0,2}$, $\rm {QD}_{0,3}$, we have 
$|{\rm QD}_{0,2}(\ell)| = \ell^2 |{\rm QD}(\ell)|$ and
$|{\rm QD}_{0,3}(\ell)| = \ell^3 |{\rm QD}(\ell)|$. By Lemma 3.2 in \cite{ARS21}, we have $|{\rm QD}_{0,2}(\ell)| = E_1 \ell^{-\frac{4}{\gamma^2}}$. By the definition of  $\widetilde {\rm QD}_{0,3}$, we see that 
$ |\widetilde {\rm QD}_{0,3}(\ell_1, \ell_2, \ell_3)| $ equals $|{\rm QD}_{0,3}(\ell_1 + \ell_2 + \ell_3)| / (\ell_1 + \ell_2 + \ell_3)^2 = E_1  (\ell_1 + \ell_2 + \ell_3)^{-\frac{4}{\gamma^2}-1}$.

Similarly,  by definition $|{\rm QD}_{1,1}(\ell)| = \ell|{\rm QD}_{1,0}(\ell)|$. By \cite[Eq.~(2.6)]{ARS21} and the equation below Theorem 3.4 therein, we have $|{\rm QD}_{1,0}(\ell)| = E_2 \ell^{-\frac{4}{\gamma^2}}$. This gives $|{\rm QD}_{1,1}(\ell)|= E_2 \ell^{-\frac{4}{\gamma^2}+1}$. \qedhere
\end{proof}

Recall the two pointed quantum disk $\mathcal{M}_{2}^{\rm disk}(W)$ from Definitions~\ref{def:thick} and \ref{def:thin-disk}. The boundary length distribution can be computed via the so-called boundary reflection coefficient of Liouville theory determined in~\cite{RZ22}.  We only need the case when $W=\gamma^2-2$, which can be extracted from \cite{AHS21}.
\begin{lemma}
    \label{lem:two-pt-QD} Fix $\gamma\in (\sqrt{2},2)$.
    For any $\mu_1, \mu_2\geq0$ with $\mu_1+\mu_2>0$, we have
    \begin{equation}
        \label{eq:twopt-disk}
        \mathcal{M}_{2}^{\rm disk}(\gamma^2-2)[e^{-\mu_1 \mathcal L-\mu_2 \mathcal R}] = E_3 \frac{\mu_1-\mu_2}{\mu_1^{4/\gamma^2} - \mu_2^{4/\gamma^2}}\,.
    \end{equation} where $\mathcal L$ and $\mathcal R$ represent the left and right boundary lengths of a sample from $\mathcal{M}_{2}^{\rm disk}(\gamma^2-2)$, and
    \begin{equation}
        \label{eq:val-e3-e4}
        E_3 = \frac{(2 \pi)^{1-\frac{4}{\gamma^2}}\Gamma(1-\frac{\gamma^2}{4})^{\frac{4}{\gamma^2}}}{\Gamma(2-\frac{4}{\gamma^2})} \,.
    \end{equation}
\end{lemma}
\begin{proof}
This follows directly from Propositions 3.6 in \cite{AHS21} and Eq.~(3.3) and (3.5) therein. The relevant boundary reflection coefficient is $\overline R(\frac{4}{\gamma}, \mu_1,\mu_2)$, whose value can be calculated explicitly using the shift equations of the double gamma function $\Gamma_{\frac{\gamma}{2}}(z)$ in \cite[Eq.~ (3.2)]{AHS21}. 
\end{proof}

\subsection{Explicit value of $C_1$}
\label{sec:sol-C1}
In this section, we determine the explicit value of $C_1$ introduced in Proposition~\ref{prop:weld6}.
\begin{proposition}
    \label{prop:val-c1}
    The value of $C_1$ is given by
    $C_1 =  \frac{\pi 2^{1-\frac{\gamma^2}{2}} (\gamma^2-4)^2 \sin(\frac{\pi \gamma^2}{4})}{\gamma^3 \sin(-\frac{4 \pi}{\gamma^2})}$.
\end{proposition}

\begin{proof}
Recall the setting of Proposition~\ref{prop:weld6}. For $\mu>0$, let $J(\mu)={\rm Weld}(\mathcal{M}_2^{\rm disk}(\gamma^2-2), {\rm QD}_{1,1})[\mathcal L e^{-\mu \mathcal L}]$ where $\mathcal L$ represents the total boundary length of a sample of the welded surface. By Proposition~\ref{prop:weld6} and the fact that $\mumu$ there has total mass 1, we have 
\begin{equation}
\label{eq:c1-1}
{\rm LF}_\mathbb{H}^{(\gamma,i),(\frac{4}{\gamma}-\gamma,0)}[\mathcal L e^{-\mu \mathcal L}]= C_1 J(\mu)\,,
\end{equation}
where $\mathcal L$ represents the total boundary length of a sample from ${\rm LF}_\mathbb{H}^{(\gamma,i),(\frac{4}{\gamma}-\gamma,0)}$.
By Lemma~\ref{lem:LF1}, we have $|{\rm LF}_\mathbb{H}^{(\gamma,i),(\frac{4}{\gamma}-\gamma,0)}(\ell)| = \frac{2}{\gamma}2^{-\frac{\gamma^2}{2}} \overline G(\gamma,\frac{4}{\gamma}-\gamma) \ell^{-1}$. Therefore,
\begin{equation}
\label{eq:C3-1}
{\rm LF}_\mathbb{H}^{(\gamma,i),(\frac{4}{\gamma}-\gamma,0)}[\mathcal L e^{-\mu \mathcal L}]= \int_0^\infty \frac{2}{\gamma}2^{-\frac{\gamma^2}{2}} \overline G(\gamma,\frac{4}{\gamma}-\gamma) e^{-\mu \ell } d \ell =\frac{2}{\gamma}2^{-\frac{\gamma^2}{2}} \overline G(\gamma,\frac{4}{\gamma}-\gamma) \mu^{-1}. 
\end{equation}
Let $|\mathcal{M}_2^{\rm disk}(\gamma^2-2; \ell, \ell_1)|$ be the joint density of the left and right boundary lengths of  a sample from $\mathcal{M}_2^{\rm disk}(\gamma^2-2)$. By the definition of $J(\mu)$ in terms of conformal welding, we have 
\begin{equation*}
  J (\mu) = \iint_0^\infty |\mathcal{M}_2^{\rm disk}(\gamma^2-2; \ell, \ell_1)| \times |{\rm QD}_{1,1}(\ell_1)| \times \ell e^{-\mu \ell } d\ell d\ell_1 \,.
\end{equation*} 
By Lemma~\ref{lem:QD}, with the constant $E_2$ from there, we have:
\begin{equation*}
 J (\mu)= E_2 \iint_0^\infty  |\mathcal{M}_2^{\rm disk}(\gamma^2-2; \ell, \ell_1)| \times \ell_1^{-\frac{4}{\gamma^2}+1} \ell e^{-\mu \ell} d\ell d\ell_1 \,.
\end{equation*} 
Since $\ell_1^{-\frac{4}{\gamma^2}+1} \Gamma(\frac{4}{\gamma^2}-1) = \int_0^\infty e^{-t \ell_1}t^{\frac{4}{\gamma^2}-2} dt$, we can rewrite the above expression as:
\begin{equation}
\label{eq:C3-2}
\begin{aligned}
    J(\mu) = \frac{E_2}{\Gamma(\frac{4}{\gamma^2}-1)} \int_0^\infty \mathcal{M}_2^{\rm disk}(\gamma^2-2)[\mathcal L e^{-\mu \mathcal L - t \mathcal R} ]\times t^{\frac{4}{\gamma^2}-2} dt\,.
\end{aligned}
\end{equation}
By Lemma~\ref{lem:two-pt-QD}, we have $ 
\mathcal{M}_{2}^{\rm disk}(\gamma^2-2)[e^{-\mu \mathcal L-t \mathcal R}] = E_3\frac{\mu-t}{\mu^{4/\gamma^2} - t^{4/\gamma^2}}$ with the constant $E_3$ from Lemma~\ref{lem:two-pt-QD}. Taking the derivative with respect to $\mu$ gives:
$$
\mathcal{M}_{2}^{\rm disk}(\gamma^2-2)[\mathcal L e^{-\mu \mathcal L-t \mathcal R}] = E_3 \left(\frac{\frac{4}{\gamma^2}\mu^{4/\gamma^2-1}(\mu-t)}{(\mu^{4/\gamma^2} - t^{4/\gamma^2})^2}- \frac{1}{\mu^{4/\gamma^2} - t^{4/\gamma^2}} \right)\,.
$$
Substituting this equation into \eqref{eq:C3-2} we obtain:
\begin{equation}
\label{eq:C3-3}
\begin{aligned}
    J (\mu) &=\frac{E_2 E_3}{\Gamma(\frac{4}{\gamma^2}-1)}\int_0^\infty \left(\frac{\frac{4}{\gamma^2}\mu^{4/\gamma^2-1}(\mu-t)}{(\mu^{4/\gamma^2} - t^{4/\gamma^2})^2}- \frac{1}{\mu^{4/\gamma^2} - t^{4/\gamma^2}} \right) t^{\frac{4}{\gamma^2}-2} dt =\frac{E_2 E_3}{\Gamma(\frac{4}{\gamma^2}-1)} \mu^{-1}.
\end{aligned}
\end{equation}
In the second equality, we substituted $t$ with $\mu t$ and used the identity:
$$
\int_0^\infty \left(\frac{\frac{4}{\gamma^2}(1-t)}{(1 - t^{4/\gamma^2})^2}- \frac{1}{1 - t^{4/\gamma^2}} \right) t^{\frac{4}{\gamma^2}-2} dt = \frac{t^{4/\gamma^2-1}(1-t)}{1-t^{4/\gamma^2}}\bigg{|}_{t=0}^\infty =1\,.
$$
Combining \eqref{eq:c1-1}, \eqref{eq:C3-1}, and \eqref{eq:C3-3}, we get:
$$
C_1 = \frac{\frac{2}{\gamma}2^{-\frac{\gamma^2}{2}} \overline G(\gamma,\frac{4}{\gamma}-\gamma) \mu^{-1}}{\frac{E_2 E_3}{\Gamma(\frac{4}{\gamma^2}-1)} \mu^{-1}} = \frac{2^{1-\frac{\gamma^2}{2}}}{E_2 E_3 \gamma} \Gamma(\frac{4}{\gamma^2}-1) \overline G(\gamma,\frac{4}{\gamma}-\gamma) \,.
$$
Recall $E_2$ from Lemma~\ref{lem:QD}, $E_3$ from Lemma~\ref{lem:two-pt-QD}, and $\overline G(\gamma, \frac{4}{\gamma}-\gamma)$ from Lemma~\ref{lem:LF4}, we have
\begin{align*}
    C_1 &= \frac{2^{1-\frac{\gamma^2}{2}}}{\frac{\Gamma(\frac{\gamma^2}{4})}{4 \pi (Q-\gamma)^2} ( \frac{2 \pi }{\Gamma(1-\frac{\gamma^2}{4})} )^{\frac{4}{\gamma^2}-1} \frac{(2 \pi)^{1-\frac{4}{\gamma^2}}\Gamma(1-\frac{\gamma^2}{4})^{\frac{4}{\gamma^2}}}{\Gamma(2-\frac{4}{\gamma^2})} \gamma} \Gamma(\frac{4}{\gamma^2}-1) \frac{2^{3-\frac{\gamma^2}{2}-\frac{4}{\gamma^2}}}{\pi^{\frac{4}{\gamma^2}-1}} \frac{\Gamma(1-\frac{\gamma^2}{4})^{\frac{4}{\gamma^2}-1} \Gamma(\frac{\gamma^2}{2}-1)}{\Gamma(2-\frac{4}{\gamma^2})\Gamma(\frac{\gamma^2}{4})} \\
    &\quad \times \frac{\Gamma(2-\frac{4}{\gamma^2}) }{\Gamma(\frac{\gamma^2}{2}-1)} \bigg(\frac{2^{-\frac{\gamma^2}{2}} 2 \pi}{\Gamma(1-\frac{\gamma^2}{4})}\bigg)^{\frac{4}{\gamma^2}-1} \Gamma(\frac{\gamma^2}{4})=\frac{2^{1-\frac{\gamma^2}{2}} 4 \pi (Q-\gamma)^2 \Gamma(2-\frac{4}{\gamma^2}) \Gamma(\frac{4}{\gamma^2}-1)}{\gamma \Gamma(\frac{\gamma^2}{4}) \Gamma(1-\frac{\gamma^2}{4})}\,.
\end{align*}
Simplifying the last expression using the identity $\Gamma(x) \Gamma(1-x) = \frac{\pi}{\sin(\pi x)}$, we conclude the proof.
\end{proof}

\subsection{An integral formula for the boundary three-point structure constant}\label{subsec:H}

Recall the expression of the joint moment $\nu_\kappa [ |\psi'(i)|^{2\Delta_\alpha - 2\Delta_\gamma} |\varphi'(1)|^{\Delta_\beta-\Delta_{\beta_0}} ]$ from Proposition~\ref{prop:g(alpha,beta)}. The only term left to analyze is the integral $\int_0^\infty \frac{1}{C_2}\overline H^{(\frac{2}{\gamma}, \beta,\frac{2}{\gamma})}_{(s,1,1)} \frac{s^{\frac{2}{\gamma}(Q-\alpha)-1}}{1+s} ds$, where $C_2$ is the constant from Lemma~\ref{lem:weld4}. By Lemma~\ref{lem:analy1}, for $\alpha \in (\gamma,Q)$, this integral is analytic in $\beta$ in a complex neighborhood of $(2(Q-\alpha) ,Q)$. To prove Theorem~\ref{thm:moment}, we will need the exact evaluation of this integral when $\beta = \beta_0 = \frac{4}{\gamma}-\gamma$. This is done in the following proposition. Note that $\beta_0\in (2(Q-\alpha) ,Q)$ for $\alpha \in (\gamma,Q)$.
\begin{proposition}
    \label{lem:3pt-integral}  
Recall $E_1$ from Lemma~\ref{lem:QD} and $E_3$ from Lemma~\ref{lem:two-pt-QD}.  For $\gamma \in (\sqrt{2},2  )$,  let  
\begin{equation}
    \label{eq:def-e5}
    E_4 =\frac{E_1 E_3^2 \gamma }{2 \Gamma(\frac{4}{\gamma^2}-2)\Gamma(\frac{4}{\gamma^2}+1)} \,.
\end{equation}
For $\theta \in (0,\frac{4}{\gamma^2}-1)$, we have
    $$
    \int_0^\infty \frac{1}{C_2}\overline H^{(\frac{2}{\gamma}, \frac{4}{\gamma}-\gamma,\frac{2}{\gamma})}_{(s,1,1)}    \frac{s^{\theta-1}}{1+s} ds =  \frac{E_4\gamma^4\pi^2(\sin(\frac{\pi \gamma^2}{2}\theta)-\theta\sin(\frac{\pi\gamma^2}{2}))}{32\cos(\frac{\pi\gamma^2}{4}) \sin(\pi\theta) \sin(\frac{\pi\gamma^2}{4}\theta)\sin(\frac{\pi\gamma^2}{4}(\theta+1))}\,.
    $$
\end{proposition}

We first establish an integral representation of $\overline H^{(\frac{2}{\gamma}, \frac{4}{\gamma}-\gamma,\frac{2}{\gamma})}_{(s,1,1)}$ using Lemma~\ref{lem:weld4}.
\begin{lemma}
    \label{lem:LF5}
    For any $\gamma \in (\sqrt{2},2)$, we have
    $$
    \frac{1}{C_2} \overline H^{(\frac{2}{\gamma}, \frac{4}{\gamma}-\gamma,\frac{2}{\gamma})}_{(s,1,1)} =  E_4\int_0^\infty  \Big(  \frac{\mu^{4/\gamma^2} (\mu-1)^2}{(\mu+s) (\mu^{4/\gamma^2} - 1)^2 } - \frac{1}{\mu^{4/\gamma^2-1}} \Big) d\mu \,.
    $$
\end{lemma}
\begin{proof}

Recall the conformal welding in Lemma~\ref{lem:weld4}. For $s>0$, let 
$$O(s) := {\rm Weld}(\mathcal{M}_{2}^{\rm disk}(\gamma^2-2), \widetilde{{\rm QD}}_{0,3}, \mathcal{M}_{2}^{\rm disk}(\gamma^2-2))[e^{-s \mathcal L_{13} - \mathcal L_{12} - \mathcal L_{23}} - 1],$$ where $\mathcal L_{13}, \mathcal L_{12}, \mathcal L_{23}$ represent the quantum lengths of the three boundary arcs of a sample from the welded surface. In particular, $\mathcal L_{13}$ corresponds to the boundary arc of $\widetilde{{\rm QD}}_{0,3}$ that is part of the outer boundary of the welded surface. In Lemma~\ref{lem:weld4}, since  $\widehat m$  is a  probability  measure hence has total mass $1$, we have
\begin{equation}
    \label{eq:os-1}
    {\rm LF}_\mathbb{H}^{(\frac{2}{\gamma},0), (\frac{4}{\gamma}-\gamma,1), (\frac{2}{\gamma} ,\infty)}[e^{-s \mathcal L_{13} -  \mathcal L_{12} -  \mathcal L_{23} } - 1 ] = C_2 \cdot O(s)\,,
\end{equation}
where $\mathcal L_{13}, \mathcal L_{12}, \mathcal L_{23}$ represent the quantum lengths of the three boundary arcs of a sample from the Liouville field as defined in Lemma~\ref{lem:LF3}.
Using Lemma~\ref{lem:LF3} with $\beta = \frac{4}{\gamma}-\gamma < \gamma$, we have
\begin{equation}
\label{eq:h1}
 {\rm LF}_\mathbb{H}^{(\frac{2}{\gamma},0), (\frac{4}{\gamma}-\gamma,1), (\frac{2}{\gamma} ,\infty)}[e^{-s \mathcal L_{13} -  \mathcal L_{12} -  \mathcal L_{23} } - 1 ] = \frac{2}{\gamma} \Gamma(\frac{4}{\gamma^2}-2) \overline H^{(\frac{2}{\gamma}, \frac{4}{\gamma}-\gamma,\frac{2}{\gamma})}_{(s,1,1)}.
\end{equation}

Let $|\mathcal{M}_{2}^{\rm disk}(\gamma^2-2; \ell_1, \ell_2)|$ be the joint density of the left and right boundary lengths of a sample from $\mathcal{M}_{2}^{\rm disk}(\gamma^2-2)$. Also, recall $|\widetilde{{\rm QD}}_{0,3}(\ell_1, \ell_2, \ell_3)|$ from Lemma~\ref{lem:QD}. By the  definition of $O(s)$ in terms of conformal welding, we have:
\begin{align*}
O(s) &= \int_{\mathbb{R}_+^5} |\mathcal{M}_{2}^{\rm disk}(\gamma^2-2; \ell_1, \ell_{12})| \times |\widetilde{{\rm QD}}_{0,3}(\ell_{13}, \ell_1, \ell_2)| \times |\mathcal{M}_{2}^{\rm disk}(\gamma^2-2; \ell_2, \ell_{23})| \\
&\qquad \qquad \qquad \qquad \qquad \qquad \qquad \quad \times (e^{-s \ell_{13} - \ell_{12} - \ell_{23}} - 1) d\ell_1 d\ell_2 d\ell_{12} d\ell_{13} d\ell_{23}\,.
\end{align*}
By Lemma~\ref{lem:QD} and the identity $( \ell_{13} + \ell_1 + \ell_2)^{-\frac{4}{\gamma^2}-1}\Gamma(\frac{4}{\gamma^2}+1) = \int_0^\infty e^{-\mu (\ell_{13} + \ell_1 + \ell_2)} \mu^{\frac{4}{\gamma^2}} d\mu$, we get:
\begin{align*}
    O(s) &= \frac{E_1 }{\Gamma(\frac{4}{\gamma^2}+1)} \int_{\mathbb{R}_+^6} |\mathcal{M}_{2}^{\rm disk}(\gamma^2-2; \ell_1, \ell_{12})| \times |\mathcal{M}_{2}^{\rm disk}(\gamma^2-2; \ell_2, \ell_{23})|  \\
    & \qquad \qquad  \qquad \quad \times (e^{-(\mu+s) \ell_{13}} e^{-\mu \ell_1 - \ell_{12}} e^{-\mu \ell_2 - \ell_{23}} -e^{-\mu \ell_{13}} e^{-\mu \ell_1 } e^{-\mu \ell_2}  )\mu^{\frac{4}{\gamma^2}} d \mu d\ell_1 d\ell_2 d\ell_{12} d\ell_{13} d\ell_{23}\,.
\end{align*}
Recall from Lemma~\ref{lem:two-pt-QD} that $\mathcal{L}$ and $\mathcal{R}$ represent the boundary lengths of a sample from $\mathcal{M}_{2}^{\rm disk}(\gamma^2-2)$. Since the integrands are negative, we can interchange the order of integration and obtain that:
\begin{equation*}
    \begin{aligned}
        O(s) = \frac{E_1 }{\Gamma(\frac{4}{\gamma^2}+1)} \iint_0^\infty  \Big(e^{-(\mu+s)\ell_{13} }\mathcal{M}_{2}^{\rm disk}(\gamma^2-2)[e^{-\mu \mathcal L - \mathcal R}]^2 - e^{-\mu \ell_{13}}\mathcal{M}_{2}^{\rm disk}(\gamma^2-2)[e^{-\mu \mathcal L}]^2 \Big) \mu^{\frac{4}{\gamma^2}} d\mu d \ell_{13}\,.
    \end{aligned}
\end{equation*}
Applying Lemma~\ref{lem:two-pt-QD} twice with $(\mu_1,\mu_2) = (\mu, 1)$ and $(\mu,0)$ respectively, and then integrating against $\ell_{13}$, we obtain that:
\begin{equation}
\label{eq:os-2}
        O(s) = \frac{E_1 E_3^2}{\Gamma(\frac{4}{\gamma^2}+1)} \int_0^\infty  \Big(  \frac{\mu^{4/\gamma^2} (\mu-1)^2}{(\mu+s) (\mu^{4/\gamma^2} - 1)^2 } - \frac{1}{\mu^{4/\gamma^2-1}} \Big) d\mu \,.
\end{equation}
Combining \eqref{eq:os-1}, \eqref{eq:h1}, \eqref{eq:os-2} with the value of $E_4$ from \eqref{eq:def-e5} yields the desired result.
\end{proof}

Proposition~\ref{lem:3pt-integral} now follows from the elementary integral formula below.
\begin{lemma}
    \label{lem:lem4.8-int}
    For any $\gamma \in (\sqrt{2},2 )$ and $\theta \in (0,\frac{4}{\gamma^2}-1)$, we have
    $$
    \int_0^\infty \Big(\frac{\mu^{4/\gamma^2}  (\mu-1)(\mu^{\theta-1}-1)}{ (\mu^{4/\gamma^2} - 1)^2 }  + \frac{1}{\mu^{4/\gamma^2-1}} \Big) d \mu = \frac{\gamma^4\pi (\theta\sin(\frac{\pi\gamma^2}{2})- \sin(\frac{\pi \gamma^2}{2}\theta))}{32\cos(\frac{\pi\gamma^2}{4})  \sin(\frac{\pi\gamma^2}{4}\theta)\sin(\frac{\pi\gamma^2}{4}(\theta+1))}\,.
    $$
\end{lemma}
We need the following two lemmas to prove  Lemma~\ref{lem:lem4.8-int}.  For $z \in \mathbb{R}$, let $\psi(z) = \frac{\Gamma'(z)}{\Gamma(z)}$ be the digamma function, which satisfies the reflection relation:
\begin{equation}
\label{eq:reflection}
\psi(1-z) - \psi(z) = \pi \cot (\pi z)\,.
\end{equation}
\begin{lemma}
\label{lem:integral1}
For any $a,b>0$, we have $\int_1^\infty \frac{t^{-a} - t^{-b}}{t-1} dt = \psi(b) - \psi(a)$.
\end{lemma}
\begin{proof}
    Fix $\epsilon \in (0,1)$. For $0<a<1$, we have $\int_0^\infty \frac{(1+t)^{-a} -1 }{t^{1+\epsilon}} dt =\frac{\Gamma(-\epsilon)\Gamma(\epsilon+a)}{\Gamma(a)}$; see e.g.\ Lemma 5.16 in \cite{RZ22}. Therefore, for $0<a,b<1$, we have
    \begin{equation}
    \label{eq:lemc.1-1}
    \begin{aligned}
    \int_0^\infty \frac{(1+t)^{-a} -(1+t)^{-b} }{t^{1+\epsilon}} dt &= \int_0^\infty \frac{(1+t)^{-a} -1 }{t^{1+\epsilon}} dt - \int_0^\infty \frac{(1+t)^{-b} -1 }{t^{1+\epsilon}} dt \\
    &= \frac{\Gamma(-\epsilon)\Gamma(\epsilon+a)}{\Gamma(a)} - \frac{\Gamma(-\epsilon)\Gamma(\epsilon+b)}{\Gamma(b)}\,.
    \end{aligned}
    \end{equation}
    By analytic continuations, we see that \eqref{eq:lemc.1-1} holds for any $a, b>0$. Taking the limit as $\epsilon$ approaches zero and replacing $t + 1$ with $t$ yields the desired result (since $\Gamma(-\epsilon) = -\frac{1}{\epsilon} + o(\frac{1}{\epsilon})$ and $\Gamma(\epsilon+a) = \Gamma(a) + \Gamma'(a) \epsilon +o(\epsilon)$ as $\epsilon \rightarrow 0$).
\end{proof}

\begin{lemma}
\label{lem:integral2}
    For any $-1<a,b<0$, we have $\int_0^\infty \frac{t^a-t^b}{t-1}dt = \pi (\cot(\pi b) - \cot(\pi a))$.
\end{lemma}
\begin{proof}
    We can compute the integral as follows:
    \begin{align*}
        \int_0^\infty \frac{t^a-t^b}{t-1}dt &= \int_1^\infty \frac{t^a-t^b}{t-1}dt + \int_0^1 \frac{t^a-t^b}{t-1}dt = \int_1^\infty \frac{t^a-t^b}{t-1}dt + \int_1^\infty \frac{t^{-b-1}-t^{-a-1}}{t-1}dt \\
        &=\psi(-b)-\psi(-a) +\psi(a+1) - \psi(b+1) = \pi(\cot(\pi b) - \cot(\pi a))\,.
    \end{align*}
   In the second equation, we changed the variable $t$ to $1/t$ in the second integral. The third equation follows from Lemma~\ref{lem:integral1}, and the last equation follows from the reflection relation \eqref{eq:reflection}.
\end{proof}
\begin{proof}[Proof of Lemma~\ref{lem:lem4.8-int}]
 Let $I= \int_0^\infty \Big(\frac{\mu^{4/\gamma^2}  (\mu-1)(\mu^{\theta-1}-1)}{ (\mu^{4/\gamma^2} - 1)^2 }  + \frac{1}{\mu^{4/\gamma^2-1}} \Big) d \mu$. Setting $z = \mu^{4/\gamma^2}$, we have
    \begin{align*}
    I &= \frac{\gamma^2}{4} \int_0^\infty \Big( \frac{z^{\frac{\gamma^2}{4}} (z^{\frac{\gamma^2}{4}} -1) (z^{\frac{\gamma^2}{4}(\theta-1)}-1) }{(z-1)^2} + z^{\frac{\gamma^2}{2}-2} \Big)dz\\
    &=\frac{\gamma^2}{4} \lim_{N \rightarrow \infty} \bigg(-\int_0^N  z^{\frac{\gamma^2}{4}} (z^{\frac{\gamma^2}{4}} -1) (z^{\frac{\gamma^2}{4}{(\theta-1)}}-1) \Big(\frac{1}{z-1}\Big)'dz  +  \frac{2}{\gamma^2-2}N^{\frac{\gamma^2}{2}-1} \bigg) \,.
    \end{align*}  In the second equality, we changed the integration interval to $(0,N)$ and then took $N$ to infinity. We also used the identity that $(\frac{1}{z-1})' = \frac{-1}{(z-1)^2}$ and $\int_0^N z^{\frac{\gamma^2}{2}-2} dz = \frac{2}{\gamma^2-2}N^{\frac{\gamma^2}{2}-1}$. By integration by parts, we obtain:
    \begin{equation}
    \label{eq:lem4.10-10}
    \begin{aligned}
        I&=\frac{\gamma^2}{4} \lim_{N \rightarrow \infty} \bigg( \frac{\gamma^2}{4} \int_0^N \frac{z^{\frac{\gamma^2}{4}-1} - z^{\frac{\gamma^2}{4}\theta-1}}{z-1}dz + \frac{\gamma^2}{2}\int_0^N \frac{z^{\frac{\gamma^2}{4}(\theta+1)-1} - z^{\frac{\gamma^2}{2}-1}}{z-1}dz \\
        &\qquad \qquad+\frac{\gamma^2{(\theta-1)}}{4}  \int_0^N \frac{z^{\frac{\gamma^2}{4}(\theta+1)-1} - z^{\frac{\gamma^2}{4}\theta-1}}{z-1}dz -\frac{N^{\frac{\gamma^2}{4}} (N^{\frac{\gamma^2}{4}}-1)(N^{\frac{\gamma^2}{4}{(\theta-1)}}-1) }{N-1} + \frac{2}{\gamma^2-2}N^{\frac{\gamma^2}{2}-1}  \bigg)\,.
    \end{aligned}
    \end{equation}

    By Lemma~\ref{lem:integral2} and the fact that $-1<\frac{\gamma^2}{4}-1,\frac{\gamma^2}{4}\theta-1,\frac{\gamma^2}{4}(\theta+1)-1<0$, we have
    \begin{equation}
        \label{eq:lem4.10-10-1}
       I_1:= \int_0^\infty \frac{z^{\frac{\gamma^2}{4}-1} - z^{\frac{\gamma^2}{4}\theta-1}}{z-1}dz = \pi\Big(\cot(\frac{\pi \gamma^2}{4}\theta) - \cot(\frac{\pi \gamma^2}{4}) \Big)\,,
    \end{equation}
    \begin{equation}
        \label{eq:lem4.10-10-2}
      I_2:=  \int_0^\infty \frac{z^{\frac{\gamma^2}{4}(\theta+1)-1} - z^{\frac{\gamma^2}{4}\theta-1}}{z-1}dz = \pi\Big(\cot(\frac{\pi \gamma^2}{4}\theta) - \cot(\frac{\pi \gamma^2}{4}(\theta+1)) \Big) \,.
    \end{equation}
        Since $\theta<\frac{4}{\gamma^2}-1$, we have $- \frac{N^{\frac{\gamma^2}{4}} (N^{\frac{\gamma^2}{4}}-1)(N^{\frac{\gamma^2}{4}{(\theta-1)}}-1) }{N-1} = N^{\frac{\gamma^2}{2}-1} +o_N(1)$. Therefore by \eqref{eq:lem4.10-10}, we have
\begin{equation}\label{eq:I}
    I= \frac{\gamma^4}{16} I_1 + \frac{\gamma^4{(\theta-1)}}{16} I_2 + \frac{\gamma^2}{4} \lim_{N \rightarrow \infty} \bigg( \frac{\gamma^2}{\gamma^2-2} N^{\frac{\gamma^2}{2}-1} + \frac{\gamma^2}{2}\int_0^N \frac{z^{\frac{\gamma^2}{4}(\theta+1)-1} - z^{\frac{\gamma^2}{2}-1}}{z-1}dz \bigg)\,.
\end{equation}

    For the integral in \eqref{eq:I}, we have
    \begin{equation}
    \label{eq:lem4.10-10-3}
        \begin{aligned}
        &\quad \int_0^N \frac{z^{\frac{\gamma^2}{4}(\theta+1)-1} - z^{\frac{\gamma^2}{2}-1}}{z-1}dz = \int_0^1  \frac{z^{\frac{\gamma^2}{4}(\theta+1)-1} - z^{\frac{\gamma^2}{2}-1}}{z-1}dz + \int_1^N  \frac{z^{\frac{\gamma^2}{4}(\theta+1)-1} - z^{\frac{\gamma^2}{2}-1}}{z-1}dz \\
        &=\int_1^\infty \frac{z^{-\frac{\gamma^2}{2}} - z^{-\frac{\gamma^2}{4}(\theta+1)}}{z-1}dz + \int_1^N  \frac{z^{\frac{\gamma^2}{4}(\theta+1)-1} - z^{\frac{\gamma^2}{2}-2}}{z-1}dz + \int_1^N  \frac{z^{\frac{\gamma^2}{2}-2} - z^{\frac{\gamma^2}{2}-1}}{z-1}dz \\
        &=\psi(\frac{\gamma^2}{4}(\theta+1)) - \psi(\frac{\gamma^2}{2}) + \psi(2-\frac{\gamma^2}{2}) - \psi(1-\frac{\gamma^2}{4}(\theta+1)) +o_N(1) -\frac{2}{\gamma^2-2}(N^{\frac{\gamma^2}{2}-1}-1)\\
        &=-\pi \cot(\frac{\pi \gamma^2}{4}(\theta+1)) + \pi \cot(\frac{\pi \gamma^2}{2}) - \frac{2}{\gamma^2-2}N^{\frac{\gamma^2}{2}-1} + o_N(1) \,.
        \end{aligned}
    \end{equation}
In the second equation, we changed the variable $z $ to $\frac{1}{z}$ in the first integral. The third and fourth equations follow from Lemma~\ref{lem:integral1}, Equation~\eqref{eq:reflection}, and the fact that $-\psi(\frac{\gamma^2}{2}) + \frac{2}{\gamma^2-2} = -\psi(\frac{\gamma^2}{2}-1)$.

    Combining~\eqref{eq:lem4.10-10-1}--\eqref{eq:lem4.10-10-3} and taking $N \rightarrow \infty$, we obtain that
    \begin{align*}
        I&=\frac{\gamma^4}{16} I_1 + \frac{\gamma^4(\theta-1)}{16} I_2 + \frac{\gamma^4 \pi}{8} \Big(-\cot(\frac{\pi \gamma^2}{4}(\theta+1)) + \cot(\frac{\pi \gamma^2}{2}) \Big) =\frac{\pi \gamma^4(\theta \sin(\frac{\pi\gamma^2}{2}) - \sin(\frac{\pi \gamma^2}{2}\theta))}{32\cos(\frac{\pi\gamma^2}{4})  \sin(\frac{\pi\gamma^2}{4}\theta)\sin(\frac{\pi\gamma^2}{4}(\theta+1))}\,. 
    \end{align*}
\end{proof}

\begin{proof}[Proof of Proposition~\ref{lem:3pt-integral}]
By Lemma~\ref{lem:LF5}, we have
\begin{equation}
\label{eq:prop4.10-111}
    \int_0^\infty \frac{1}{C_2}\overline H^{(\frac{2}{\gamma}, \frac{4}{\gamma}-\gamma,\frac{2}{\gamma})}_{(s,1,1)}    \frac{s^{\theta-1}}{1+s} ds= E_4  \int_0^\infty \Big(\int_0^\infty    \frac{\mu^{4/\gamma^2} (\mu-1)^2}{(\mu+s) (\mu^{4/\gamma^2} - 1)^2 } - \frac{1}{\mu^{4/\gamma^2-1}}  d\mu \Big)\frac{s^{\theta-1}}{1+s} ds \,.
\end{equation}
For any $\mu>0$ and $\theta \in (0,1)$, we have:
\begin{align*}
    &\int_0^\infty \frac{s^{\theta-1}}{(\mu+s)(1+s)}ds = \frac{-\pi}{\sin(\pi\theta)} \cdot \frac{\mu^{\theta-1}-1}{\mu-1} \quad \mbox{and} \quad \int_0^\infty \frac{s^{\theta-1}}{1+s}ds =\frac{\pi}{\sin(\pi \theta)} \,.
\end{align*}
Since the integrands in \eqref{eq:prop4.10-111} are negative, we can interchange the order of integration and apply the above identities. Therefore,
\begin{align*}
    \int_0^\infty \frac{1}{C_2}\overline H^{(\frac{2}{\gamma}, \frac{4}{\gamma}-\gamma,\frac{2}{\gamma})}_{(s,1,1)}    \frac{s^{\theta-1}}{1+s} ds =\frac{- E_4 \pi }{\sin(\pi \theta)} \int_0^\infty \Big(\frac{\mu^{4/\gamma^2} (\mu-1)(\mu^{\theta-1}-1)}{ (\mu^{4/\gamma^2} - 1)^2 }  + \frac{1}{\mu^{4/\gamma^2-1}} \Big)d \mu.
\end{align*}
 By Lemma~\ref{lem:lem4.8-int}, we get the desired result.  
\end{proof}

\subsection{Proof of Theorem~\ref{thm:moment} by analytic continuation}\label{sec:analytic}

Our proof is based on analytic continuation. We start with the following lemma.
\begin{lemma}\label{lem:analtic0}
Suppose $x_0, x_1, x_2 \in \mathbb{R}$ with $x_0 < x_1 < x_2$ and $g(\lambda)$ is an analytic function on an open domain $V\subset \mathbb{C}$ that contains $(x_0, x_2)$. Let $\nu$ be a (possibly infinite) measure on positive reals, and let $X$ be a positive random variable distributed according to $\nu$. We have
\begin{enumerate}
    \item If $\nu[e^{-\lambda X}] = g(\lambda)$ for $\lambda \in (x_1, x_2)$, then 
    $\nu[e^{-\lambda X}]$ is finite and analytic for all $\lambda$ with $\mathrm{Re} \lambda> x_0$, so it agrees with $g(\lambda)$ for $\lambda \in V \cap \{\mathrm{Re}\lambda> x_0\}$.

    \item If $\nu[e^{-\lambda X} - 1] = g(\lambda)$ for $\lambda \in (x_1, x_2)$, then 
    $\nu[e^{-\lambda X}-1]$ is finite and analytic for all $\lambda$ with $\mathrm{Re} \lambda> x_0$, so it agrees with $g(\lambda)$ for $\lambda \in V\cap \{\mathrm{Re}\lambda> x_0\}$.
\end{enumerate}
\end{lemma}
\begin{proof}
We start with the first claim. Let $\widetilde \lambda := \inf \{ \lambda \in \mathbb{R} : \nu[e^{-\lambda X}]<\infty \}$. By assumption, $\widetilde \lambda \leq x_1$. Note that $e^{-\lambda X}$ is analytic in $\lambda$, and its integral with respect to $\nu$ can be uniformly controlled in any compact subset of $\{ \mathrm{Re} \lambda> \widetilde \lambda \}$. By Morera's theorem and applying Fubini's theorem, we deduce that $\nu[e^{-\lambda X}]$ is an analytic function in $\lambda$ for $\lambda \in \{ \mathrm{Re} \lambda> \widetilde \lambda \}$. Therefore, by analytic continuation, $\nu[e^{-\lambda X}] = g(\lambda)$ for $\lambda \in V\cap\{\mathrm{Re} \lambda> \widetilde \lambda\}$.

In order to show the first claim, it remains to show that $\widetilde \lambda \leq x_0$.  Suppose that $\widetilde \lambda > x_0$. By taking the derivative of $\nu[e^{-\lambda X}]$ with respect to $\lambda$ and then sending $\lambda$ to $\widetilde \lambda$, we obtain that $\nu[X^n e^{-\widetilde \lambda X}] = (-1)^n g^{(n)}(\widetilde \lambda)$ for any $n \geq 0$, where $g^{(n)}$ is the $n$-th derivative of $g$. Therefore, for $\lambda< \widetilde \lambda$, we have
$
\nu[e^{-\lambda X}] =  \nu [ \sum_{n=0}^\infty \frac{1}{n!}(\widetilde \lambda - \lambda)^n X^{n} e^{-\widetilde \lambda X} ] = \sum_{n=0}^\infty \frac{1}{n!}(\lambda - \widetilde \lambda )^n g^{(n)}(\widetilde \lambda)
$.
We can interchange the summation and integral because the integrands are positive. The summation in the last term is the Taylor expansion of $g(\lambda)$ around $\widetilde \lambda$. By assumption, there exists $\epsilon>0$ such that $g(\lambda)$ is analytic in $B_\epsilon(\widetilde \lambda)$. Hence, for $\lambda \in (\widetilde \lambda -\epsilon, \widetilde \lambda)$, the summation in the last term is absolutely convergent and thus finite. This contradicts with the definition of $\widetilde \lambda$. Therefore, $\widetilde \lambda \leq x_0$, and we obtain the first claim.

The second claim follows in an identical way. Let $\widetilde \lambda := \inf \{ \lambda \in \mathbb{R} : \nu[|e^{-\lambda X} - 1|]<\infty \}$. By assumption, $\widetilde \lambda \leq x_1$. In addition, we have $\nu[X \mathbbm{1}_{\{X \leq 1 \}}] <\infty$, $\nu[\mathbbm{1}_{\{X > 1 \}}] <\infty$, and $\nu[e^{-(\widetilde \lambda + \epsilon) X } \mathbbm{1}_{\{X > 1 \}}] <\infty$ for any $\epsilon>0$. For any compact subset $K$ of $\{ \mathrm{Re} \lambda> \widetilde \lambda \}$, there exist constants $\epsilon, C>0$ that depend on $K$ such that $|e^{-\lambda x} - 1| \leq C(x \mathbbm{1}_{\{x \leq 1 \}} +(1 + e^{-(\widetilde \lambda + \epsilon) x })\mathbbm{1}_{\{x > 1 \}})$ for all $x \geq 0$ and $\lambda \in K$. By Morera's theorem and Fubini's theorem, we deduce that $\nu[e^{-\lambda X} - 1]$ is an analytic function in $\lambda$ for $\lambda \in \{ \mathrm{Re} \lambda> \widetilde \lambda \}$. Therefore, by analytic continuation, $\nu[e^{-\lambda X} - 1] = g(\lambda)$ for $\lambda \in V\cap \{\mathrm{Re} \lambda> \widetilde\lambda\}$. In the same way, we can show that $\widetilde \lambda \leq x_0$, and this gives the desired result.
\end{proof}

We now use the analytic continuation in $\beta$  of $g(\alpha,\beta)$ from Proposition~\ref{prop:g(alpha,beta)} to prove the following.

\begin{lemma}\label{lem:diff}
        For $\alpha \in (\gamma,Q)$, let
    \begin{equation}
    \label{eq:def-fi}
    f(\alpha) = \frac{4C_1\overline{G}(\alpha,\gamma)\Gamma(\frac{4}{\gamma^2}-2)}{\gamma^2 \overline{G}(\alpha,\frac{4}{\gamma}-\gamma)\Gamma(\frac{2}{\gamma}(Q-\alpha))\Gamma (\frac{2}{\gamma}(\alpha-\gamma))}  \int_0^\infty \frac{1}{C_2}\overline H^{(\frac{2}{\gamma}, \frac{4}{\gamma}-\gamma,\frac{2}{\gamma})}_{(s,1,1)}    \frac{s^{\frac{2}{\gamma}(Q-\alpha)-1}}{1+s} ds\,.
    \end{equation}
Then, $\nu_\kappa\left[|\psi'(i)|^{2 \Delta_{\alpha_1} - 2} -|\psi'(i)|^{2 \Delta_{\alpha_2} - 2}\right] = f(\alpha_1) - f(\alpha_2)$ for $\alpha_1,\alpha_2 \in (\gamma,Q)$.
\end{lemma}
\begin{proof}
Fix $\alpha_1,\alpha_2 \in (\gamma, Q)$. Let $\widetilde \alpha = 2(Q-\alpha_1) \vee 2(Q-\alpha_2)$. For $\alpha \in (\gamma, Q)$ and $\beta \in ( 2(Q-\alpha) , Q)$,  let $$h(\alpha, \beta) = \frac{4C_1\overline{G}(\alpha,\gamma)}{\gamma^2C_2 \overline{G}(\alpha,\beta)\Gamma(\frac{2}{\gamma}(Q-\alpha))\Gamma (\frac{2}{\gamma}(\frac{\beta}{2} + \alpha - Q))}  \int_0^\infty \overline H^{(\frac{2}{\gamma}, \beta,\frac{2}{\gamma})}_{(s,1,1)}    \frac{s^{\frac{2}{\gamma}(Q-\alpha)-1}}{1+s} ds\,.$$ 
By Lemma~\ref{lem:analy1}, the integral in the above equation is finite. In the case where $\beta \in (\gamma \vee \widetilde \alpha,Q)$, by Proposition~\ref{prop:g(alpha,beta)}, and taking the difference of $g(\alpha_1,\beta)$ and $g(\alpha_2,\beta)$, we obtain:
    \begin{equation}
    \label{eq:for-gab-d}
    \nu_\kappa[(|\psi'(i)|^{2\Delta_{\alpha_1} - 2} - |\psi'(i)|^{2\Delta_{\alpha_2} - 2})|\varphi'(1)|^{\Delta_\beta-\Delta_{\beta_0}}]= \Gamma(\frac{\beta}{\gamma}-1) (h(\alpha_1, \beta) - h(\alpha_2, \beta))\,.
    \end{equation}
    Next, we will use the first claim in Lemma~\ref{lem:analtic0} to show that \eqref{eq:for-gab-d} holds for any $\beta \in (\widetilde \alpha,Q)$. We first prove that the right-hand side of \eqref{eq:for-gab-d} is analytic for $\beta$ in a complex neighborhood of $(\widetilde \alpha,Q)$. By Lemma~\ref{lem:analy1} and the explicit formula of $\overline G$ in \cite[Theorem 1.7]{RZ22}, $h(\alpha_1, \beta)$ and $h(\alpha_2, \beta)$ are both analytic for $\beta$ in a complex neighborhood of $(\widetilde \alpha, Q) $. Therefore, the function $\Gamma(\frac{\beta}{\gamma}-1) (h(\alpha_1, \beta) - h(\alpha_2, \beta))$ is analytic for $\beta$ in a complex neighborhood of $(\widetilde \alpha, Q) $, except for a possible pole at $\beta = \gamma$. 

    We now show that $\beta = \gamma$ is not a pole through showing that $h(\alpha_1, \gamma) = h(\alpha_2, \gamma)$. By Lemma 3.10 in \cite{RZ22}, we have $\overline H^{(\frac{2}{\gamma}, \gamma,\frac{2}{\gamma})}_{(s,1,1)} = 1$ for any $s>0$, and thus 
    \begin{align*}
        h(\alpha, \gamma) &=\frac{4C_1\overline{G}(\alpha,\gamma)}{\gamma^2C_2 \overline{G}(\alpha,\gamma)\Gamma(\frac{2}{\gamma}(Q-\alpha))\Gamma (\frac{2}{\gamma}(\frac{\gamma}{2} + \alpha - Q))}  \int_0^\infty    \frac{s^{\frac{2}{\gamma}(Q-\alpha)-1}}{1+s} = \frac{4C_1}{\gamma^2 C_2 } \,.
    \end{align*}
    In the second equation, we used the identity $\int_0^\infty \frac{s^{x}}{1+s} ds =  \Gamma(x+1)\Gamma(-x)$ for any $-1<x<0$. Therefore, $h(\alpha_1, \gamma) = h(\alpha_2, \gamma)$. Combining this with the previous argument, we conclude that the right-hand side of \eqref{eq:for-gab-d}, as a function of $\beta$, is analytic in a complex neighborhood of $(\widetilde \alpha, Q)$. 

    Let $\widehat{\nu}_\kappa =(|\psi'(i)|^{2\Delta_{\alpha_1} - 2} - |\psi'(i)|^{2\Delta_{\alpha_2} - 2})\cdot \nu_\kappa$, and let $X = - \log |\varphi'(1)|$. Since $|\varphi'(1)|<1$, we have $X>0$. Let $x_0 = \Delta_{\widetilde \alpha} - \Delta_{\beta_0}, x_1 = \Delta_{\beta_0 \vee \widetilde \alpha} - \Delta_{\beta_0}$, and $x_2 = \Delta_Q - \Delta_{\beta_0}$. Then, $x_0 < x_1 < x_2$. We define the function $g(\lambda)$ in a small complex neighborhood of $(x_0,x_2)$ by requiring $g(\lambda) = \Gamma(\frac{\beta}{\gamma}-1) (h(\alpha_1, \beta) - h(\alpha_2, \beta))$, where $\beta$ is the unique solution to $\lambda = \Delta_\beta - \Delta_{\beta_0}$ with ${\rm Re} \beta < Q$. We see that $g(\lambda)$ is analytic in a small complex neighborhood of $(x_0,x_2)$. By \eqref{eq:for-gab-d}, $\widehat{\nu}_\kappa[e^{-\lambda X}] = g(\lambda)$ for $\lambda \in (x_1,x_2)$. Applying the first claim in Lemma~\ref{lem:analtic0} with $\widehat{\nu}_\kappa$, we obtain that $\widehat{\nu}_\kappa[e^{-\lambda X}] = g(\lambda)$ for $\lambda \in (x_0,x_2)$, and thus \eqref{eq:for-gab-d} holds for any $\beta \in (\widetilde \alpha, Q)$. Taking $\beta = \beta_0$ and using the fact that $f(\alpha) = \Gamma(\frac{\beta_0}{\gamma}-1) h(\alpha, \beta_0)$ yields the lemma.
    \end{proof}

By Lemma~\ref{lem:diff} and Proposition~\ref{lem:3pt-integral}, we get Theorem~\ref{thm:moment} restricted to $\alpha\in (\gamma,Q)$. Namely, we have: 
\begin{proposition}
\label{prop:for-moment}
For $\alpha \in (\gamma, Q)$ and $\theta = \frac{2}{\gamma}(Q-\alpha)$, we have
\begin{equation}
\label{eq:moment-analytic}
\nu_\kappa[|\psi'(i)|^{2 \Delta_{\alpha} - 2} -1] = 1 + \frac{\gamma^2 \Gamma(\frac{\gamma^2(1-\theta)}{4})\Gamma(\frac{\gamma^2(\theta+1)}{4}) }{8 \cos(\frac{\pi \gamma^2}{4}) \Gamma(\frac{\gamma^2}{2}-1) \sin(\frac{\pi \gamma^2 \theta}{4})}\Big(\sin(\frac{\pi \gamma^2\theta}{2}) - \theta\sin(\frac{\pi\gamma^2}{2})\Big)\,.
\end{equation}
\end{proposition}
\begin{proof}
We first compute the constant $E_4$ from Proposition~\ref{lem:3pt-integral}.
 Using the values of $E_1$ from~\eqref{eq:val-e1-e2} and $E_3$ from~\eqref{eq:val-e3-e4}, we obtain that
    \begin{align*}
    &\quad E_4 = \frac{\frac{( 2 \pi)^{\frac{4}{\gamma^2}-1}}{(1-\frac{\gamma^2}{4})\Gamma(1-\frac{\gamma^2}{4})^{\frac{4}{\gamma^2}}} \gamma }{2  \Gamma(\frac{4}{\gamma^2}-2)\Gamma(\frac{4}{\gamma^2}+1)}  \cdot \frac{(2 \pi)^{2-\frac{8}{\gamma^2}}\Gamma(1-\frac{\gamma^2}{4})^{\frac{8}{\gamma^2}}}{\Gamma(2-\frac{4}{\gamma^2})^2}\\
    &=\frac{( 2 \pi)^{-\frac{4}{\gamma^2}+1} \gamma \Gamma(1-\frac{\gamma^2}{4})^{\frac{4}{\gamma^2}}}{2(1-\frac{\gamma^2}{4}) \Gamma(\frac{4}{\gamma^2}-2)\Gamma(\frac{4}{\gamma^2}+1) \Gamma(2-\frac{4}{\gamma^2})^2 } =-\frac{(2 \pi)^{-1-\frac{4}{\gamma^2}} 4\gamma^3(\gamma^2-2) \sin^2(\frac{4 \pi}{\gamma^2}) \Gamma(1-\frac{\gamma^2}{4})^{\frac{4}{\gamma^2}}}{(\gamma^2-4)^2} \,,
    \end{align*}
where the last equation used $\Gamma(x+1) = x \Gamma(x)$ and $\Gamma(\frac{4}{\gamma^2}) \Gamma(1-\frac{4}{\gamma^2}) = \pi/\sin(\frac{4\pi}{\gamma^2})$.
    
In the expression~\eqref{eq:def-fi} for $f(\alpha)$, we plug in the values of $C_1$ from Proposition~\ref{prop:val-c1}, $\overline G( \alpha,\gamma)$ and $\overline G( \alpha,\frac{4}{\gamma}-\gamma)$ from Lemma~\ref{lem:LF4}, $\int_0^\infty \frac{1}{C_2}\overline H^{(\frac{2}{\gamma}, \frac{4}{\gamma}-\gamma,\frac{2}{\gamma})}_{(s,1,1)}    \frac{s^{\frac{2}{\gamma}(Q- \alpha)-1}}{1+s} ds$ from Proposition~\ref{lem:3pt-integral} and $E_4$ from above. This gives 
    \begin{align*}
        f(\alpha) &= \frac{4C_1 \Gamma(\frac{4}{\gamma^2}-2)}{\gamma^2 \Gamma(\frac{2}{\gamma}(Q- \alpha) \Gamma(\frac{2}{\gamma}( \alpha-\gamma))} \times \frac{\overline G( \alpha,\gamma)}{\overline G( \alpha,\frac{4}{\gamma}-\gamma)} \times \int_0^\infty \frac{1}{C_2}\overline H^{(\frac{2}{\gamma}, \frac{4}{\gamma}-\gamma,\frac{2}{\gamma})}_{(s,1,1)}    \frac{s^{\frac{2}{\gamma}(Q- \alpha)-1}}{1+s} ds\\
        &=\frac{4\frac{\pi 2^{1-\frac{\gamma^2}{2}} (\gamma^2-4)^2 \sin(\frac{\pi \gamma^2}{4})}{\gamma^3 \sin(-\frac{4 \pi}{\gamma^2})} \Gamma(\frac{4}{\gamma^2}-2)}{\gamma^2 \Gamma(\frac{2}{\gamma}(Q- \alpha) \Gamma(\frac{2}{\gamma}( \alpha-\gamma))}  \times \frac{\pi^{\frac{4}{\gamma^2}-2}}{2^{3-\frac{\gamma^2}{2}-\frac{4}{\gamma^2}}}\frac{\Gamma(2-\frac{4}{\gamma^2})\Gamma(\frac{\gamma^2}{4})}{\Gamma(1-\frac{\gamma^2}{4})^{\frac{4}{\gamma^2}-1} \Gamma(\frac{\gamma^2}{2}-1)}  \frac{\Gamma(\frac{2 \alpha}{\gamma}-1)\Gamma(\frac{\gamma \alpha}{2}-1)}{\Gamma(\frac{2 \alpha}{\gamma}-\frac{4}{\gamma^2}) \Gamma(\frac{\gamma \alpha}{2}+1-\frac{\gamma^2}{2})}\\
        &\quad \times \frac{-(2 \pi)^{-1-\frac{4}{\gamma^2}} 4\gamma^3(\gamma^2-2) \sin^2(\frac{4 \pi}{\gamma^2}) \Gamma(1-\frac{\gamma^2}{4})^{\frac{4}{\gamma^2}}}{(\gamma^2-4)^2}  \frac{\gamma^4\pi^2(\sin(\frac{\pi \gamma^2 \theta}{2}) - \theta\sin(\frac{\pi\gamma^2}{2})  )}{32\cos(\frac{\pi\gamma^2}{4}) \sin(\pi \theta) \sin(\frac{\pi\gamma^2 \theta}{4})\sin(\frac{\pi\gamma^2( \theta +1)}{4})}\,.
    \end{align*}
    Replacing $ \alpha$ with $Q- \frac{\gamma}{2} \theta $, after some algebraic computations we get \begin{align*}
        f(\alpha)&=\frac{\gamma^2(\gamma^2-2)\sin(\frac{\pi \gamma^2}{4}) \Gamma(\frac{\gamma^2}{4}) \Gamma(1-\frac{\gamma^2}{4}) \sin(\frac{4 \pi}{\gamma^2}) \Gamma(\frac{4}{\gamma^2}-2) \Gamma(2-\frac{4}{\gamma^2})}{16 \cos(\frac{\pi \gamma^2}{4}) \Gamma(\frac{\gamma^2}{2}-1)} \\
        &\quad \times \frac{\Gamma(\frac{4}{\gamma^2}- \theta) \Gamma(\frac{\gamma^2(1- \theta)}{4}) (\sin(\frac{\pi \gamma^2 \theta}{2}) - \theta\sin(\frac{\pi\gamma^2}{2}) )}{\Gamma( \theta)\Gamma(1- \theta)\sin(\pi \theta) \Gamma(\frac{4}{\gamma^2}- \theta-1) \Gamma(2-\frac{\gamma^2( \theta+1)}{4}) \sin(\frac{\pi \gamma^2  \theta}{4}) \sin(\frac{\pi \gamma^2 ( \theta+1)}{4})}\,.
    \end{align*}
    Using the following identities:
    \begin{align*}
        &\sin(\frac{\pi \gamma^2}{4}) \Gamma(\frac{\gamma^2}{4}) \Gamma(1-\frac{\gamma^2}{4}) = \pi\,; \quad\quad   \sin(\frac{4 \pi}{\gamma^2}) \Gamma(\frac{4}{\gamma^2}-2) \Gamma(2-\frac{4}{\gamma^2}) = \frac{\pi}{2-\frac{4}{\gamma^2}}\,;\\
        &\Gamma( \theta)\Gamma(1- \theta)\sin(\pi \theta) = \pi\,; \quad\quad \quad \quad  \frac{\Gamma(\frac{4}{\gamma^2}- \theta)}{\Gamma(\frac{4}{\gamma^2}- \theta-1)} = \frac{4}{\gamma^2}- \theta-1\,;\\
        &\Gamma(2-\frac{\gamma^2( \theta+1)}{4}) \sin(\frac{\pi \gamma^2 ( \theta+1)}{4}) = \frac{\pi(1-\frac{\gamma^2( \theta+1)}{4})}{\Gamma(\frac{\gamma^2( \theta+1)}{4}) }\,, 
    \end{align*}
we have    \begin{align*}
        f(\alpha) &=\frac{\gamma^2(\gamma^2-2)\pi  \frac{\pi}{2-\frac{4}{\gamma^2}}}{16 \cos(\frac{\pi \gamma^2}{4}) \Gamma(\frac{\gamma^2}{2}-1)} \times \frac{ (\frac{4}{\gamma^2}- \theta-1) \Gamma(\frac{\gamma^2(1- \theta)}{4}) (\sin(\frac{\pi \gamma^2 \theta}{2}) - \theta\sin(\frac{\pi\gamma^2}{2})  )}{\pi \sin(\frac{\pi \gamma^2  \theta}{4}) \frac{\pi(1-\frac{\gamma^2( \theta+1)}{4})}{\Gamma(\frac{\gamma^2( \theta+1)}{4}) }}\\
        &=\frac{ \gamma^2 \Gamma(\frac{\gamma^2(1- \theta)}{4})\Gamma(\frac{\gamma^2( \theta+1)}{4}) (\sin(\frac{\pi \gamma^2 \theta}{2}) - \theta\sin(\frac{\pi\gamma^2}{2}) )}{8 \cos(\frac{\pi \gamma^2}{4}) \Gamma(\frac{\gamma^2}{2}-1) \sin(\frac{\pi \gamma^2  \theta}{4})}\,.
    \end{align*}
    For $\alpha_1,\alpha_2 \in (\gamma, Q)$ and $\theta_i = \frac{2}{\gamma}(Q-\alpha_i)$, by Lemma~\ref{lem:diff}, we have $\nu_\kappa[|\psi'(i)|^{2 \Delta_{\alpha_1} - 2} -|\psi'(i)|^{2 \Delta_{\alpha_2} - 2}]$ equals
    \begin{align*}
 \frac{\gamma^2 \Gamma(\frac{\gamma^2(1- \theta_1)}{4})\Gamma(\frac{\gamma^2( \theta_1+1)}{4}) (\sin(\frac{\pi \gamma^2 \theta_1}{2}) - \theta_1\sin(\frac{\pi\gamma^2}{2})  )}{8 \cos(\frac{\pi \gamma^2}{4}) \Gamma(\frac{\gamma^2}{2}-1) \sin(\frac{\pi \gamma^2  \theta_1}{4})} 
 - \frac{ \gamma^2 \Gamma(\frac{\gamma^2(1-\theta_2)}{4})\Gamma(\frac{\gamma^2(\theta_2+1)}{4}) (\sin(\frac{\pi \gamma^2\theta_2}{2}) - \theta_2\sin(\frac{\pi\gamma^2}{2}) )}{8 \cos(\frac{\pi \gamma^2}{4}) \Gamma(\frac{\gamma^2}{2}-1) \sin(\frac{\pi \gamma^2 \theta_2}{4})}\,.
    \end{align*}
  Since $\Delta_\gamma=2$, taking $\alpha_1 = \alpha$ and sending $\alpha_2$ to $\gamma$ (namely, $\theta_2$ to $\frac{4}{\gamma^2}-1$), we conclude the proof.
\end{proof}
We now use the analytic continuation in $\alpha$ of the expression in Proposition~\ref{prop:for-moment} to prove Theorem~\ref{thm:moment}.

\begin{proof}[Proof of Theorem~\ref{thm:moment}]
Let $X=-\log |\psi'(i)|$. Since $|\psi'(i)|<1$,  we have $X>0$. 
Let $\hat g(\theta)$ be the right hand side of~\eqref{eq:moment-analytic}, which is an explicit meromorphic function. Recall that 
$\alpha = Q - \frac{\gamma}{2} \theta$. We define the function  $g(\lambda)$ by requiring 
$g(\lambda)=\hat g(\theta)$ when $\lambda=2\Delta_\alpha-2=-\frac{\gamma^2}{8} \theta^2 + \frac{1}{2}Q^2-2$. Since $f(\theta)=f(-\theta)$, we see that $g(\lambda)$ is a well defined meromorphic function. By inspection, we see that $g$ is analytic  on $\{ \mathrm{Re} z> \frac{2}{\kappa}-1\}$ and has a pole at $\frac{2}{\kappa}-1$. By Proposition~\ref{prop:for-moment}, we have $\nu_\kappa[e^{-\lambda X}-1]= g(\lambda)$ for $\lambda \in (0, 2\Delta_Q-2)$. Setting $\nu=\nu_\kappa$ and $(x_0,x_1,x_2) = (\frac{2}{\kappa}-1, 0, 2\Delta_Q-2)$ in the second claim of Lemma~\ref{lem:analtic0}, we get Theorem~\ref{thm:moment}.
\end{proof}

\section{Proof of Theorem~\ref{thm:transc}} \label{sec:trans}
In this section, we prove Theorem~\ref{thm:transc}. We begin with some background on number theory. The proof of these facts can be found in Chapter 3 of \cite{N56}. For two positive integers $a$ and $b$, let ${\rm gcd}(a,b)$ be their greatest common divisor. Recall that the field of rational numbers is denoted by $\mathbb{Q}$. For an algebraic number $u$, its minimal polynomial is defined as the monic polynomial of the lowest degree with rational coefficients such that $u$ is a root of the polynomial. 
The degree of $u$ refers to the degree of its minimal polynomial, and $u$ is rational if and only if the degree of $u$ is 1. Let $\mathbb{Q}[u]$ be the field extension of $\mathbb{Q}$ by $u$. 

We first recall the minimal polynomials for $\zeta_n := e^{\frac{2 \pi i}{n}}$ and $\zeta_n + \zeta_n^{-1} = 2 \cos(\frac{2 \pi}{n})$. For a positive integer $n \geq 1$, the cyclotomic polynomial $F_n(x)$ and its degree $\phi(n)$ are given by
\begin{equation} \label{eq:cyclotomic}
F_n(x) = \prod_{\substack{k=1\\ {\rm gcd}(k,n)=1}}^{n-1} (x- e^{2 \pi i k/n}) \quad \mbox{and} \quad \phi(n) = n \left(1-\frac{1}{p_1}\right)\left(1-\frac{1}{p_2}\right) \ldots \left(1-\frac{1}{p_r}\right)\,,
\end{equation}
where $p_1, p_2,\ldots,p_r$ range over all the distinct prime factors of $n$. Then the cyclotomic polynomial $F_n(x)$ is the minimal polynomial of $\zeta_n^k$ for all integer $k$ with ${\rm gcd}(k,n)=1$. For $n \geq 3$, we can find a monic polynomial $\psi_n(x)$ with integral coefficients such that $$\psi_n(x+x^{-1}) = x^{-\phi(n)/2} F_n(x)\,.$$ This polynomial $\psi_n(x)$ has degree $\phi(n)/2$ and is the minimal polynomial of $2\cos(\frac{2\pi k}{n})$ for all integer $k$ with ${\rm gcd}(k,n)=1$. In the case of $n=1$ or $2$, the minimal polynomials of $2 \cos(\frac{2 \pi}{1})$ and $2 \cos(\frac{2 \pi}{2})$ are respectively $x-2$ and $x+2$. We will now prove a lemma concerning rational numbers.
\begin{lemma}
    \label{lem:rational}
    For any rational numbers $q_1,q_2,q$, we have
    $$
    \cos(\pi q_1) = q \cos (\pi q_2) \quad \mbox{only if} \quad \big[ 2\cos(\pi q_1) \mbox{ and }2\cos (\pi q_2) \mbox{ are integers or } q \in \{0, 1,-1\} \big].
    $$
\end{lemma}
\begin{proof}
    Without loss of generality, we assume that $q_1,q_2 \geq 0$. For any rational number $r \geq 0$, we can find a pair of positive integers $(n,k)$ with ${\rm gcd}(n,k) = 1$ such that $\cos(\pi r) = \cos(\frac{2 \pi k}{n})$. This is possible because if $r = 0$, we can choose $(n, k) = (1, 1)$, and if $r > 0$, we can always find a pair of positive integers $(n, k)$ with ${\rm gcd}(n, k) = 1$ such that $r = \frac{2k}{n}$. Therefore, it suffices to show that for any positive integers $n_1,n_2,k_1,k_2$ with ${\rm gcd}(n_1,k_1) = {\rm gcd}(n_2,k_2) = 1$ and rational number $q$,
    \begin{equation}
    \label{eq:rational-eq}
    \cos \left(\frac{2\pi k_1}{n_1} \right) = q \cos \left(\frac{2\pi k_2}{n_2}\right) \: \mbox{only if} \: \left[ 2\cos\left(\frac{2\pi k_1}{n_1}\right) \mbox{ and }2\cos \left(\frac{ 2\pi k_2}{n_2}\right) \mbox{ are integers or } q \in \{0, 1,-1\} \right].
    \end{equation}
    
    Suppose that $n_1,n_2,k_1,k_2$ satisfy the equation in the left part of \eqref{eq:rational-eq}, we will now show that they satisfy the condition in the right part of \eqref{eq:rational-eq}. Let us assume that $q \neq 0$, and that at least one of $2\cos(\frac{2\pi k_1}{n_1})$ and $2\cos (\frac{2\pi k_2}{n_2})$ is not an integer. We will now prove that $q$ must be equal to $1$ or $-1$, which will yield \eqref{eq:rational-eq}. By the fact that the degree of $2 \cos(\frac{2\pi k}{n})$ is $\phi(n)/2$ for all $n \geq 3$ and $k$ with ${\rm gcd}(n,k)=1$, we see that $\cos(\frac{2\pi k}{n})$ is a rational number if and only if $n \in \{1,2,3,4,6\}$. Furthermore, in the case where $n \in \{1,2,3,4,6\}$, by enumerating all possible values, we have that $2\cos(\frac{2\pi k}{n})$ must be an integer. 
    Combining with the assumption that at least one of $2\cos(\frac{2\pi k_1}{n_1})$ and $2\cos(\frac{2\pi k_2}{n_2})$ is not an integer, along with \eqref{eq:rational-eq} and $q \neq 0$, we see that both $\cos(\frac{2\pi k_1}{n_1})$ and $ \cos (\frac{2\pi k_2}{n_2})$ are irrational numbers, and $\phi(n_1),\phi(n_2) \geq 3$. 

    Without loss of generality, we assume that $n_1\geq n_2$. Since $ \cos(\frac{2\pi k_1}{n_1}) = q  \cos(\frac{2\pi k_2}{n_2})$ and $q \neq 0$, we have \begin{equation}
    \label{eq:field-eq}
    \mathbb{K}:=\mathbb{Q}\left[2 \cos\left(\frac{2\pi k_1}{n_1}\right)\right] = \mathbb{Q}\left[2 \cos\left(\frac{2\pi k_2}{n_2}\right)\right]\,.
    \end{equation}
    On the other hand, we know that $\mathbb{K}$ is a subfield of $\mathbb{Q}[\zeta_{n_1}]$, and also of $\mathbb{Q}[\zeta_{n_2}]$. Let $m = {\rm gcd}(n_1,n_2)$, then $\mathbb{K}$ is a subfield of $\mathbb{Q}[\zeta_{n_1}] \cap \mathbb{Q}[\zeta_{n_2}] = \mathbb{Q}[\zeta_m]$ (for the equality, see e.g.\ Theorem 4.27(v) in \cite{N04}). Therefore, the degree of $\zeta_m$ should be greater than those of $2 \cos(\frac{2\pi k_1}{n_1})$ and $2 \cos(\frac{2\pi k_2}{n_2})$. Moreover, these two latter numbers have the same degree (recall that their ratio is equal to $q$, which is rational). Therefore, we have
    \begin{equation} \label{eq:rel_phi}
    \frac{1}{2}\phi(n_1) = \frac{1}{2}\phi(n_2) \leq \phi(m) = \phi({\rm gcd}(n_1,n_2))\,.
    \end{equation}
    Using the formula \eqref{eq:cyclotomic} for $\phi$, we can obtain from \eqref{eq:rel_phi} that any prime number $p \geq 5$ must appear with the same exponent in $n_1$ and $n_2$. Indeed, we would have otherwise that $\phi(n_1)$ or $\phi(n_2)$ is $\geq (p-1)\phi({\rm gcd}(n_1,n_2))$. Hence, by considering the possible powers of $2$ and $3$ appearing in the decompositions of $n_1$ and $n_2$ into prime factors, we see that there are only three possible cases where \eqref{eq:rel_phi} is satisfied and $\phi(n_1),\phi(n_2) \geq 3$, that we now consider successively.
    \begin{enumerate}
    \item $n_1 = n_2$. In this case, $2 \cos(\frac{2\pi k_1}{n_1})$ and $2 \cos(\frac{2\pi k_2}{n_2})$ have the same minimal polynomial $\psi_{n_1}(x)$. Using \eqref{eq:rational-eq}, we deduce that $2 \cos(\frac{2\pi k_2}{n_2})$ is also a root of the monic polynomial $q^{-\phi(n_1)/2} \psi_{n_1}(qx)$, and thus $\psi_{n_1}(x) = q^{-\phi(n_1)/2} \psi_{n_1}(qx)$. By comparing the constant coefficients and using the fact that $q$ is rational, we deduce that $q$ must be equal to $1$ or $-1$.
    \item $n_1 = 2n_2$ and $n_2$ is odd. We further divide our analysis into two subcases:
    \begin{enumerate}
        \item If $k_1$ is odd, then we have $q \cos(\frac{2 \pi k_2}{n_2}) = \cos(\frac{2 \pi k_1}{n_1}) =- \cos(\frac{ 2\pi((k_1+n_2)/2)}{n_2}) 
                $.
        \item If $k_1$ is even, then we have $
            q \cos(\frac{2 \pi k_2}{n_2}) = \cos(\frac{2 \pi k_1}{n_1}) = \cos(\frac{2 \pi (k_1/2)}{n_2})
            $.
        \end{enumerate}
        In both subcases, we can reduce the problem to case (1) and deduce that $q \in \{-1,1\}$.
    \item $n_1 = 3m$, $n_2 = 2m$, $m$ is an even number and ${\rm gcd}(3,m) = 1$. We now show that this case is impossible. Since $\zeta_m$ has degree $\phi(m) = \phi(n_1)/2$ (using the condition ${\rm gcd}(3,m) = 1$), we have $\zeta_m \in \mathbb{K}$. Since $2\cos(\frac{2\pi k_2}{n_2}) \in \mathbb{K}$ and $\cos(\frac{2\pi m' k_2}{n_2})$ can be expressed as an integer-valued polynomial of $\cos(\frac{2\pi k_2}{n_2})$ for all $m' \geq 1$, it follows that $2\cos(\frac{2\pi k_2 m'}{n_2}) \in \mathbb{K}$ for all $m' \geq 1$. Consequently, using that ${\rm gcd}(k_2,n_2) = 1$, we have $2\cos(\frac{2\pi}{n_2}) \in \mathbb{K}$, and thus $2\cos(\frac{2\pi }{n_2}) \zeta_m /(\zeta_m+1) = \zeta_{n_2} \in \mathbb{K}$. However, $\zeta_{n_2}$ has degree $\phi(n_2) > \phi(n_2)/2$, which leads to a contradiction.
    \end{enumerate}
    Therefore, we conclude that $q$ must be equal to $1$ or $-1$. This completes the proof of the lemma.
\end{proof}

We now present the proof of Theorem~\ref{thm:transc}.

\begin{proof}[Proof of Theorem~\ref{thm:transc}]
    Let $\rho = \sqrt{\kappa \xi/2+(1-\kappa/4)^2}$. Then, $\rho \in (\frac{\kappa}{4}-1,\frac{\kappa}{4}) \subset (0,2)$ and from \eqref{eq:backbone-generick}, it solves the equation:
    $$
    \sin \left(\frac{8 \pi}{\kappa}\right)\rho = \sin\left(\frac{8 \pi \rho}{\kappa}\right)\,.
    $$
    We argue by contradiction, and assume that $\rho$ is algebraic. Since $\kappa$ is rational, $\sin(\frac{8 \pi}{\kappa})$ is algebraic. Then $x := e^{i\frac{8 \pi}{\kappa} \rho} = (-1)^{\frac{8 \rho}{\kappa}}$ is also algebraic, as it satisfies $\sin(\frac{8 \pi}{\kappa}) \rho - \frac{1}{2i}(x - x^{-1}) = 0$. By the Gelfond-Schneider theorem (see e.g.\ Theorem 10.1 in \cite{N56}), $\frac{8 \rho}{\kappa}$ must be rational. Therefore, $\rho$ itself is necessarily rational. 
    
    We then use that $\cos(\frac{8\pi}{\kappa} - \frac{\pi}{2}) \rho = \cos(\frac{8 \pi \rho}{\kappa} - \frac{\pi}{2})$. Since $\rho$ is rational, Lemma~\ref{lem:rational} implies that $$
    \mbox{either }2 \cos\left(\frac{8\pi}{\kappa} - \frac{\pi}{2}\right) \mbox{ and }2\cos\left(\frac{8 \pi \rho}{\kappa} - \frac{\pi}{2}\right)\mbox{ are integers, or } \rho \in \{0,1,-1\}\,.$$ 
    
    We now discuss the two cases above separately:
    \begin{enumerate}
        \item If $\rho \in \{0,1,-1\}$, combined with the positivity of $\rho$, we have $\rho = 1$. Then $\xi = 1-\frac{\kappa}{8}$ is a double root to the equation \eqref{eq:backbone-generick}. So, it is also a root of the derivative of \eqref{eq:backbone-generick} with respect to $\xi$, which implies after some simple calculations that $\tan(\frac{8\pi}{\kappa}) = \frac{8 \pi}{\kappa}$. Note that $\tan(\frac{8 \pi}{\kappa})$ is an algebraic number (since $\kappa$ is rational), while $\frac{8 \pi}{\kappa}$ is a transcendental number. This leads to a contradiction.
        \item If $2 \cos(\frac{8\pi}{\kappa} - \frac{\pi}{2})$ and $2\cos(\frac{8 \pi \rho}{\kappa} - \frac{\pi}{2})$ are both integers, then they can only take values in $\{-2,-1,0,1,2\}$. Combined with the fact that $\rho \in (0,2)$ and $2\cos(\frac{8\pi}{\kappa} - \frac{\pi}{2})<0$, there are only three possible subcases:
        \begin{enumerate}
            \item $2\cos(\frac{8\pi}{\kappa} - \frac{\pi}{2}) = -2$, $2\cos(\frac{8 \pi \rho}{\kappa} - \frac{\pi}{2}) = -2$ and $\rho = 1$.
            \item $2\cos(\frac{8\pi}{\kappa} - \frac{\pi}{2}) = -1$, $2\cos(\frac{8 \pi \rho}{\kappa} - \frac{\pi}{2}) = -1$ and $\rho = 1$.
            \item $2\cos(\frac{8\pi}{\kappa} - \frac{\pi}{2}) = -2$, $2\cos(\frac{8 \pi \rho}{\kappa} - \frac{\pi}{2}) = -1$ and $\rho = \frac{1}{2}$.
        \end{enumerate}
        The first two subcases are impossible based on the discussion in case (1). The third subcase is also impossible by solving for the value of $\kappa$.
    \end{enumerate}
    In conclusion, $\rho$ is a transcendental number, and so is $\xi$. 
\end{proof}

\appendix

\section{Background on Bernoulli percolation} \label{sec:app_perc}

In this short section, we give more details on the results that we use for two-dimensional Bernoulli percolation at criticality. In particular, we provide a proof of the quasi-multiplicativity property for two black arms, which plays a central role when transferring the backbone exponent back to the discrete model. This proof uses a standard property of percolation known as the Russo-Seymour-Welsh (RSW) lemma, which can be stated as follows. Consider the rectangle $R_n = ([0,4n] \times [0,n]) \cap V_{\bbT}$ on the lattice $\bbT$, $n \geq 1$. At criticality ($p = p_c$), the probability that there exists a horizontal crossing in the difficult direction, i.e. a black path connecting a vertex on the left side of $R_n$ to a vertex on its right side (in other words, two vertices which have a neighbor with $x$-coordinate $< 0$ and $> 4n$, resp.), is bounded away from $0$ uniformly in $n$.

\begin{figure}[t]
\centering
\includegraphics[width=0.5\textwidth]{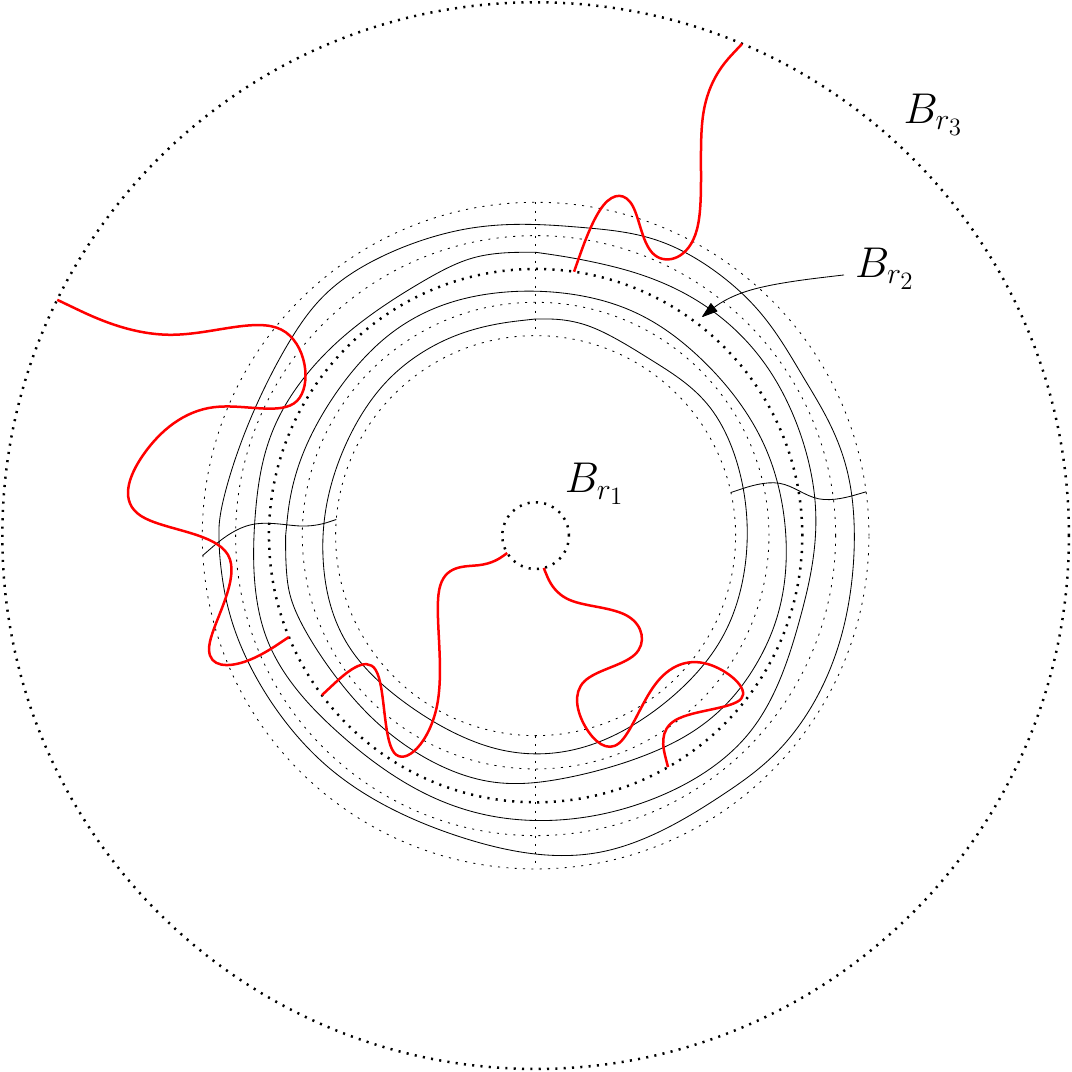}
\caption{This figure shows the construction used for the proof of Lemma~\ref{lem:qm}. We require the existence of four black circuits, one in each of the annuli $A_{(1+i\delta) r_2, (1+(i+1)\delta) r_2}$, $i \in \{-2,-1,0,1\}$, for some small fixed $\delta >0$, e.g. $\delta = \frac{1}{10}$. In addition we ask for the existence of two black paths, each connecting $\dout B_{(1-2\delta) r_2}$ and $\din B_{(1+2\delta) r_2}$, respectively in $A_{(1-2\delta) r_2, (1+2\delta) r_2} \cap \{\Re z > 0\}$ and $A_{(1-2\delta) r_2, (1+2\delta) r_2} \cap \{\Re z < 0\}$. All these paths can be used to connect the two pairs of arms, in $A_{r_1,r_2}$ and $A_{r_2,r_3}$, drawn in red.}
\label{fig:qm}
\end{figure}

This property can then be combined with a correlation inequality that we recall now, the Harris inequality. An event $A \subseteq \Omega$ is said to be increasing if its indicator function $\mathbbm{1}_A$ is non-decreasing for the natural partial order on $\Omega$, that is: for all $\omega$, $\omega' \in A$, $[\omega \in A$ and for all $v \in V_{\bbT}$,  $\omega_v \leq \omega'_v] \implies \omega' \in A$. The Harris inequality then states that for any two increasing events $A, B \subseteq \Omega$,
\begin{equation} \label{eq:Harris}
\text{for each $p \in [0,1]$,} \quad \bbP_p(A \cap B) \geq \bbP_p(A) \bbP_p(B).
\end{equation}
This inequality is a special case, for Bernoulli percolation, of the Fortuin-Kasteleyn-Ginibre (FKG) inequality, which applies to a variety of measures arising in statistical mechanics. It is a convenient tool that allows one to combine paths of the same color. Together with the RSW bounds, it can be used to deduce that at $p_c$, the existence of any prescribed ``macroscopic'' connections can be ensured with uniformly positive probability, which is what we use in the next proof.

\begin{proof}[Proof of Lemma~\ref{lem:qm}]
Let $r_1$, $r_2$, $r_3$ be as in the statement of the lemma. As noted above, the inequality $\pi_{BB}(r_1,r_3) \leq \pi_{BB}(r_1,r_2) \pi_{BB}(r_2,r_3)$ is clear, so we only need to prove that
$$c_1 \cdot \pi_{BB}(r_1,r_2) \pi_{BB}(r_2,r_3) \leq \pi_{BB}(r_1,r_3).$$
For this purpose, we can use the construction depicted in Figure~\ref{fig:qm}, denoting by $E_{r_2}$ the corresponding event. It follows from multiple applications of the RSW bound mentioned above, combined by applying repeatedly \eqref{eq:Harris}, that for some universal constant $c_1 > 0$,
\begin{equation} \label{eq:const_RSW}
\bbP_{p_c}(E_{r_2}) \geq c_1.
\end{equation}
If in addition to $E_{r_2}$, both $\cA_{BB}(A_{r_1,r_2})$ and $\cA_{BB}(A_{r_2,r_3})$ occur, then Menger's theorem implies that $\cA_{BB}(A_{r_1,r_3})$ occurs as well. Indeed, if we consider the union of the two pairs of arms (in each of $A_{r_1,r_2}$ and $A_{r_2,r_3}$) together with all paths requested in $E_{r_2}$, then at least two vertices need to be removed in order to disconnect $\dout B_{r_1}$ and $\din B_{r_3}$. Hence,
\begin{align*}
\pi_{BB}(r_1,r_3) = \bbP_{p_c}(\cA_{BB}(A_{r_1,r_3})) & \geq \bbP_{p_c}(E_{r_2} \cap \cA_{BB}(A_{r_1,r_2}) \cap \cA_{BB}(A_{r_2,r_3}))\\
& \geq \bbP_{p_c}(E_{r_2}) \cdot \bbP_{p_c}(\cA_{BB}(A_{r_1,r_2})) \cdot \bbP_{p_c}(\cA_{BB}(A_{r_2,r_3}))\\
& \geq c_1 \pi_{BB}(r_1,r_2) \pi_{BB}(r_2,r_3),
\end{align*}
where we used \eqref{eq:Harris} for the second inequality, since all three events which appear are increasing, and \eqref{eq:const_RSW} for the third one.

\end{proof}

\section{Proof of Description \ref{descrip}}\label{app:descrip}
In this section, we prove the statements in Description \ref{descrip}.
In Section~\ref{subsec:chordal_sle}, we define the chordal $\mathrm{SLE}_\kappa(\kappa-6)$ and provide some details about the description in \cite[Section 6.1]{MR3708206}: The sequence of loops in the CLE$_\kappa$ discovered by the chordal  $\mathrm{SLE}_\kappa(\kappa-6)$ induces a Poisson point process of SLE$_\kappa$ bubbles.
In Section~\ref{subsec:radial_sle}, we explore along a radial $\mathrm{SLE}_\kappa(\kappa-6)$ instead: This induces a sequence of bubbles which relates to the sequence of bubbles in Section~\ref{subsec:chordal_sle} by a suitable time change and time-dependent conformal maps.
It then takes some technical work to prove that the new sequence of bubbles is also a Poisson point process. 

\subsection{Chordal $\mathrm{SLE}_\kappa(\kappa-6)$ and $\mathrm{SLE}_\kappa$ bubbles}\label{subsec:chordal_sle}
Let us first define the chordal SLE$_\kappa(\kappa-6)$ process, for $\kappa\in(4,8)$. Recall that a chordal Loewner evolution  $(K_t)_{t>0}$ in $\mathbb{H}$ is parametrized by the equation
\begin{align}\label{eq:loewner}
\partial_t g_t =\frac{2dt}{g_t(z)-W_t}, \qquad g_0(z)=z,
\end{align}
where $g_t$ is the unique conformal map from $\bbH\setminus K_t$ onto $\bbH$ with $g_t(z)=z+ 2t/z + o(1/z)$ as $z\to \infty$, and $W_t$ is the driving function. Fix $a>0$.
Let $Y_t$ be a squared Bessel process with dimension $d=3-8/\kappa$, started at $Y_0=a^2/\kappa$. For $\kappa \in (4,8)$, we have $d \in (1,2)$,  so $Y_t$ a.s.\ hits $0$ infinitely many times. Let $X_t:=\sqrt{Y_t}$.
A chordal SLE$_\kappa(\kappa-6)$ started from $0$ targeting $\infty$ with marked point $-a$ 
is parametrized by \eqref{eq:loewner} with a driving function $W_t$ given by 
\begin{align}\label{eq:sle}
O_t= -a-2\kappa^{-1/2} \int_0^t X_s^{-1} ds, \quad W_t=O_t+\sqrt{\kappa} X_t.
\end{align}
We have $O_t=g_t(-a)$. 
It was shown in  \cite{MR3477777} that SLE$_\kappa(\kappa-6)$ is a.s.\ a continuous curve. 

For all $d>0, x>0$, we use $\bbP^d_x$ to denote the law of a Bessel process with dimension $d$ started from $x$.
In \cite{MR0656509}, an infinite measure on Bessel bridges from $0$ to $0$ was constructed as
\begin{align}\label{eq:bessel_bridge}
 \mu_{0,0}^d: =\lim_{\eps\to 0} \eps^{2-d} \bbP^d_\eps.
\end{align}
In our case, as $d=3-8/\kappa$, the exponent $2-d=-1+8/\kappa$ coincides with the one in \eqref{eq:sle_bubble}.
When $d<2$, it was further shown in \cite{MR0656509} that $X_t$ (i.e., the process previously defined which is a reflected $d$ dimensional Bessel process) is equal to the concatenation of its initial part from $X_0$ to $0$ and a sequence of bridges given by a Poisson point process with intensity $\mu_{0,0}^d$.
The measure on Bessel bridges gives rise the following measure on SLE$_\kappa$ bridges. 
\begin{definition}[SLE$_\kappa$ bridges]\label{def:bridge}
Let $(\wt X_t, 0\le t \le \tau )$ be a Bessel bridge sampled according to $c(\kappa) \mu_{0,0}^d$, where $c(\kappa)$ is the constant in \eqref{eq:sle_bubble}. If we inject $\wt X_t$ in the place of $X_t$ in \eqref{eq:sle}, then the resulting curve $(\omega_t, 0\le t\le \tau)$ is called an SLE$_\kappa$ bridge.
\end{definition}
\begin{lemma}\label{lem:sle_bridge}
Let $\omega$ be an SLE$_\kappa$ bridge. Let $e$ be an SLE$_\kappa$ bubble, distributed according to $c(\kappa)^{-1}\mu_\kappa$ (where $c(\kappa)$ and $\mu_\kappa$ are defined in \eqref{eq:sle_bubble}), stopped at the first time that $e$ hits $\bbR^-\setminus\{0\}$. Then $\omega$ is equal in distribution to $e$, as curves modulo monotone time reparametrization.
\end{lemma}
\begin{proof}
Definition~\ref{def:bridge} and \eqref{eq:bessel_bridge} imply that the distribution of $\omega$ is equal to the limit as $\eps\to 0$ of $\eps^{-1+8/\kappa}$ times the probability law of a chordal SLE$_\kappa(\kappa-6)$ process $\omega_\eps$ from $0$ to $\infty$ with marked point $-\eps$, stopped at the first time that it disconnects $-\eps$ from $\infty$. By  \cite{MR2188260}, $\omega_\eps$ is further equal in distribution to an SLE$_\kappa$ from $0$ to $-\eps$ in $\bbH$, modulo monotone time reparametrization, also stopped at the first time that it disconnects $-\eps$ from $\infty$. Letting $\eps\to 0$ completes the proof.
\end{proof}

Now, let $a:=0$, so that $X_t$ is exactly equal to the concatenation of Bessel bridges $\{b_s, s\in J\}$, where $\{b_s, s\in J\}$ is a time-indexed Poisson point process with intensity $c(\kappa) \mu_{0,0}^d$ (the constant $c(\kappa)$ only rescales the time clock of the Poisson point process) and
$J\subset \bbR^+$ is the set of times $s>0$ at which a Bessel bridge $b_s$ occurs. In this case, the chordal SLE$_\kappa(\kappa-6)$ from $0$ to $\infty$ with marked point $0^-$, generated by \eqref{eq:sle}, is equal to the composition (as Loewner chains) of a Poisson point process of SLE$_\kappa$ bridges $\{\omega_s, s\in J\}$, so that each SLE$_\kappa$ bridge $\omega_s$ is generated by  \eqref{eq:sle} where we replace $X_t$ by $b_s$. 
For each SLE$_\kappa$ bridge $\omega$, let $x(\omega)$ be the left-most point on $\omega\cap\bbR$. Let $\varphi_\omega$ be the conformal map from the unbounded connected component of $\bbH\setminus \omega$ onto $\bbH$, with $\varphi_\omega(x(\omega))=0$ and $\varphi_\omega(z)\sim z$ as $z\to \infty$. For all $u>0$, let $\Phi_u$ be the composition of the conformal maps $\varphi_{\omega_s}$ for all $s\in J$ with $s\le u$, in the order of their appearance. For each $s\in J$, let $a(\omega_s)$ be the chordal capacity of $\omega_s$. Let
\begin{align*}
t(u):=\sum_{s\in J, s\le u} a(\omega_s).
\end{align*}
Then $\Phi_u(z)= g_{t(u)}(z)-W_{t(u)}$. These definitions can be extended to $u^-$ if we replace $s\le u$ by $s<u$.
The construction of CLE$_\kappa$ using the branching SLE$_\kappa(\kappa-6)$ by \cite{MR2494457} (see Section~\ref{subsec:link_bubble}) ensures that for each $s\in J$, $\Phi_{s^-}^{-1}(\omega_s)$ is part of a loop $\wt\ell_s$ in the  CLE$_\kappa$. Moreover, conditionally on $\omega_s$, the part $\Phi_{s^-}(\wt\ell_s)\setminus \omega_s$ is equal to an SLE$_\kappa$ from $x(\omega_s)$ to $0$ in the bounded connected component of $\bbH\setminus \omega_s$ that intersects $\bbR^-$. For $s\in J$, let $\wt e_s:=\Phi_{s^-}(\wt\ell_s)$.
Combined with Lemma~\ref{lem:sle_bridge}, we conclude that $\{\wt e_s, s\in J\}$ is a Poisson point process of SLE$_\kappa$ bubbles with intensity $\mu_\kappa$. Thus we have obtained the description given in \cite[Section 6.1]{MR3708206}.

As pointed out in \cite[Section 6.1]{MR3708206}, the set of loops $\{\wt\ell_s, s\in J\}$ is a strict subset of the non-nested CLE$_\kappa$. By the Markov property of the branching SLE$_\kappa(\kappa-6)$, conditionally on $\{\wt\ell_s, s\in J\}$, to obtain the remaining loops in the non-nested CLE$_\kappa$, it suffices to sample an independent non-nested CLE$_\kappa$ in each of the connected components of $\bbH\setminus \cup_{s\in J}\wt \ell_s$ which is not surrounded by any loop in $\{\wt \ell_s, s\in J\}$. This can be achieved by iterating the exploration process above.

\subsection{Radial  $\mathrm{SLE}_\kappa(\kappa-6)$ and proof of Description \ref{descrip}}
\label{subsec:radial_sle}

The radial $\mathrm{SLE}_\kappa(\kappa-6)$ was defined in \cite[Proposition 3.15]{MR2494457} by concatenating pieces of (time-changed) chordal SLE$_\kappa(\kappa-6)$ curves. The radial $\mathrm{SLE}_\kappa(\kappa-6)$ was shown in \cite[Proposition 3.16]{MR2494457} to be target-invariant. For example, a radial $\mathrm{SLE}_\kappa(\kappa-6)$ in $\bbH$ from $0$ to $i$ with marked point $0^-$ coincides (up to a time change) with a chordal SLE$_\kappa(\kappa-6)$  in $\bbH$ from $0$ to $\infty$ with marked point $0^-$, up to the first time that $i$ is disconnected from $\infty$ in $\bbH$. Note that even though the radial and chordal SLE$_\kappa(\kappa-6)$ curves coincide up to some stopping time, the sequences of bubbles they induce are not the same, since we do not use the same normalisation of conformal map for the two processes, and SLE$_\kappa$ bubbles are only conformally covariant instead of invariant.

In Description \ref{descrip}, we explore the CLE$_\kappa$ $\Gamma$ along a radial SLE$_\kappa(\kappa-6)$ curve $\eta$ from $0$, with marked point $0^-$, targeting $i$. This exploration discovers a sequence of loops $\{\ell_s, s\in I\}$, where $I\subset\bbR^+$ is a parametrization of the loops that indicates their order of appearance.  While the ordered sequence of loops is determined by $\Gamma$ and $\eta$, the exact time-parametrization by $I\subset\bbR^+$ is only defined up to monotone reparametrization for the moment. Later on, we will choose a specific $I$, and prove that $\{\ell_s, s\in I\}$ indeed gives rise to a Poisson point process of SLE$_\kappa$ bubbles $\{e_s, s\in I\}$ as claimed in Description \ref{descrip}. Note that $\{\ell_s, s\in I\}$ and $\{e_s, s\in I\}$ mutually determine each other, as explained in Description \ref{descrip}.

The target-invariance of SLE$_\kappa(\kappa-6)$ ensures that $\eta$ is equal in distribution to the chordal SLE$_\kappa(\kappa-6)$ defined in Section~\ref{subsec:chordal_sle} (modulo time change) up to the first time $\sigma$ that $\eta$ disconnects $i$ from $\infty$ in $\bbH$.
By Section~\ref{subsec:chordal_sle}, $\eta([0,\sigma])$ has discovered a sequence of loops $\{\wt \ell_s, s\in J_0\}$, where $J_0=J \cap [0, s_0]$ and $s_0$ is the unique point in $J$ such that $\eta(\sigma)$ belongs to $\wt\ell_{s_0}$. 
Consequently, there is an increasing function $\chi: [0,s_0] \to \bbR^+$  such that $I\cap [0,\chi(s_0)]= \{\chi(s), s\in J_0\}$ and $\ell_{\chi(s)}=\wt \ell_s$ for all $s\in J_0$. Below, we will show that for some explicit $\chi$, the sequence $\{\ell_s, s\in I, s\le \chi(s_0)\}$ is indeed related to a Poisson point process of SLE$_\kappa$ bubbles.

According to Section~\ref{subsec:chordal_sle},  $\{\wt \ell_s, s\in J_0\}$ induces a Poisson point process of SLE$_\kappa$ bubbles $\{\wt e_s, s\in J_0\}$ stopped at time $s_0$, where $\wt e_s= \Phi_{s^-}(\wt\ell_s)$.
 According to Description~\ref{descrip}, $\{\ell_s, s\in I, s\le \chi(s_0)\}$ also induces a sequence of bubbles $\{e_s, s\in I, s\le \chi(s_0)\}$, where $e_s=\Psi_{s^-}(\ell_s)$.  
Recall the definition of  $O_{u^-}$ in Description \ref{descrip}, which is the connected component containing $i$ of $\bbH \setminus \cup_{s\in I, s<u} \ep\ell_s$. Even though we have not yet specified the function $\chi$, the domain $O_{\chi(s)^-}$ for $s\le s_0$ does not depend on the choice of $\chi$: For all $s\le s_0$, 
both $\Psi_{\chi(s)^-}$ and $\Phi_{s^-}$ are conformal maps from $O_{\chi(s)^-}$ onto $\bbH$, with different normalizations. In particular,  $e_{\chi(s)}= f_{s} (\wt e_s)$, where $f_s:=\Psi_{\chi(s)^-}\circ \Phi_{s^-}^{-1}$. One can see that $f_s$ is the unique conformal map from $\bbH$ onto $\bbH$ which sends $0$ to $0$ and $\Phi_{s^-}(i)$ to $i$, and is given by
\begin{align}\label{eq:fs}
f_s(z)= \frac{-( \Im\Phi_{s^-}(i)) z}{(\Re \Phi_{s^-}(i)) z -|\Phi_{s^-}(i)|^2}.
\end{align}
Note that $f_s$ also does not depend on the choice of $\chi$ and $f_s'(0)\in(0,\infty)$ for all $s\le s_0$.
Let
\begin{align}\label{eq:chi}
\chi(s):=\int_0^s f_u'(0)^{(8-\kappa)/\kappa} du, \quad \text{so that } \chi^{-1}(s)=\int_0^s f_{\chi^{-1}(u)}'(0)^{(\kappa-8)/\kappa} du.
\end{align}
We use the convention that $\chi^{-1}(s)=s_0$ if $s\ge \chi(s_0)$.
Let $\wt \cF_s$ (resp.\ $\wt \cF_{s^-}$) be  the $\sigma$-algebra  generated by $\{\wt \ell_u, u\in J_0, u\le s\}$ (resp.\ $\{\wt \ell_u, u\in J_0, u< s\}$). 
Note that $\chi(s)$ is measurable w.r.t.\  $\wt\cF_{s^-}$. 
Let $\cF_{s}$ (resp.\ $\cF_{s^-}$) be the $\sigma$-algebra generated by $\{\ell_u, u\in I, u\le s, u\le \tau\}$ (resp.\ $\{\ell_u, u\in I, u< s, u\le \tau\}$). Then $\cF_{\chi(s)} =\wt\cF_{s}$ and $\cF_{\chi(s)^-} =\wt\cF_{s^-}$. Let us show that with the choice \eqref{eq:chi}, the following holds.

\begin{proposition}\label{prop:ap_ppp}
 The time-indexed sequence $\{e_s, s\in I, s\le \chi(s_0)\}$ is a Poisson point process with intensity $\mu_\kappa$, stopped at the first time $s>0$ such that the related sequence of loops $\{\ell_s, s\in I\}$ (generated according to Description~\ref{descrip}) contains a loop $\ell_s$ which disconnects $i$ from $\infty$ in $\bbH$.
\end{proposition}

In order to prove Proposition~\ref{prop:ap_ppp}, we need to first collect a few lemmas. 
\begin{lemma}\label{lem:no_atom}
For $r>0$, let $D_r$ be the set of bubbles $e$ such that $\sup_{z\in e} |z| =r$. Then $\mu_\kappa(D_r)=0$.
\end{lemma}
\begin{proof}
Suppose that there exists $r>0$ such that $\mu_\kappa(D_r)>0$.
For $s>0$, let $y_s(z):= sz$ be the scaling of $\bbH$.
Then $y_s(D_r) = D_{sr}$. By conformal covariance \eqref{eq:bubble_cov} of $\mu_\kappa$, we have
$\mu_\kappa(D_{sr}) = s^{(\kappa-8)/\kappa} \mu_\kappa(D_r)>0$. In particular, the $\mu_\kappa$ measure of the set of loops with $\sup_{z\in e} |z| \in [1,2]$ would be infinite, which is impossible. This proves the lemma.
\end{proof}

Let $\lambda(z):=(i-z)/(i+z)$ be a conformal map from $\bbH$ onto $\bbD$. We use the following metric on $\bbH$
\begin{align*}
d_\lambda(z_1, z_2) := |\lambda(z_1) -\lambda(z_2)| \quad \text{ for all } z_1, z_2\in\bbH,
\end{align*}
which  induces the same topology as the Euclidean metric. Under $d_\lambda$, $\ol \bbH$ is compact. For any two compact sets $X, Y\subset\ol\bbH$, we also define the Hausdorff distance
\begin{align*}
d_\mathrm{H}(X, Y):= \max\big\{\sup_{x\in X} d_\lambda(x, Y), \, \sup_{y\in Y} d_\lambda(X,y) \big\}.
\end{align*}

\begin{lemma}\label{lem:cadlag}
Let $0< s_1<\ldots s_n$. For all $\delta>0$, there exists $\eps>0$, such that if $t\in (s_j -\eps, s_j]$ for some $1\le j\le n$, then
\begin{align}\label{eq:f_norm}
\sup_{z\in\bbH} d_\lambda \big(f_{t}(z), f_{s_j}(z)\big) \le \delta.
\end{align}
\end{lemma}
\begin{proof}
The left-hand side of \eqref{eq:f_norm} defines a distance between conformal maps from $\bbH$ onto itself. 
The function $s\mapsto f_s$ is \emph{c\`agl\`ad} for this distance, because $s\mapsto \Phi_{s^-}$ is \emph{c\`agl\`ad} (due to Section~\ref{subsec:chordal_sle}), and $f_s$ is given by \eqref{eq:fs}. By the left-continuity, for each $s_j$, there exists $\eps_j$, such that \eqref{eq:f_norm} holds for all $t\in(s_j-\eps_j, s_j]$. The lemma follows by letting $\eps:=\min(\eps_1,\ldots, \eps_j)$.
\end{proof}

\begin{lemma}\label{lem:total_variation}
Let $h(s):=f_s'(0)^{(8-\kappa)/\kappa}$. For $u>0$, $\eps>0$ and $k\in\bbN$, let $u_k:=u+ k \eps$. Let
\begin{align*}
\theta(u,\eps):=\one_{u< s_0} \sup_{u\le s < (u+\eps) \wedge s_0}  |h(s)-h(u)|, \quad V:=\limsup_{\eps\to 0}\sum_{k=0}^{\infty} \theta(u_{k}, \eps).
\end{align*}
Then $V<\infty$ a.s.
\end{lemma}
\begin{proof}
Taking the derivative of \eqref{eq:fs}, we get
$
f_s'(0)=\Im(\Phi_{s^-}(i))/|\Phi_{s^-}(i)|^2.
$
Note that for each $s\in [0, s_0]$, there exists $t(s)\in[0,\sigma]$ such that $\Phi_{s^-} = g_{t(s)}$. It follows that $ h(s)= H_{t(s)}$ where
$
H_t:=\left(\Im(g_t(i))/|g_t(i)|^2\right)^{(8-\kappa)/\kappa}.
$
The process $H_t$ has finite total variation on the finite interval $[0,\sigma]$, because its Ito derivative (using \eqref{eq:loewner}) has only the $dt$ term.
This implies that $h(s)$ has finite total variation on the interval $[0, s_0]$, which implies the lemma.
\end{proof}

\begin{lemma}\label{eq:festimate}
There exists a (random) $C>0$, such that for all $0\le s< t\le s_0$, 
\begin{align*}
\sup_{\theta \in [0,\pi]} |f_t\circ f_s^{-1} (r e^{i\theta})| \ge Cr.
\end{align*}
\end{lemma}
\begin{proof}
For simplicity, we denote $x_s:= \Re \Phi_{s^-}(i)$, $y_s:= \Im \Phi_{s^-}(i)$ and $r_s:= |\Phi_{s^-}(i)|$. Note that $r_s^2=x_s^2+y_s^2$.
By~\eqref{eq:fs}, we have
\begin{align*}
f_s(z)=\frac{- y_s z}{x_s z - r_s^2}, \quad f_s^{-1}(z)=\frac{r_s^2 z}{z x_s + y_s }, \quad
f_t \circ f_s^{-1}(z)= \frac{-y_t r_s^2 z}{x_t r_s^2 z - x_s r_t^2 z- y_s r_t^2}.
\end{align*}
Fix $r>0$. We can see that $ |f_t\circ f_s^{-1}(r e^{i\theta})| $ achieves its supremum when $\theta=0$ or $\pi$, and 
\begin{align}\label{eq:frnorm}
\sup_{\theta \in [0,\pi]} |f_t \circ f_s^{-1}(r e^{i\theta})|=  \frac{ |y_t |r_s^2 r }{\left| |x_t r_s^2 -x_s r_t^2| r - y_s r_t^2 \right |}.
\end{align}
As we pointed out in the proof of Lemma~\ref{lem:total_variation}, $\Phi_{s^-}(i) =g_{t(s)}(i)$. The process $g_t(i)$ is continuous on $[0,\sigma]$, and $g_t(i)\not \in \bbR$ for all $t< \sigma$. Moreover, $t(s_0)<\sigma$ a.s., because $t(s_0)$ corresponds to the time that the chordal SLE$_\kappa(\kappa-6)$ started to trace the bubble $\ell_{s_0}$. This implies that there exists a (random) $\eps>0$, such that
$\eps^{-1}\ge \Im g_t(i) \ge \eps$ and $\eps^{-1}\ge |g_t(i)|\ge \eps$ for all $t$ in the compact interval $[0, t(s_0)]$.
This implies that
$\eps^{-1}\ge r_s\ge y_s \ge \eps$ for all $s\in[0, s_0]$.
For all $r<\eps/2$, \eqref{eq:frnorm} is at least
\begin{align*}
 \frac{ |y_t |r_s^2 r }{|x_t r_s^2|r +|x_s r_t^2| r +| y_s r_t^2  |} \ge \frac{\eps^{3} r}{\eps^{-2} +\eps^{-3}}.
\end{align*}
The lemma follows by letting $C:=\eps^6/(\eps+1)$.
\end{proof}

\begin{proof}[Proof of Proposition~\ref{prop:ap_ppp}]
We will construct a sequence of (stopped) Poisson point processes $\cN_n$ which converge  to $\cN:=\{e_s, s\in I, s\le \chi(s_0)\}$ in a certain sense that we will specify.
\medbreak

\noindent \textbf{Step 1. Construction of $\cN_n$.} Let $\eps:=1/n$. Let $s_{n,0}:=0$. For all $k\ge 0$, if $s_{n,k} <s_0$, then let $s_{n,k+1}:= s_{n,k} + f_{s_{n,k}}'(0)^{(\kappa-8)/\kappa} \eps$, otherwise let $s_{n, k+1}:= s_0$. For all $k\ge 0$, we have $s_{n,k+1} \in \wt\cF_{s_{n,k}}$. In particular, $s_{n,k}$ is a stopping time and $s_{n,k}\not \in J_0\setminus \{s_0\}$ a.s. Conditionally on $\wt\cF_{s_{n,k}}$ and on $s_{n,k}<s_0$, the process 
$$\wt\cN_{n,k}:=\{\wt e_{s_{n,k}+s}, s_{n,k}+s\in J_0, s_{n,k}+s \le s_{n, k+1}\}$$
is an independent Poisson point process with parameter $\mu_\kappa$, stopped at $(s_0 -s_{n,k})\wedge (s_{n, k+1} -s_{n,k})$. Note that the Poisson point process itself is independent from $\wt\cF_{s_{n,k}}$, but the stopping time depends on $\wt\cF_{s_{n,k}}$. It is clear that $s_{n, k+1} -s_{n,k}$ depends on $\wt\cF_{s_{n,k}}$, and we explain in the next paragraph the dependence of $s_0 -s_{n,k}$ on $\wt\cF_{s_{n,k}}$.

On the event $s_{n,k}< s_0$, let $\Gamma_k$ be the image under $\Psi_{s_{n,k}}$ of the CLE$_\kappa$ restricted to $\ol{O_{\chi(s_{n,k})}}$. Conditionally on $s_{n,k}< s_0$, $\Gamma_k$ is distributed as an CLE$_\kappa$ and is independent from $\wt\cF_{s_{n,k}}$.
The bubbles in $\wt \cN_{n,k}$ determine a sequence of loops $\{\Psi_{s_{n,k}}(\wt \ell_{s+s_{n,k}}), s+s_{n,k}\in J_0, s_{n,k}+s \le s_{n, k+1}\}$ in $\Gamma_k$. The time $s_0 - s_{n,k}$ corresponds to the first time that there is a loop in the previous sequence which disconnects $\Psi_{s_{n,k}}(i)$ from $\infty$ in $\bbH$. 

Let $\wh\cN_{n,k}$ be a new process obtained by rescaling the time of $\wt \cN_{n,k}$ by $f_{s_{n,k}}'(0)^{(8-\kappa)/\kappa}$. This scaling factor is measurable w.r.t.\ $\wt\cF_{s_{n,k}}$, but the Poisson point process which defines $\wt \cN_{n,k}$ (without the stopping time) is independent from $\wt\cF_{s_{n,k}}$. Hence the scaling factor acts like a constant. In particular, $\wh\cN_{n,k}$ is still a Poisson point process,
 with parameter $f_{s_{n,k}}'(0)^{(\kappa-8)/\kappa}\mu_\kappa$, stopped at time $(\wt s_0 - \wt s_{n,k}) \wedge (\wt s_{n, k+1} - \wt s_{n,k})$, where $\wt s_0:=f_{s_{n,k}}'(0)^{(8-\kappa)/\kappa} s_0$ and $\wt s_{n,j}:= f_{s_{n,k}}'(0)^{(8-\kappa)/\kappa} s_{n,j}$ for $j=k, k+1$. Note that
\begin{align*}
\wt s_{n, k+1} -\wt s_{n,k}= f_{s_{n,k}}'(0)^{(8-\kappa)/\kappa} (s_{n, k+1} -s_{n,k})= \eps.
\end{align*}
The time $\tau_k:=\wt s_0 - \wt s_{n,k}$ is the first time that a loop generated by $\wh\cN_{n,k}$ disconnects $\Psi_{s_{n,k}}(i)$ from $\infty$ in $\bbH$ (since $\wh\cN_{n,k}$ contains the same ordered sequence of bubbles as $\wt\cN_{n,k}$, the ordered sequence of loops that it generates is also the same as the one generated by $\wt\cN_{n,k}$, up to a time-rescaling).

Let $\cN_{n,k}$ be another new process obtained by taking the image of $\wh\cN_{n,k}$ under $f_{s_{n,k}}$, namely
\begin{align*}
\cN_{n,k}:=\left\{f_{s_{n,k}} \big(\wt e_{s_{n,k}+ f_{s_{n,k}}'(0)^{(\kappa-8)/\kappa} s}\big),\,\, s_{n,k}+f_{s_{n,k}}'(0)^{(\kappa-8)/\kappa} s\in J_0,  \, s_{n,k}+ f_{s_{n,k}}'(0)^{(\kappa-8)/\kappa} s \le s_{n, k+1} \right\}.
\end{align*}
By conformal covariance \eqref{eq:bubble_cov} of $\mu_\kappa$, we deduce that $\cN_{n,k}$ is still a Poisson point process, with parameter
\begin{align*}
f_{s_{n,k}}'(0)^{(\kappa-8)/\kappa} (f_{s_{n,k}})_*  \mu_\kappa=\mu_\kappa.
\end{align*}
The stopping time for $\cN_{n,k}$ is $\eps \wedge \tau_k$, the same as $\wh \cN_{n,k}$.

Now, let $\cN_n$ be the concatenation of the processes $\cN_{n,0}, \ldots, \cN_{n, k_0}$, where $k_0$ is the biggest $k$ such that $s_{n,k}< s_0$. Then $\cN_n$ is again a Poisson point process, with parameter $\mu_\kappa$ (same parameter as $\cN_{n,k}$) stopped at $T_n:=k_0 \eps + \tau_{k_0}$. 

\medbreak
\noindent \textbf{Step 2. Convergence of $\cN_n$ to $\cN$.} We aim to prove that $\cN_n$ converges to $\cN$, in the following sense. 
\begin{enumerate}[(1)]
\item\label{eq:stop_time} First, we will show that as $n\to \infty$ (recall $\eps=1/n$),
$
T_n \overset{a.s.}\longrightarrow \chi(s_0).
$
\item\label{eq:stop_time_bb} For $\iota>0$, let $A_\iota$ be the set of bubbles $e$ in $\bbH$ rooted at $0$ such that $\sup_{z\in e} |z|>\iota$. Suppose that in $\cN_n$, there are exactly $N_n$ bubbles in $A_\iota$, which occur at times $u_{n,1}< \cdots < u_{n, N_n} \le T_n$, denoted by $e^n_{u_{n,j}}$ for $1\le j \le N_n$. Suppose that in $\cN$, there are exactly $N$ bubbles in $A_\iota$, which occur at times $u_{1}< \ldots < u_{N}\le \chi(s_0)$. We will show that for all $\delta>0$, there exists $n_0>0$, such that for all $n\ge n_0$, we have $N_n=N$ and $|u_{n,j} - u_j|\le \delta$, $d_\mathrm{H}(e^n_{u_{n,j}}, e_{u_j})\le \delta$ for all $1\le j\le N$.
\end{enumerate}

Let us first show \eqref{eq:stop_time}. By definition, we have 
$
|s_0- s_{n, k_0}|\le  f_{s_{n,k_0}}'(0)^{(\kappa-8)/\kappa} \eps.
$
By Lemma~\ref{lem:total_variation}, we have $\sup_{s\in[0, s_0]} f_s'(0)^{(\kappa-8)/\kappa}< V<\infty$. So $|s_0- s_{n, k_0}|\to 0$. Since $\chi$ is continuous, we also have $|\chi(s_0)- \chi(s_{n, k_0})|\to 0$.
On the other hand, by \eqref{eq:chi}, for each $1\le k \le k_0-1$, 
\begin{align}\label{eq:sk}
|\chi(s_{n,k}) - k \eps|=\sum_{j=0}^{k-1} \bigg| \int_{s_{n,j}}^{s_{n, j+1}} f_{u}'(0)^{(8-\kappa)/\kappa} du  -\eps \bigg|
\le  \sum_{j=0}^{k-1} \theta(s_{n,j}, \eps) \eps \le V\eps \to 0,
\end{align}
where the last inequality follows from Lemma~\ref{lem:total_variation}. Combining with $|\chi(s_0)- \chi(s_{n, k_0})|\to 0$, we get $|\chi(s_0) -k_0\eps|\to 0$. Note that $\tau_{k_0}\le \eps$ by definition, hence \eqref{eq:stop_time} follows.

Let us now show \eqref{eq:stop_time_bb}. If we define $N(u)$ to be the number of bubbles in $A_\iota$ in $\cN$ up to time $u$, then $N(u)$ is \emph{c\`adl\`ag} and has jumps of size $1$. In particular, $N=N(\chi(s_0))<\infty$ a.s.
For each $1\le j\le N$, let $\wt u_j:=\chi^{-1}(u_j)$. We have $f_{\wt u_j}(\wt e_{\wt u_j}) =e_{u_j} \in A_\iota$.
Let $z_j \in e_{u_j}$ be the point that achieves  $\sup_{z\in e_{u_j}} |z|$. Let
 \begin{align*}
\delta_0: =\min_{1\le j\le N} \inf\{d_\lambda (z_j, z):  z \in \partial \cB(0, \iota) \cap \bbH  \}. 
\end{align*}
Note that $\delta_0>0$, because $A_\iota$ is an open set w.r.t.\ $d_\mathrm{H}$, and $N<\infty$. 
Since $\wt u_j\in [0, s_0]$ and $s_{n, k_0}\to s_0$ as $n\to\infty$, for $n$ big enough, say $n\ge n_0$, there is a unique integer $k(j) \in [0, k_0-1]$ such that $s_{n,k(j)}\le \wt u_j < s_{n,k(j)+1}$. By Lemma~\ref{lem:cadlag}, for all $\delta \in (0, \delta_0/2)$, there exists $\eps_0$, such that if $\eps \le \eps_0$, then
$$d_\mathrm{H} \big(f_{s_{n,k(j)}}(\wt e_{\wt u_j}), f_{\wt u_j}(\wt e_{\wt u_j})\big)\le \delta \, \text{ for all } 1\le j \le N.$$
It follows that $f_{s_{n,k(j)}}(\wt e_{\wt u_j}) \in A_\iota$ for $1\le j\le N$, hence these bubbles contribute to $N_n$. We increase $n_0$ so that $\eps_{n_0}\le \eps_0$, then $N_n\ge N$ for $n\ge n_0$. By Step 1, for each $1\le j\le N$, $\wt u_j$ corresponds to the time
\begin{align*}
\wh u_j= k(j)  \eps + f_{s_{n,k(j)}}'(0)^{(8-\kappa)/\kappa} (\wt u_j -s_{n, k(j)})
\end{align*}
in the process $\cN_n$. Since $u_j=\chi(\wt u_j)$, similarly to \eqref{eq:sk}, we have
\begin{align*}
|u_j - \wh u_j|\le  |\chi(s_{n, k(j)}) - k(j)  \eps | + \left|\int_{s_{n, k(j)}}^{\wt u_j}  f_{u}'(0)^{(8-\kappa)/\kappa} du -f_{s_{n,k}}'(0)^{(8-\kappa)/\kappa} (\wt u_j -s_{n, k(j)})  \right| \le V \eps.
\end{align*}
By increasing $n_0$, we can let $|u_j - \wh u_j|\le \delta$. 

To complete the proof of \eqref{eq:stop_time_bb}, it remains to show that $N_n\le N$ for $n$ big enough.
For the sake of contradiction, suppose that for all $n\ge n_0$, there is a time $\wh u^n$ in $[0, T_n]\setminus \{\wh u_1, \ldots, \wh u_N\}$ such that the bubble in $\cN_n$ at $\wh u^n$ is in $A_\iota$. Let $k_n$ be the unique integer in $[0, k_0]$ such that $k_n  \eps \le \wh u^n < (k_n +1) \eps$. Let
\begin{align}\label{eq:relationu}
\wt u^n:=s_{n, k_n} + f_{s_{n,k_n}}'(0)^{(\kappa-8)/\kappa} (\wh u^n - k_n \eps).
\end{align}
Then $ f_{s_{n,k_n}} (\wt e_{\wt u^n}) \in A_\iota$, and $u^n:=\chi (\wt u^n) \in [0, \chi(s_0)]\setminus \{u_1, \ldots, u_N\}$. 
This means that $e_{u^n}\not \in A_\iota$, namely $\sup_{z\in e_{u^n}}|z|\le \iota$. 
We have
\begin{align*}
e_{u^n}= f_{\wt u^n} (\wt e_{\wt e_n})= f_{\wt u^n} \circ f_{s_{n, k_n}}^{-1}\left(f_{s_{n,k_n}} (\wt e_{\wt u^n})\right).
\end{align*}
Since $ f_{s_{n,k_n}} (\wt e_{\wt u^n}) \in A_\iota$,  by Lemma~\ref{eq:festimate}, we have $e_{u^n} \in A_{C \iota}$. Suppose that $n(j)$ is a subsequence such that $\wt u^{n(j)}\to \wt u$ as $j\to \infty$, then $u^{n(j)}\to u=\chi(\wt u)$. 
The subsequence $(e_{u^{n(j)}})_{j \ge 0}$ must be stablized after some $j_0$, equal to $e_{u}$ for some $u\in I$, because otherwise, there would be infinitely many bubbles in $\cN$ in $A_{C\iota}$, leading to a contradiction. Note that $e_u \in A_{C\iota} \setminus A_\iota$. By \eqref{eq:relationu}, we also have $s_{n(j), k_{n(j)}} \to \wt u$. By Lemma~\ref{lem:cadlag}, for all $\delta_1>0$, there exists $\eps_1>0$, such that if $\eps\le \eps_1$, then for all $j\ge j_0$ (where we have possibly increased $j_0$ so that $n(j_0)\ge 1/\eps_1$),
\begin{align*}
d_{\mathrm H} \big(f_{s_{n(j),k_{n(j)}}}(\wt e_{\wt u}), e_{u} \big)=d_{\mathrm H} \big(f_{s_{n(j),k_{n(j)}}}(\wt e_{\wt u}), f_{\wt u} (\wt e_{\wt u}) \big)\le \delta_1.
\end{align*}
Since $ f_{s_{n(j),k_{n(j)}}}(\wt e_{\wt u})  \in A_\iota$, and $e_{u}\not \in A_\iota$, the only possibility is that $\sup_{z\in e_u} |z| =\iota$. However, this is impossible due to Lemma~\ref{lem:no_atom}.
This proves that $N_n\le N$ for $n\ge n_0$, and completes the proof of \eqref{eq:stop_time_bb}.

\medbreak
\noindent \textbf{Step 3. Conclusion.} In Step 1, we have constructed $\cN_n$, which is a Poisson point process with parameter $\mu_\kappa$, stopped at $T_n$. The result of Step 2 implies that $\cN_n$ converges a.s.\ to $\cN$ as point processes. By definition, $\cN$ is stopped at $\chi(s_0)$, and $\chi(s_0)$ is a stopping time for $\cN$. To see the Poisson point process nature of $\cN$, one can first construct the unstopped versions of $\cN_n$ and $\cN$, by concatenating to each of them (after their respective stopping time) a same independent Poisson point process with parameter $\mu_\kappa$. Then the unstopped versions of $\cN_n$ (which are Poisson point processes) converge a.s.\ to the unstopped version of $\cN$, which is consequently also Poisson point processes with parameter $\mu_\kappa$.
This completes the proof of Proposition~\ref{prop:ap_ppp}.
\end{proof}

Now, we are ready to conclude the proof of Description \ref{descrip}.
\begin{proof}[Proof of Description \ref{descrip}]
Proposition~\ref{prop:ap_ppp} shows that $\{e_s, s\in I, s\le \chi(s_0)\}$ is a Poisson point process stopped at the first time that a loop in $\{\ell_s, s\in I\}$ disconnects $i$ from $\infty$ in $\bbH$. If $\ep{e_{\chi(s_0)}}$ encircles $i$, then the stopping time $\tau$ in Description \ref{descrip} is equal to $\chi(s_0)$. Otherwise, let $D_i$ be the connected component containing $i$ of $\bbH\setminus e_{\chi(s_0)}$. Let $f$ be the conformal map from $D_i$ onto $\bbH$ with $f(i)=i$ and $f'(i)>0$.
Conditionally on $\{e_s, s\in I, s\le \chi(s_0)\}$, the image under $f$ of the CLE$_\kappa$ restricted to $\ol D_i$ is equal to an independent CLE$_\kappa$ in $\bbH$. We can then iterate the same exploration process for this new CLE$_\kappa$. The concatenation of these iterated processes construct the Poisson point process in Description \ref{descrip}.
\end{proof}

\section{Proof of Proposition~\ref{prop:weld6}} \label{sec:app}
We first recall the construction of the  $\SLE_{\gamma^2}(\gamma^2-4)$ bubble measure via a limiting procedure due to~\cite{Z22}. The $\SLE_{\gamma^2}(\gamma^2-4)$ bubble measure is defined by $\mathsf{m}_{\gamma^2} := c(\gamma) \lim_{\epsilon \rightarrow 0} \epsilon^{\frac{\gamma^2}{2}-1}\mathsf{m}^\epsilon_{\gamma^2}$, where $\mathsf{m}^\epsilon_{\gamma^2}$ is the law of a chordal ${\rm SLE}_{\gamma^2}(\gamma^2-4)$ from $0$ to $-\epsilon$ in $\mathbb{H}$ with the force point being $0^+$. We choose the constant $c(\gamma)$ such that $\mathsf{m}_{\gamma^2}[E_i] = 1$, where $E_i$ is the event that $i$ is enclosed by the loop sampled from $\mathsf{m}_{\gamma^2}$. Let $\widehat{\mathsf{m}}_{\gamma^2}$ be the restriction of $\mathsf{m}_{\gamma^2}$ on the event $E_i$.
The conformal welding  result for the $\SLE_{\gamma^2}(\gamma^2-4)$ bubble measure can be stated as follows. 
\begin{proposition}\label{prop:Wu}
Proposition~\ref{prop:weld6} holds with $\epmumu$ replaced by $\widehat{\mathsf{m}}_{\gamma^2}$. 
\end{proposition}
Proposition~\ref{prop:Wu} is a special case of 
 Wu's result~\cite[Theorem 1.1]{Wu23} on the conformal welding for the $\SLE_{\gamma^2}(\rho)$ bubble measure. The argument for this case works for all $\SLE_{\gamma^2}(\rho)$ with $\gamma\in (0,2)$ and $\rho\in(-2,\gamma^2/2-2)$, which we sketch it in Section~\ref{subsec:sketch} for completeness. In Section~\ref{subsec:collapse}, we prove that $\epmumu=\widehat{\mathsf{m}}_{\gamma^2}$ which will conclude the proof of Proposition~\ref{prop:weld6}. 
\subsection{Sketch of the proof of Proposition~\ref{prop:Wu}}\label{subsec:sketch}
The proof consists of two steps. 

{\bf Step 1.} This step is a straightforward adaptation of an argument from~\cite[Section 4]{ARS21} with $\mathcal M^{\rm disk}_2(\gamma^2/2)$ there  replaced by $\mathcal M^{\rm disk}_2(\gamma^2-2)$.  We start with the conformal welding for chordal $\SLE_{\gamma^2}(\gamma^2-4)$ from~\cite{AHS20}. For $W=\gamma^2-2$ and some $C\in (0,\infty)$ we have
\begin{equation}\label{eq:chordal}
\mathcal M^{\rm disk}_2(W + 2) \times {\rm chordal} \mbox{ } {\rm SLE}_{\gamma^2}(W-2) = C {\rm Weld}(\mathcal M^{\rm disk}_2(W), \mathcal M^{\rm disk}_2(2)).    
\end{equation}
We first add a marked point to the sample from $\mathcal M^{\rm disk}_2(2)$ and embed the resulting welded surface to $\bbH$ with the starting point of the interface at 0 and the bulk marked point at $i$. Then the endpoint $y$ of the interface becomes a random point lying on $\mathbb R$. Conditioning on $y$, the law of the interface is the chordal $\SLE_{\gamma^2}(\gamma^2-4)$
from $0$ to $y$ conditioned on the event that $i$ is on its left.  Now fix $\eps$ small, we restrict to the event that the quantum length of the interval $(y,0)$ is between $\eps$ and $2\eps$. At the quantum surface level, this conditioning will not affect the sample from $\mathcal M^{\rm disk}_2(\gamma^2-2)$, but change the other surface to a sample from $\rm QD_{1,1}$. Hence the resulting picture is $\rm Weld(\mathcal M^{\rm disk}_2(\gamma^2-2), \rm QD_{1,1})$. On the other hand, the point $y$ converges to $0$ hence the law of the interface will converge to $\widehat{\mathsf{m}}_{\gamma^2}$. Moreover, this convergence holds even after we freeze the randomness of the field. This way, we get that the conformal welding of a sample from $\rm Weld(\mathcal M^{\rm disk}_2(\gamma^2-2), \rm QD_{1,1})$ can be described by $(\bbH, i,0, \phi, \eta)$, where $(\phi,\eta)$ is sampled from the product measure $\mathbb F\times \widehat{\mathsf{m}}_{\gamma^2}$, for some law $\mathbb F$ on fields.

{\bf Step 2.} To conclude the proof, we must show that $\mathbb F= C_3 {\rm LF}_\mathbb{H}^{(\gamma,i),(\frac{4}{\gamma}-\gamma,0)} $ for some constant $C_3\in (0,\infty)$. We first consider the conformal welding  $\rm Weld(\mathcal M^{\rm disk,\bullet}_2(\gamma^2-2), \rm QD_{0,1})$.
Let $(S_1,S_2,S_3)$ be a sample from $\mathcal M^{\rm disk}_2(\gamma^2-2) \times \mathcal M^{\rm disk}_2(2)\times \mathcal M^{\rm disk}_2(\gamma^2-2)$. Then by de-weighting the quantum area of the middle surface in Lemma~\ref{lem:thin-thick}, the concatenation of  $(S_1,S_2,S_3)$ has the law of $\mathcal M^{\rm disk}_2(\gamma^2-2)$ with a marked bead from the counting measure on all of its beads. Let us denote its law by $\mathcal M^{\rm disk,\bullet}_2(\gamma^2-2)$. As in the proof of Lemma~\ref{lem:weld-decompose},  $\rm Weld(\mathcal M^{\rm disk,\bullet}_2(\gamma^2-2), \rm QD_{0,1})$ equals $\rm Weld(\mathcal M^{\rm disk}_2(\gamma^2-2), \mathcal M^{\rm disk}_2(2),\mathcal M^{\rm disk}_2(\gamma^2-2),\widetilde{\rm QD}_{0,3})$. By repeatedly applying the conformal welding result for quantum triangles from~\cite{ASY22} (see the proof of Lemma~\ref{lem:weld4}), we obtain a quantum triangle with three weights given by $2\gamma^2-2, \gamma^2+2,\gamma^2+2$. The three corresponding log singularities are $\beta_0=\frac{4}{\gamma}-\gamma, 0,0$, respectively. 
Once we embed the welded surface to $(\bbH,-1,0,1)$ with $0$ being the $\beta_0$-insertion, we get a sample from $\LF^{(\beta_0,0)}_{\bbH}\times m$, where $m$ is a measure on simple curves that pass $-1,0,1$.
Now we consider  the conformal welding  $\rm Weld(\mathcal M^{\rm disk,\bullet}_2(\gamma^2-2), \rm QD_{1,1})$. Once it is embedded to $(\bbH,-1,0,1)$, the law of the field $\phi$, the curve $\eta$ and the additional bulk marked point $z$ is given by $ 1_{E_{z,\eta}} e^{\gamma \phi(z)}dz \cdot \LF^{(\beta_0,0)}_{\bbH}(d \phi) \times m(d\eta)$, where $E_{z,\eta}$ is the event that $z$ is in the region bounded by $\eta$ that corresponds to $\rm QD_{1,1}$.  By the coordinate change formula of Liouville field (Lemma~\ref{lem:coordinate-change}), if we embed a sample of  $\rm Weld(\mathcal M^{\rm disk,\bullet}_2(\gamma^2-2), {\rm QD}_{1,1})$ to $(\bbH, i,0)$ where $i$ is the bulk insertion and $0$ has the $\beta_0$-log singularity, then we get $\LF^{(\gamma, i),(\beta_0, 0)}\times m^{\bullet,\bullet}$, where $m^{\bullet,\bullet}$ is a measure on simple curves on $\bbH$ with two additional marked points on $\partial \bbH$.

Comparing the conclusions from Step 1 and Step 2, we see that modulo a multiplicative constant the measure $m^{\bullet,\bullet}$ can be obtained  from $\widehat{\mathsf{m}}_{\gamma^2}$ as follows. First sample $\eta$ from  $\widehat{\mathsf{m}}_{\gamma^2}$, then choose a connected component $B$ of $\bbH \setminus \eta$ according to the counting measure, and finally choose  the two endpoints of  $\partial B\cap \bbH$ as the two additional marked points. On the other hand, the field measure $\mathbb F$  must equal $C_3\LF_{\mathbb{H}}^{(\gamma, i), (\beta_0, 0)}$ for some $C_3>0$. The finiteness of $C_3$ can be seen from the fact that the event that the total boundary length lies in $[1,2]$ is finite on both sides. This concludes the proof of Proposition~\ref{prop:Wu}.

\begin{remark}\label{rmk:CR-formula} 
Recall the setting of Lemma~\ref{lem:er2}, which asserts $\bbE [|\psi_{e_\tau}'(i)|^\zeta]<\infty$ for $\zeta<1-2/\kappa$. 
We now recall the exact formula for $\bbE [|\psi_{e_\tau}'(i)|^\zeta]$ from \cite[Proposition 7.12]{Wu23} and \cite[Subsection 4.3]{ARSZ23} which is based on the conformal welding result in Proposition~\ref{prop:Wu}. Let $\beta = \frac{4}{\gamma} - \gamma \in (0,\gamma)$. In our notation,  the second equation above Lemma 4.3 in \cite{ARSZ23} states that for any $Q - \frac{\beta}{2} < \alpha < Q$,
\begin{equation*}
    \bbE [|\psi_{e_\tau}'(i)|^{2 - 2\Delta_\alpha}] =   \frac{ c_{\gamma, \beta} \Gamma(\frac{\gamma \alpha}{2} - \frac{\gamma^2}{4}) \Gamma_{\frac{\gamma}{2}}(\alpha ) }{ \Gamma(\frac{\gamma\alpha}{2}+\frac{\gamma\beta}{4}-\frac{\gamma^2}{4}) \Gamma_{\frac{\gamma}{2}}(\alpha \pm \frac{\beta}{2} ) }  \frac{ (\alpha - Q)  \Gamma_{\frac{\gamma}{2}}(Q -\alpha ) \Gamma_{\frac{\gamma}{2}}(\alpha + \frac{\beta}{2} - Q) }{  \Gamma( \frac{2}{\gamma}(Q -\alpha) +1 ) \Gamma( \frac{2}{\gamma}( \frac{\beta}{2} +\alpha - Q ) ) \Gamma_{\frac{\gamma}{2}}(2Q - \alpha - \frac{\beta}{2}) } .
\end{equation*}
Here $\Gamma_{\frac{\gamma}{2}}(x)$ is the double Gamma function and $\Gamma_{\frac{\gamma}{2}}(\alpha \pm \frac{\beta}{2} )$ means $\Gamma_{\frac{\gamma}{2}}(\alpha + \frac{\beta}{2} ) \Gamma_{\frac{\gamma}{2}}(\alpha - \frac{\beta}{2} )$. It is straightforward to verify that the right-hand side is analytic for $\alpha \in (\frac{\gamma}{2}, Q)$ (when $\alpha = Q - \frac{\beta}{2}$, both $\Gamma( \frac{2}{\gamma}( \frac{\beta}{2} +\alpha - Q ) )$ and $\Gamma_{\frac{\gamma}{2}}(\alpha + \frac{\beta}{2} - Q)$ have a single pole so the right-hand side is still analytic). Therefore, the analytic continuation argument in Lemma~\ref{lem:analtic0} yields that $\bbE [|\psi_{e_\tau}'(i)|^\zeta] < \infty$ for all $\zeta < 2 - 2 \Delta_{\frac{\gamma}{2}} = 1 - \frac{2}{\kappa}$. This proves Lemma~\ref{lem:er2}.

\end{remark}

\subsection{Duality for the ${\rm SLE}_\kappa$ bubble measure}
\label{subsec:collapse}
We now prove $\epmumu=\widehat{\mathsf{m}}_{\gamma^2}$ as stated in Lemma~\ref{lem:A.4}, which is an instance of the ${\rm SLE}$ duality. 
The proof follows from the limiting construction of ${\rm SLE}$ bubble measures and the classical duality for chordal SLE. Recall from \eqref{eq:sle_bubble} that $\mu_\kappa = c(\kappa) \lim_{\epsilon \rightarrow 0} \epsilon^{-1+8/\kappa} \mu_\kappa^\epsilon$ where $\mu_\kappa^\epsilon$ denotes the probability law of a chordal ${\rm SLE}_\kappa$ curve from $0$ to $-\epsilon$. The constant $c(\kappa)$ is chosen such that $\mu_\kappa[E_i] = 1$ where $E_i$ denotes the event that $i$ is enclosed by the outer boundary of the loop. The first step is to show that their outer boundaries also converge in law, as shown in Lemma~\ref{lem:appendix-outer-boundary}. 
\begin{lemma}
\label{lem:appendix-outer-boundary}
    For $\epsilon>0$, let $\widehat{\mu}_\kappa^\epsilon$ be the conditioning law of $\mu_\kappa^\epsilon$ on the event that the outer boundary of a loop encloses $i$. Then, $\lim_{\epsilon \rightarrow 0} \overline{\widehat{\mu}_\kappa^\epsilon} = \epmumu$. 
\end{lemma}
\begin{proof}
    Fix small $r>0$.  Let $\mu_{\kappa,r}^{\epsilon}$ (resp. $\mu_{\kappa,r}$)  be the probability measure obtained by conditioning $\mu_\kappa^{\epsilon}$ (resp. $\mu_\kappa$) on the event that the curve intersects $\partial B_r(0)$. 
    By \eqref{eq:sle_bubble}, we have $\lim_{\epsilon \rightarrow 0} \mu_{\kappa,r}^{\epsilon} = \mu_{\kappa,r}$. Let $\eta$ 
    and $\eta^\epsilon$ be curves sampled from $\mu_{\kappa,r}^{\epsilon}$ and  $\mu_{\kappa,r}$, respectively. For concreteness we parameterize $\eta$ and $\eta^\epsilon$   by $[0,1]$. By Skorohod's representation theorem, we can find a coupling such that $\eta^\epsilon$ converges to $\eta$ almost surely as $\epsilon$ tends to zero in the uniform topology.
    
    Let $\tau^\epsilon$ (resp. $\tau$) be the first time that $\eta^\epsilon$ (resp. $\eta$) hits $\partial B_r(0)$. 
Let $g^\eps$ (resp. $g$) be the conformal map from the infinite connected component of $\bbH \setminus \eta^\epsilon[0,\tau^\epsilon]$ (resp. $\bbH \setminus \eta[0,\tau]$) to $\mathbb{H}$ such that $(\eta^\epsilon(\tau^\epsilon), \infty, -\epsilon )$ (resp. $(\eta(\tau), \infty, 0)$) is mapped to $(0,\infty,-1)$. Then $g^\eps$ converges to $g$ uniformly on any compact set away from $B_r(0)$. 
Let $\eta_1^\eps=g^\eps(\eta^\eps[\tau^\eps,1])$ and $\eta_1=g(\eta[\tau,1])$. By domain Markov property, given $\eta^\epsilon[0,\tau^\epsilon]$ and $\eta[0,\tau]$, both $\eta^\eps_1$ and $\eta_1$  has the same law as a chordal ${\rm SLE}_\kappa$ curve on $\bbH$ from $0$ to $-1$. Therefore, we can re-sample $\eta^\eps[\tau^\eps,1]$ and $\eta[\tau,1]$ so that in our coupling $\eta_1^\eps=\eta_1$. 

Let $\ep\eta$ (resp. $\ep\eta^\eps$) be the outer boundary of $\eta$ (resp. $\eta^\eps$). Under our coupling, we must have $\ep\eta^\eps$ converge to 
 $\ep\eta$ within any compact set away from $B_r(0)$. This remains true if we further condition on the event that $i$ is enclosed by $\ep\eta$ or $\ep\eta^\eps$.
 Now sending $r$ to 0 we have $\lim_{\epsilon \rightarrow 0} \overline{\widehat{\mu}_\kappa^\epsilon} = \epmumu$. \qedhere
\end{proof}
\begin{lemma}
\label{lem:A.4}
    We have $\epmumu=\widehat{\mathsf{m}}_{\gamma^2}$.
\end{lemma}

\begin{proof}
By ${\rm SLE}$ duality (see \cite{Z08}), we see that the outer boundary of $\mu_\kappa^\epsilon$ has the same law as a chordal ${\rm SLE}_{\gamma^2}(\frac{1}{2}\gamma^2-2, \gamma^2-4)$ from $0$ to $-\epsilon$ with the force points being $0^-$ and $0^+$. We will denote its law by $\mathsf{n}_{\gamma^2}^\epsilon$. 
Let $\widehat{\mathsf{n}}_{\gamma^2}^\epsilon$ be the conditioning law of $\mathsf{n}_{\gamma^2}^\epsilon$ to the event that $i$ is enclosed by the curve. Then $\overline{\mu_\kappa^\epsilon} =  \mathsf{n}_{\gamma^2}^\epsilon$ and thus $\overline{\widehat{\mu}_\kappa^\epsilon} = \widehat{\mathsf{n}}_{\gamma^2}^\epsilon$. Combined with Lemma~\ref{lem:appendix-outer-boundary}, we see that $\lim_{\epsilon \rightarrow 0} \widehat{\mathsf{n}}_{\gamma^2}^\epsilon = \epmumu$. 

Recall that $\mathsf{m}^\epsilon_{\gamma^2}$ is the law of a chordal ${\rm SLE}_{\gamma^2}(\gamma^2-4)$ from $0$ to $-\epsilon$ with the force point being $0^+$.    Although $\mathsf{n}_{\gamma^2}^\epsilon \neq \mathsf{m}_{\gamma^2}^\epsilon$ due to the additional force point at $0^-$, we can  coupling them together and prove  $\epmumu=\widehat{\mathsf{m}}_{\gamma^2}$ by sending $\eps\to 0$.  Fix small $\epsilon>0$. 
Let $(\eta^\eps,\eta^\eps_0)$ be a pair of $\SLE$ curves where the marginal law of $\eta^\eps$ is chordal 
${\rm SLE}_{\gamma^2}(\gamma^2-4)$ on $\bbH$ from $0$ to $-\epsilon$ with the force point being $0^+$, and conditioning on $\eta^\eps$, the law of $\eta^\eps_0$ is a chordal ${\rm SLE}_{\gamma^2}( -\frac{1}{2}\gamma^2, \frac{1}{2}\gamma^2-2)$ from $0$ to $-\epsilon$ in the domain bounded by $\eta^\eps$ and $[-\eps, 0]$. The existence of such a pair is ensured by imaginary geometry~\cite{MR3477777}. Now we condition on the event that $i$ is enclosed by $\eta^\epsilon$, and we denote this conditioned probability measure by $\mathbb P^\eps$. The marginal law of $\eta^\eps$ under $\mathbb P^\eps$ is $\widehat{\mathsf{m}}_{\gamma^2}^\epsilon$. Moreover, under  $\mathbb P^\eps$ the diameter of $\eta^\eps_0$ converges to 0 in probability. 

Let $G^\eps$ be the event that $i$ is not enclosed by $\eta^\eps_0$. Then $\mathbb P^\eps[G^\eps]=1-o_\eps(1)$. 
On this event, let $f^\eps$ be the conformal map from the unbounded connected component of $\bbH\setminus \eta^\eps_0$ to $\bbH$ such that $f^\eps(0)=0$ and $f^\eps(i)=i$. Let $x_\eps=f^\eps(-\eps)<0$ and $\eta^\eps_2=f^\eps(\eta^\eps)$. Then the conditional law of $\eta^\eps_2$ is $\widehat{\mathsf{n}}_{\gamma^2}^{x_\epsilon}$. Sending $\eps\to 0$, this conditional law of $\eta^\eps_2$ weakly converges to $\epmumu$. On the other hand, since $\mathbb P^\eps[G^\eps]=1-o_\eps(1)$ and the diameter of $\eta^\eps_0$ converges to 0, the law of  $\eta^\eps$ under $\mathbb P^\eps$ also converges to the same limit $\epmumu$.  Therefore $\lim_{\eps\to 0}\widehat{\mathsf{m}}_{\gamma^2}^\epsilon= \epmumu$. Since $\lim_{\eps\to 0}\widehat{\mathsf{m}}_{\gamma^2}^\epsilon=\widehat{\mathsf{m}}_{\gamma^2}$, we conclude.
\end{proof}

\begin{proof}[Proof of Proposition~\ref{prop:weld6}]
    Combine Proposition~\ref{prop:Wu} and Lemma~\ref{lem:A.4}.
\end{proof}

\bibliographystyle{alpha}
\bibliography{theta}

\end{document}